\documentclass{amsart}

\usepackage{amsmath,comment}
\usepackage{amscd}
\usepackage{amssymb}
\usepackage{amsthm}
\usepackage{cancel}
\usepackage{caption}
\usepackage{color}
\usepackage{float}
\usepackage{graphicx}
\usepackage[linktoc=all]{hyperref}
\usepackage{mathrsfs}
\usepackage{thmtools}
\usepackage{tikz-cd}
\usepackage{mathtools}
\usepackage[all]{xy}
\usepackage{microtype}
\usepackage{soul, xcolor}
\usepackage{url}

\usetikzlibrary{calc}
\usepackage{tikz-cd}
\usepackage[mathscr]{euscript}

\DeclareMathOperator{\depth}{depth}

\begin{document}
\newtheorem{theorem}{Theorem}[section]
\newtheorem{prop}[theorem]{Proposition}
\newtheorem{lemma}[theorem]{Lemma}
\newtheorem{cor}[theorem]{Corollary}

\newtheorem{conj}[theorem]{Conjecture}
\newtheorem{claim}[theorem]{Claim}
\newtheorem{qn}[theorem]{Question}

\newtheorem{defn}[theorem]{Definition}

\newtheorem{defth}[theorem]{Definition-Theorem}
\newtheorem{hyp}[theorem]{Hypothesis}
\newtheorem{obs}[theorem]{Observation}
\newtheorem{rmk}[theorem]{Remark}
\newtheorem{ans}[theorem]{Answers}
\newtheorem{slogan}[theorem]{Slogan}
\newtheorem{assume}[theorem]{Assumption}
\newtheorem{conv}[theorem]{Convention}

\newtheorem*{theorem*}{Theorem}
\newtheorem*{rmk*}{Remark}

\newtheorem{eg}[theorem]{Example}
\newtheorem*{eg*}{Example}

%\renewcommand{\baselinestretch}{1.5}

%Newcommands

\newcommand{\ds}{\displaystyle}
\newcommand{\define}{\textbf}
\newcommand{\inclusion}{\hookrightarrow}
\newcommand{\lap}{\Delta}
\newcommand{\mc}{\mathcal}
\newcommand{\mf}{\mathfrak}
\newcommand{\mr}{\mathscr}
\newcommand{\noin}{\noindent}
\newcommand{\one}{\mathbb 1}
\newcommand{\ra}{\Rightarrow}
\newcommand{\vp}{\varphi}
\newcommand{\wap}{\todo{Write a Proof!}}

\newcommand{\ab}[1]{\langle#1\rangle} %angle bracket
\newcommand{\lrb}[1]{\left[#1\right]}
\newcommand{\lrbigp}[1]{\big(#1\big)}
\newcommand{\lrBigp}[1]{\Big(#1\Big)}
\newcommand{\lrp}[1]{\left(#1\right)}
\newcommand{\lrmod}[1]{\left|#1\right|}
\newcommand{\lrnorm}[1]{\left\|#1\right\|}
\newcommand{\lrset}[1]{\left\{#1\right\}}
\newcommand{\norm}[1]{\|#1\|}
\newcommand{\set}[1]{\{#1\}}
\newcommand{\suppressthis}[1]{}
\newcommand{\todo}[1]{\textcolor{red}{#1}}

\newcommand{\reason}[2]{\stackrel{\text{#1}}{#2}}

\newcommand{\twopartdef}[4]
{
	\left\{
		\begin{array}{ll}
			#1 & \mbox{if } #2 \\
			#3 & \mbox{if } #4
		\end{array}
	\right.
}
\newcommand{\De}{\Delta}
\newcommand{\de}{\delta}
\newcommand{\md}{{-\delta}}
\newcommand{\Deo}{\Delta^{opp}}
\newcommand{\R}{\mathbb R}
\newcommand{\N}{\mathbb N}
\newcommand{\Q}{\mathbb Q}
\newcommand{\Z}{\mathbb Z}
\renewcommand{\injlim}{\varinjlim}

\newcommand{\boundary}{\partial}
\newcommand{\hhat}{\widehat}
\newcommand{\C}{{\mathbb C}}
\newcommand{\Ga}{{\Gamma}}
\newcommand{\G}{{\Gamma}}
\newcommand{\s}{{\Sigma}}
\newcommand{\PSL}{{PSL_2 (\mathbb{C})}}
\newcommand{\pslc}{{PSL_2 (\mathbb{C})}}
\newcommand{\pslr}{{PSL_2 (\mathbb{R})}}
\newcommand{\Gr}{{\mathcal G}}
\newcommand{\integers}{{\mathbb Z}}
\newcommand{\natls}{{\mathbb N}}
\newcommand{\ratls}{{\mathbb Q}}
\newcommand{\reals}{{\mathbb R}}
\newcommand{\proj}{{\mathbb P}}
\newcommand{\lhp}{{\mathbb L}}
\newcommand{\tube}{{\mathbb T}}
\newcommand{\cusp}{{\mathbb P}}
\newcommand\AAA{{\mathcal A}}
\newcommand\HHH{{\mathbb H}}
\newcommand\BB{{\mathcal B}}
\newcommand\CC{{\mathscr{C}}}
\newcommand\DD{{\mathcal D}}
\newcommand\EE{{\mathcal E}}
\newcommand\FF{{\mathcal F}}
\newcommand\GG{{\mathcal G}}
\newcommand\HH{{\mathcal H}}
\newcommand\II{{\mathcal I}}
\newcommand\JJ{{\mathcal J}}
\newcommand\KK{{\mathcal K}}
\newcommand\LL{{\mathcal L}}
\newcommand\MM{{\mathcal M}}
\newcommand\NN{{\mathcal N}}
\newcommand\OO{{\mathcal O}}
\newcommand\PP{{\mathcal P}}
\newcommand\QQ{{\mathcal Q}}
\newcommand\RR{{\mathcal R}}
\newcommand\SSS{{\mathcal S}}
\newcommand\TT{{\mathcal T}}
\newcommand\UU{{\mathcal U}}
\newcommand\VV{{\mathcal V}}
\newcommand\WW{{\mathcal W}}
\newcommand\XX{{\mathcal X}}
\newcommand\YY{{\mathcal Y}}
\newcommand\ZZ{{\mathcal Z}}
\newcommand\CH{{\CC\Hyp}}
\newcommand{\Chat}{{\hat {\mathbb C}}}
\newcommand\MF{{\MM\FF}}
\newcommand\PMF{{\PP\kern-2pt\MM\FF}}
\newcommand\ML{{\MM\LL}}
\newcommand\PML{{\PP\kern-2pt\MM\LL}}
\newcommand\GL{{\GG\LL}}
\newcommand\Pol{{\mathcal P}}
\newcommand\half{{\textstyle{\frac12}}}
\newcommand\Half{{\frac12}}
\newcommand\Mod{\operatorname{Mod}}
\newcommand\diag{\operatorname{diag}}
\newcommand\Area{\operatorname{Area}}
\newcommand\Maps{\operatorname{Maps}}
\newcommand\ep{\varepsilon}
\newcommand\Hypat{\widehat}
\newcommand\Proj{{\mathbf P}}
\newcommand\htop{{\mathrm{hTop}}}
\newcommand\hsset{{\mathrm{hSSet}}}
\newcommand\sset{{\mathrm{SSet}}}
\newcommand\str{{\mathrm{Str}}}
\renewcommand\top{{\mathrm{Top}}}
\newcommand\U{{\mathbf U}}
\newcommand\Hyp{{\mathbf H}}
\newcommand\D{{\mathbf D}}
\renewcommand\set{{\mathrm{Set}}}
\newcommand\E{{\mathbb E}}
\newcommand\EXH{{ \EE (X, \HH_X )}}
\newcommand\EYH{{ \EE (Y, \HH_Y )}}
\newcommand\GXH{{ \GG (X, \HH_X )}}
\newcommand\GYH{{ \GG (Y, \HH_Y )}}
\newcommand\ATF{{ \AAA \TT \FF }}
\newcommand\PEX{{\PP\EE  (X, \HH , \GG , \LL )}}
\newcommand{\lct}{\Lambda_{CT}}
\newcommand{\lel}{\Lambda_{EL}}
\newcommand{\lgel}{\Lambda_{GEL}}
\newcommand{\lre}{\Lambda_{\mathbb{R}}}
\newcommand{\p}{\raisebox{0.1mm}{p}}
\renewcommand{\ss}{\raisebox{0.1mm}{s}}
\renewcommand\epsilon{\varepsilon}

\newcommand\til{\widetilde}
\newcommand\length{\operatorname{length}}
\newcommand\diff{\operatorname{Diff}}
\newcommand\sdiff{\operatorname{StratDiff}}
\newcommand\sj{\operatorname{Strat_0\,\mathcal{J}}}
\newcommand\sjr{\operatorname{Strat\,\mathcal{J}}}
\newcommand\diffc{\diff_c}
\newcommand\tr{\operatorname{tr}}
\newcommand\gesim{\succ}
\newcommand\lesim{\prec}
\newcommand\simle{\lesim}
\newcommand\simge{\gesim}
\newcommand{\simmult}{\asymp}
\newcommand{\simadd}{\mathrel{\overset{\text{\tiny $+$}}{\sim}}}
\newcommand{\ssm}{\setminus}
\newcommand{\diam}{\operatorname{diam}}
\newcommand{\pair}[1]{\langle #1\rangle}
\newcommand{\T}{{\mathbf T}}
\newcommand{\I}{{\mathbf I}}

\newcommand{\height}{\operatorname{height}}
\newcommand{\base}{\operatorname{base}}
\newcommand{\trans}{\operatorname{trans}}
\newcommand{\rest}{|_}
\newcommand{\bbar}{\overline}
\newcommand{\bH}{\overline{\HH}}
\newcommand{\ho}{{\HH\OO}}
\newcommand{\UML}{\operatorname{\UU\MM\LL}}
\newcommand{\EL}{\mathcal{EL}}
\newcommand{\qle}{\lesssim}
\newcommand{\rank}{\operatorname{rank}}
\newcommand{\hofib}{\operatorname{hofib}}
\newcommand{\res}{\operatorname{res}}
\newcommand{\maps}{\operatorname{Maps}}
\newcommand{\imm}{\operatorname{Imm}}
\newcommand{\op}{\operatorname{Op}}

\newcommand\Gomega{\Omega_\Gamma}
\newcommand\nomega{\omega_\nu}
\newcommand\cgh{compactly generated Hausdorff \,}

%for commutative diagrams
\newcommand{\xra}[1]{\!\!\!\!\xrightarrow{\quad#1\quad}\!\!\!\!}
\newcommand{\xda}[1]{\left\downarrow{\scriptstyle#1}\vphantom{\displaystyle\int_0^1}\right.}
\newcommand{\xla}[1]{\!\!\!\!\xleftarrow{\quad#1\quad}\!\!\!\!}
\newcommand{\xua}[1]{\left\uparrow{\scriptstyle#1}\vphantom{\displaystyle\int_0^1}\right.}

%for color-coded edits
\newcommand{\bcomment}[1]{\textcolor{blue}{#1}}
\newcommand{\mcomment}[1]{\textcolor{magenta}{#1}}

\title{$h$-principle for stratified spaces}

\author{Mahan Mj}
\address{School
	of Mathematics, Tata Institute of Fundamental Research. 1, Homi Bhabha Road, Mumbai-400005, India}

\email{mahan@math.tifr.res.in}

\author{Balarka Sen}
\address{School
	of Mathematics, Tata Institute of Fundamental Research. 1, Homi Bhabha Road, Mumbai-400005, India} \email{balarka2000@gmail.com}

\email{balarka@math.tifr.res.in}

\thanks{Both authors are supported by the Department of Atomic Energy, Government of India, under project no.12-R\&D-TFR-5.01-0500. MM is supported
	 in part by a DST JC Bose Fellowship,  and an endowment from the Infosys Foundation. MM also thanks  CNRS for support during a visit to Institut Henri Poincar\'e in the summer of 2022.
}

\date{\today}

\begin{abstract}   We extend Gromov and Eliashberg-Mishachev's $h-$principle on manifolds to stratified spaces. This is done in both  the sheaf-theoretic framework of Gromov and the smooth jets framework of Eliashberg-Mishachev. The generalization involves developing 
	\begin{enumerate}
	\item 	the notion of stratified continuous sheaves to extend Gromov's theory,
	\item the notion of smooth stratified  bundles to extend Eliashberg-Mishachev's theory.
	\end{enumerate}
A new feature is the role played by homotopy fiber sheaves. We show, in particular,  that   stratumwise flexibility of stratified continuous sheaves
along with flexibility of homotopy fiber sheaves furnishes the parametric $h-$principle. We  extend the Eliashberg-Mishachev holonomic approximation theorem to stratified spaces. We also prove a stratified analog of the Smale-Hirsch immersion theorem.
\end{abstract}

\maketitle

\tableofcontents

\section{Introduction}\label{sec-intro} Gromov developed the $h-$principle \cite{Gromov_PDR} as a soft topological approach to finding solutions to partial differential relations, and this was refined subsequently by several others, notably Eliashberg-Mishachev \cite{em-book}.  These two  references \cite{Gromov_PDR,em-book} emphasize somewhat different points of view. While \cite{em-book} uses the standard terminology of differential topology in terms of jets,
\cite{Gromov_PDR} uses a more abstract, formal sheaf-theoretic framework. The main applications of both these approaches is to solve partial differential relations on \emph{smooth manifolds}. The aim of this paper is to extend the domain of applicability of the $h-$principle to \emph{smooth stratified spaces} (cf.\ \cite{GM_SMT}). 

There is one immediate difficulty that we face. Let $X$ be a smooth stratified space. Then, any natural notion of a tangent bundle $TX$ (cf.\ Definition \ref{str-tbl}) to $X$ gives a structure that is not a bundle in the usual sense of a topological bundle.
This leads us to the notion of \emph{stratified bundles} (Definition \ref{def-strbdl}) which consist of bundles along strata of $X$ with appropriate gluing conditions across strata. It is precisely this {``gluing data"} that distinguishes the manifold framework from the stratified spaces framework.
A prototypical example of a stratified bundle
to keep in mind is that of $P: M\longrightarrow M/G$, where $M$ is a smooth manifold, and $G$ is a compact Lie group 
with a not necessarily free smooth action on $M$. Then $M$ admits a stratification by orbit type, where the strata $M_{(H)}$ are indexed by closed subgroups $H \leq G$, consisting of points with isotropy group conjugate to $H$. This stratification descends  to a stratification of $M/G$, and $P : M \to M/G$ becomes a stratified fiber bundle. Let $x \in M$, $G_x$ be the isotropy group, and $[x]=P(x)$. Then the fiber of $P$ over $[x]$ is $G/G_x$.

We found the sheaf-theoretic formalism of Gromov \cite[Ch. 2]{Gromov_PDR} more convenient  to address the algebraic topology issues. We refer the reader to \cite[Ch. 2]{Gromov_PDR} for details on sheaves of quasitopological spaces, i.e.\ continuous sheaves.
Note however that the sheaf of sections of a stratified fiber bundle $E$ over a stratified space $X$ forms something more involved than just a sheaf over a topological space, as one may restrict $E$ to any stratum-closure $\bbar{L} \subset X$ and take sections thereof. We therefore extend 
the sheaf-theoretic formalism of Gromov to \emph{stratified sites} (Definition \ref{def-stratfdsite}) and \emph{stratified sheaves} (Definition \ref{def-sssheaf}) over these. In cases of interest in this paper, a stratified sheaf typically assigns a topological space to an open subset $U$ of a stratum closure $\bbar L$. Here, $L$  ranges over
 strata of $X$. 
 Following Gromov \cite{Gromov_PDR}, we call such a stratified sheaf a \emph{stratified continuous sheaf}. Thus, a 
stratified continuous sheaf is a collection of continuous sheaves $\{\FF_{\bbar{L}}\}$, one for every stratum $L$ of $X$. In this sense, stratified continuous sheaves are analogous to constructible sheaves {(Remark \ref{rmk-stratshcnstrble})}.

Next, for every pair $S<L$, the data of a stratified continuous sheaf gives a restriction map $$\res^L_S: i_{\bbar{S}}^* \FF_{\bbar{L}} \to \FF_{\bbar{S}},$$ where $ i_{\bbar{S}}^* \FF_{\bbar{L}}$ denotes the pullback of $ \FF_{\bbar{L}}$ to ${\bbar{S}}$. A homotopy theoretic construction that arises naturally is that of the  homotopy fiber ($\hofib$): $$\bH^L_S= \hofib (\res^L_S: i_{\bbar{S}}^* \FF_{\bbar{L}} \to \FF_{\bbar{S}}).$$ Set $\HH^L_S= \bH^L_S|S$, the restriction of $\bH^L_S$ to the stratum $S$ of  ${\bbar{S}} \subset X$. Thus, 
{$$\HH^L_S= \hofib (\res^L_S:  i_S^*\FF_{\bbar{L}} \to \FF_{{S}}).$$} 
We shall refer to $\bH^L_S$
and $\HH^L_S$ as the closed and open homotopy fiber sheaves respectively for the pair of strata $S<L$.
One of the aims of this paper is to find conditions on the homotopy fiber sheaves that guarantee the $h-$principle for the stratified sheaf $\FF$.

The sheaf of sections of a bundle comes naturally equipped with the compact-open topology, and therefore constitutes a continuous sheaf. A key  difference between Gromov's $h-$principle \cite[Section 2.2]{Gromov_PDR} and the present paper stems from the fact that a stratified sheaf is a collection of sheaves, and not a single sheaf. The difference already shows up when the base space is a simplex equipped with its natural stratification.
This change in setup  was motivated by a question due to Sullivan. 
In fact, this paper and its companion \cite{ms-gt} were born in part by trying to address the following two questions due to Sullivan \cite{sullivan-pihes-formal} and Gromov \cite{Gromov_PDR}. After discussing smooth forms on simplices in \cite{sullivan-pihes-formal}, Sullivan suggests the following example/problem/test-question.

\begin{qn}[stratified spaces] \cite[p. 298]{sullivan-pihes-formal}\label{qn-sullivan0}
	``In the $\cdots$ cell-space abstraction we didn't require that cells be contractible.
	Thus these notions can be extended to stratified sets -- thought of inductively as obtained
	by attaching manifolds with boundary with a careful statement about the geometry of
	the attaching map.
	
	It would be interesting to carry this out in detail -- the basic idea being that a
	form should have values only on multivectors tangent to the strata."
\end{qn}

We interpret Question \ref{qn-sullivan0} as follows:
\begin{qn}\label{qn-sullivan}
Provide an inductive description of forms over stratified spaces.
\end{qn}

Another source of inspiration for this paper comes from the following question due to Gromov.

\begin{qn}\cite[p. 343]{Gromov_PDR}\label{qn-gromov}
	Can one define singular symplectic (sub) varieties? 
\end{qn}

In this paper, we address Question \ref{qn-sullivan}, by providing an inductive description of sheaves of smooth forms over stratified spaces. In fact, we set up the more general framework of a stratified bundle over a stratified space,
and provide an inductive description of the sheaf of sections of a stratified bundle over a stratified space. We postpone a full treatment of Question \ref{qn-gromov} to the companion paper \cite{ms-gt}, but develop a general conceptual framework in this paper.

Having fixed the framework in terms of stratified sheaves over stratified spaces, the main aim of the paper then boils down to extending some of the basic notions introduced by Gromov in \cite[Ch. 2]{Gromov_PDR} to the stratified context, and proving the stratified $h-$principle using these generalizations. Two crucial concepts were important in \cite[Ch. 2]{Gromov_PDR} for  a sheaf $\FF$ (of topological spaces) over a manifold $V$:

\begin{enumerate}
	\item Flexibility of $\FF$: This means that for every {pair of compact subsets $K\subset K'$ of $V$}, $\FF(K') \to \FF(K)$ is a (Serre) fibration.
	\item $\diff(V)-$invariance of $\FF$: This means that the action of the pseudogroup $\diff(V)$ of partially defined diffeomorphisms of $V$ lifts to an action on $\FF$.
\end{enumerate}

Flexibility of sheaves is generalized to flexibility of stratified  sheaves by demanding
two kinds of conditions:
\begin{enumerate}
	\item \emph{Stratumwise flexibility:} For every stratum $S <X$, the restricted sheaf $\FF|S$ (assigning $\FF(U)$ only to open subsets $U$ of $S$) is flexible as a sheaf over the manifold $S$. Using Gromov's work \cite{Gromov_PDR}, this hypothesis allows us to conclude 
	the $h-$principle for the restrictions of $\FF$ to \emph{open strata}.
	\item \emph{Flexibility across strata}: For every stratum $S < X$, the open homotopy fiber sheaf $\HH^L_S$ is flexible on the open stratum $S$.
	As pointed out before,
	for any pair of strata $S<L$ of $X$, there exist closed and open homotopy fiber sheaves $\bH^L_S$ and $\HH^L_S$ respectively. Of these, the open homotopy fiber sheaf $\HH^L_S$ is a more tractable object and is defined on the open (manifold) stratum $S$. \emph{ $\HH^L_S$ is the key new player introduced in this paper in the context of stratified spaces.} No analog
		exists in the smooth manifold context. It is the sheaf $\HH^L_S$ that encodes {``gluing data"} across the pair of strata $S, L$.
\end{enumerate}

We establish (Theorem \ref{thm-hofibsflexg}) that if $(1)$ and $(2)$ holds, then $\FF$ satisfies the $h-$principle.\\

The condition of $\diff(V)-$invariance in \cite{Gromov_PDR} is generalized to invariance of the stratified sheaf $\FF$ under stratified diffeomorphisms (Definition \ref{def-stratdiff}). Thus, $\FF$ is $\sdiff-$invariant if for every pair  of strata $S<L$ of $X$, $\FF|S$, $i_S^* \FF_{\bbar{L}}$ and the restriction map $\res^L_S$ are naturally
$\diff(S)-$invariant. \\

\noindent {\bf Terminology, examples and non-examples:}  What we have called ``stratified bundles" have their origins in work of Thom \cite{Thom_stratmaps}.  Mather \cite[p. 500]{mather-notes} defines the notion of a ``Thom map". In this paper  a stratified bundle map is a Thom map satisfying an extra  hypothesis  (Definition \ref{def-strbdl}, condition (iii))  making the notion closer to a bundle. 
\begin{enumerate}
\item The main class of examples of stratified bundles over stratified spaces, as mentioned earlier, consists of $P: M \to M/G$, where $M$  is a smooth manifold, and $G$ is a compact Lie group. This includes symplectic reductions \cite{SL_stratsympred} \cite[Theorems 1.4.2, 2.4.2]{marsden-lnm}. In fact, if $G_1, G_2$ admit commuting actions on $M$, e.g.\ if the $G_1$ action is on  the left, and 
the $G_2$ action is on  the right, then there exists a stratified bundle $P_1: G_1\backslash M \to G_1\backslash M /G_2$, where $G_1\backslash M$  itself is allowed to be a stratified space. Thus, there are natural examples of stratified bundles of the form $P: X \to X/G$, where $X$  is itself a stratified space.
\item \emph{Caveat:} The reader should be warned that in the context of complex analytic spaces, a 
genuine topological bundle over a stratified complex analytic space with holomorphic total space, and fiber a complex manifold is sometimes referred to as a stratified bundle  (cf.\ \cite{forstneric}).  We use  stratified bundle  in a much more general sense than this. In particular, the fibers over different strata need not be homeomorphic in our context.
\item Let $D \subset V$ be  a  singular divisor in a smooth complex variety $V$. Let
$N_\ep (D)$ be a regular neighborhood, and  $\partial N_\ep (D)$ its boundary.
Let $r: N_\ep (D) \to D$ be the retraction map to the divisor.
Then the restriction $r| \partial N_\ep (D): \partial N_\ep (D) \to D$  is \emph{not} an example of 
 a stratified bundle in our sense. This is because lower  dimensional strata in $D$ have higher dimensional fibers under $r| \partial N_\ep (D): \partial N_\ep (D) \to D$. For a  stratified bundle in our sense,
 the opposite happens: for instance, lower  dimensional strata in $M/G$ have lower dimensional 
  fibers under $P: M \to M/G$.
\end{enumerate}

\subsection{Statement of results}
We are now in a position to state the first main theorem of our paper
(referring to 
Theorem \ref{thm-hofibsflexg} for a more precise statement).
\begin{theorem}\label{thm-sheafh-intro}
	Let $\FF=\{ \FF_{\bbar{L}}: L < X \text{ stratum}\}$ be a stratified continuous sheaf over a stratified space $X$,
	such that $\FF$  is  stratumwise flexible, i.e.\  $\FF_{\bbar{L}}|L$ for each $L < X$ is
 flexible.
	Further, suppose that $\FF$ is infinitesimally flexible across strata, i.e.\ each open homotopy fiber sheaf $\HH^L_S$  is
	flexible. 
	Then  $\FF$ satisfies the parametric $h-$principle.
\end{theorem}

Gromov  deduces the $h-$principle from the homotopy-theoretic condition of flexibility 
for sheaves, notably over manifolds \cite[p. 76]{Gromov_PDR}.  Theorem \ref{thm-sheafh-intro} is the stratified analog of Gromov's theorem: the underlying space is replaced by a stratified space, and sheaves are replaced by stratified sheaves.

The following theorem of Gromov  connects flexibility and microflexibility  \cite[Ch. 2.2]{Gromov_PDR}.

\begin{theorem}\cite[p. 78]{Gromov_PDR}\label{thm-grmicro2flexintro}
	Let $Y = V \times  \mathbb{R}$ and let $\Pi: Y \to V$ denote the projection
	onto the first factor. Let $\FF$ be a microflexible continuous sheaf over $Y$ invariant under $\diff( V, \Pi)$.
	Then the restriction $\FF \vert V \times \{0\}$ is a flexible sheaf over $V (= V \times \{0\})$.
	
	Let $\FF$ be a microflexible $\diff(V)-$invariant continuous sheaf
	over a manifold $V$. Then the restriction to an arbitrary piecewise smooth polyhedron
	$K \subset V$ of positive codimension, $\FF|K$, is a flexible sheaf over $K$.
\end{theorem}

The stratified analog of Theorem \ref{thm-grmicro2flexintro} is then furnished by the following (see
Theorems \ref{thm-micro2flexs} and \ref{thm-micro2flexs2}):

\begin{theorem}
	Let $\FF=\{ \FF_{\bbar{L}}: L < X \text{ stratum}\}$ be a stratified continuous sheaf over a stratified space $X$,
such that $\FF$ is $\mathrm{StratDiff}$-invariant. Further, suppose for each stratum $L < X$, $\FF_{\bbar{L}}|L$ is microflexible and for each pair of strata $S < L < X$, $\HH^L_S$ is microflexible. Then the restriction $\FF|K$ to a stratified subspace $K \subset X$ of stratumwise positive codimension satisfies the parametric $h-$principle.
\end{theorem}

In Section \ref{sec-formalfnstrat}, we address Sullivan's Question \ref{qn-sullivan} by developing a flag-like structure for jets on stratified spaces. The main results are given by   Propositions \ref{prop-formalfnmflds}, \ref{prop-decompcones},  and  \ref{prop-formalfnnbhdstrat2}. These results give a stratified analog 
 of the sheaf of formal $r-$jets in \cite{em-book,em-expo}  (see  Definition \ref{def-sjr} giving the corresponding sheaf $\sjr^r$). In a sense, Section \ref{sec-formalfnstrat} interpolates between the algebraic topology of Section \ref{sec-hprin} and the differential topology of Section 
\ref{sec-hat}.

In Section \ref{sec-hat}, we return to the differential topological setting of jets and jet bundles as a concrete example to which the above sheaf-theoretic theorems may be applied.
Let $p:E \to X$  be a smooth stratified bundle over a smooth stratified space $X$. Then the sheaves of sections of $p:E \to X$  and their jets come with natural control conditions. Let  $\FF$ denote the stratified sheaf of controlled sections of $E$ over $X$.  Then we have the following
(see Theorem \ref{thm-diffrlnflex}).
\begin{theorem}\label{thm-diffrlnflexintro}
	$\FF$   is flexible, in particular it satisfies the parametric $h-$principle.
\end{theorem}

A caveat is in order. The continuous sheaf $\FF$ is strictly smaller than the sheaf  of \emph{all} sections. It  is rather easy to see that the 
sheaf of  all sections satisfies the parametric $h-$principle. Theorem \ref{thm-diffrlnflexintro} says that this continues to hold in the presence of control conditions. In fact, $\FF$ can be identified with the sheaf of holonomic stratified $r-$jets (see Definition \ref{def-sjr} for the sheaf $\sjr^r$ of formal stratified $r-$jets), and hence  Theorem \ref{thm-diffrlnflexintro} is also true for the sheaf of holonomic stratified jets.

 We establish the following stratified holonomic approximation theorem (see Theorem \ref{em-hats}), generalizing Eliashberg-Mishachev's 
holonomic approximation theorem \cite[Theorem 1.2.1]{em-expo} for manifolds (Theorem \ref{em-hat}).
\begin{theorem}\label{em-hatsintro}
		Let $X$ be {an abstractly stratified space} equipped with a compatible metric, $E \to X$ be a stratified bundle, and $K \subset X$ be a relatively compact
	stratified subspace of positive codimension.
	Let {$f \in \sjr^r_E(\op K)$} be a $C^r-$regular formal section. Then for arbitrarily small $\ep > 0$, 
	 $\delta > 0$, there exist a stratified
	diffeomorphism $h : X \to X$ with $$||h-{\mathrm{Id}}||_{C^0} < \delta,$$ and a stratified holonomic section {$\til{f} \in \sjr^r_E(\op K)$} such
	that 
	\begin{enumerate}
		\item the image $h(K)$ is contained in the domain of  definition of the section $f$, 
		\item $||\til{f}- f|\op\, h(K) ||_{C^0} < \ep$.
		\item $\til{f}, f|\op\, h(K)$ are normally $\varepsilon$ $C^r$-close.
	\end{enumerate}
\end{theorem}
We should point out that neither Theorem \ref{thm-diffrlnflexintro}, nor Theorem \ref{em-hatsintro} follow from the relative holonomic approximation theorem  \cite[Theorem 3.2.1]{em-book}. The basic 
issue can be illustrated in the simple case of a pair of strata $S<L$. Let $E_S, E_L$ denote the bundles over $S, L$. In order to prove 
either of these theorems, we need to consider \emph{extensions} of a  jet 
(formal or holonomic) of $E_S$  over $S$  to a  jet over a germ of a neighborhood of  $S$ in $\bbar L$. This echoes the fact mentioned earlier, that in the stratified sheaf context, the open homotopy fiber $\HH^L_S$ is a new player in the game. Alternately, the extension may be thought of as ``gluing data" that allows us to go between jets over $S$ and jets over $L$. 
 
A host of applications in the manifold context have been enumerated by 	Eliashberg-Mishachev. All have potential generalizations to the stratified context in the light of Theorem \ref{em-hatsintro}.  We give an application of our techniques  at the end of Section \ref{sec-emhat} by showing that stratified immersions of positive codimension between stratified spaces satisfy the $h-$principle (Theorem \ref{thm-immns}): this is the stratified analog of the Smale-Hirsch theorem \cite[Chapter 8.2]{em-book}.

From a more algebraic topology perspective, a very general local-global principle was established in \cite[Theorem 1.2.3]{aft}. In the light of the main result of \cite{nv}, this result is applicable to Whitney stratified spaces. 
{In principle, this reduces the problem of establishing a stratified analog of Smale-Hirsch theorem over a general stratified space, to stratified spaces of the form $\Bbb R^n \times cA$, where $cA$ denotes cone on another stratified space $A$. Since $A$ is in general not a manifold, establishing such a Smale-Hirsch theorem remains as difficult in the local context as in the global context. In particular, Theorem \ref{thm-immns} does not follow from \cite{aft,nv}.}

{We adopt a very different strategy in the proof of Theorem \ref{thm-immns}: instead of reducing the global problem to a local one, we reduce it to an \emph{infinitesimal} one near each stratum, by Theorem \ref{thm-sheafh-intro}. To illustrate in the local context of $\Bbb R^n \times cA$, let $c_A \subset cA$ denote the cone point. We study germs of immersions at $\Bbb R^n \times c_A \subset \Bbb R^n \times cA$, and show that such a thing decouples into two components:
\begin{enumerate}
\item An immersion of $\Bbb R^n \times \{c_A\}$,
\item An $\Bbb R^n$--parametric family of germs of immersions at $c_A \subset cA$.
\end{enumerate}
This description fits very concisely into the framework of stratified continuous sheaves: Item $(1)$ defines the restriction of the stratified continuous sheaf of stratified immersions, to the stratum $\Bbb R^n \times \{c_A\}$. Item $(2)$ effectively defines the open homotopy fiber sheaf over the stratum $\Bbb R^n \times \{c_A\}$. From the classical Smale-Hirsch theorem \cite[Chapter 8.2]{em-book} we deduce stratumwise flexibility. In the proof of Theorem \ref{thm-immns}, we establish flexibility across strata by analyzing Item $(2)$. Theorem \ref{thm-immns} now follows by appealing to Theorem \ref{thm-sheafh-intro}.}

\subsection{Outline of the paper} The aim of Section \ref{sec-prel}  is to set up the context of the $h-$principle for stratified spaces. This section is in the spirit of \cite{em-book}, except that the technology is for stratified spaces in place of manifolds.  We start by describing the setup of stratified spaces  following \cite{GM_SMT,mather-notes}. Stratified spaces are recalled in Section \ref{sec-stratfdsps} and stratified maps in Section \ref{sec-stratfdmaps}. We define stratified bundles and related notions in Section \ref{sec-stratfdbdl}, where the basic fact we prove is the local structure of bundle maps
(Lemma \ref{lem-strbdl-trivialization} and Corollary \ref{strbdl-cone-link}). It is well-known that locally a stratified space looks like a product $\R^n \times cA$
of Euclidean space with a cone $cA$ on a compact stratified space $A$. Lemma \ref{lem-strbdl-trivialization} and Corollary \ref{strbdl-cone-link} upgrade this to a statement about the local structure of a stratified bundle over a stratified space.
We then proceed
in Section  \ref{sec-stratfdjet} to define
stratified jets, jet bundles,  
and formal and holonomic sections in the stratified context.

While Section \ref{sec-prel} ends by setting up the context of the $h-$principle for stratified spaces by describing jet bundles and their sections and is in the spirit of \cite{em-book},   Section  \ref{sec-sheafflexdiff} has more of  an algebraic topology flavor, and is in the spirit of
\cite{Gromov_PDR}. Here, we look at sheaves  over stratified spaces. 
The crucial  notion of a \emph{stratified  sheaf} over a stratified space is introduced in  Section \ref{sec-flexdefs}. 
Flexibility conditions are introduced in this context in Section \ref{sec-flexdefs}. A principal condition used in \cite{Gromov_PDR} is 
$\rm{Diff}-$invariance of sheaves. In Section \ref{sec-diffinv}, we describe the stratified analog, $\sdiff-$invariance, in the context of stratified spaces. 

In Section \ref{sec-hprin}, we prove one of the main  theorems of the paper, Theorem \ref{thm-hofibsflexg} (or Theorem \ref{thm-sheafh-intro} above), establishing the parametric $h-$principle for stratified sheaves over stratified spaces. The main idea or mnemonic may be  summarized as follows: flexibility (a precursor to the $h-$principle) normal to strata plus flexibility  tangential to strata furnishes the $h-$principle for  stratified sheaves over stratified spaces. In a sense, this is  in the spirit of Goresky-MacPherson's fundamental theorem on Morse data \cite{GM_SMT} on stratified spaces where
total Morse data can be recovered from normal Morse data and tangential Morse data.
On the way, we establish Theorem \ref{thm-flex2shprin}, spelling out the connection between flexibility and the $h-$principle in the context of stratified sheaves. A tool we use 
in Section \ref{sec-hprin} is Milnor's theory of microbundles. This allows us to simplify Gromov's formalism from 
\cite[Chapter 2]{Gromov_PDR}.

In Section \ref{sec-micro}, we establish the connection between flexibility  and microflexibility  of stratified sheaves (see Theorem \ref{thm-micro2flexs}). It  follows (see Theorem \ref{thm-micro2flexs2}) that the restriction of microflexible $\sdiff-$invariant sheaves to positive codimension stratified subspaces satisfies the parametric $h-$principle.

Section \ref{sec-formalfnstrat} is devoted to  using microbundles and developing a homotopy model of the Gromov diagonal normal sheaf $\FF^\bullet$
for a sheaf $\FF$ of controlled sections of a stratified bundle $P: E \to X$. In so doing, we answer Sullivan's Question \ref{qn-sullivan} by developing a formalism of flag-like sets.
The description is hybrid in nature. There is a tangential component given by sections along manifold strata $S$ and there is a normal component given by sections over the link of $S$ in $X$. Since the link $A$ of a stratum $S$ is itself a stratified space, the restriction of $P: E \to X$ to $P:P^{-1}(A)\to A$ is again a stratified bundle of lesser complexity;  hence  an  inductive description. 

We return to jets and jet bundles of stratified bundles over stratified spaces in Section \ref{sec-hat}.
Theorem \ref{thm-diffrlnflex}   establishes that the
sheaf of sections of a stratified jet bundle  satisfies the parametric $h-$principle. We also prove the stratified analog of Eliashberg-Mishachev's holonomic approximation theorem in Theorem \ref{em-hats}. As an application of Theorem \ref{thm-sheafh-intro}, we establish a stratified Smale-Hirsch theorem: stratified immersions of positive codimension between stratified spaces satisfy the $h-$principle (Theorem \ref{thm-immns}).\\

\noindent {\bf Acknowledgments:} The authors thank Yasha Eliashberg for comments on an earlier draft. The authors thank the anonymous referee(s) for 
their detailed and helpful comments and suggestions.

\section{Smoothly stratified objects and maps} \label{sec-prel}

\subsection{Smoothly stratified spaces}\label{sec-stratfdsps}

\begin{defn}\label{def-Idec} Let $X$ be a locally compact second countable metric space and let $(I, \leq)$ be a partially ordered set. An \emph{$I$-decomposition} of $X$ is a locally finite collection $\{S_\alpha\}_{\alpha \in I}$ of disjoint locally closed subsets of $X$ such that 
\begin{enumerate}
\item $S_\alpha$ is a topological manifold for all $\alpha \in I$
\item $X = \bigcup_{\alpha \in I} S_\alpha$ 
\item $S_\alpha \cap \overline{S_\beta} \neq \emptyset \iff S_\alpha \subseteq \overline{S_\beta} \iff \alpha \leq \beta$. 
\end{enumerate}\end{defn}

If $X$ is an $I$-decomposed space as above, we shall call $S_\alpha, \, \alpha \in I,$ the \emph{strata} of $X$, and denote by $\Sigma$ the collection of strata of $X$ indexed by $I$. We shall use $\alpha \leq \beta$ and $S_\alpha \leq S_\beta$ interchangeably, partially ordering $\Sigma$ instead. If $S_\alpha \leq S_\beta$ and $S_\alpha \neq S_\beta$ we shall write $\alpha < \beta$ (or $S_\alpha < S_\beta$). Note that $S_\alpha < S_\beta$ is equivalent to saying that $S_\alpha$ lies in the boundary $$\partial S_\beta := \overline{S_\beta} \setminus S_\beta$$ of $S_\beta$. For any stratum $S \in \Sigma$ we define the \emph{depth} of $S$ to be 
$$\depth(S): = \sup\{n : S_i \in \Sigma \text{ such that } S < S_1 < \cdots < S_{n-1}\}.$$ 
Similarly, define  the \emph{height} of $S$ to be 
$$\height(S): = \sup\{n : S_i \in \Sigma \text{ such that } S > S_1 > \cdots >S_{n-1}\}.$$ 

We shall moreover define the \emph{depth} and \emph{dimension} of $X$ respectively as
$$\depth (X) := \sup_{\alpha \in I}\, \depth (S_\alpha), \;\;\; \dim X := \sup_{\alpha \in I} \,\dim (S_\alpha)$$

\begin{defn}\label{def-whitneyab}Let $(S, L)$ be a pair of smooth (not necessarily properly) embedded submanifolds of a smooth manifold $M$ such that $S \subset \overline{L}$.
\begin{enumerate}
\item The pair $(S, L)$ is said to satisfy \emph{Whitney condition (a)} if for any sequence $\{x_n\} \subset L$ converging to $x \in S$ such that the sequence of tangent planes $T_{x_n} L$ converge to a plane $\tau \subset T_x M$, the inclusion $T_x S \subset \tau$ holds. 
\item The pair $(S, L)$ is said to satisfy \emph{Whitney condition (b)} if the following holds.\\ Let  $\{x_n\} \subset L$, $\{y_n\} \subset S$ be any pair of sequences both converging to $x \in S$ such that the tangent planes $T_{x_n} L$ converge to some plane $\tau \in T_x M$. Further suppose that the secants $\overline{x_n y_n}$ converge to a line $\ell \in T_x M$. Then $\ell \subset \tau$.
\end{enumerate}
\end{defn}

The notions of convergence of planes and lines mentioned above are defined locally by choosing a coordinate chart around $x$ in $M$, such that the chart  contains a tail of  the sequences $\{x_n\}, \{y_n\}$. It is straightforward to check that the property of the pair $(S, L)$ satisfying either of the above conditions is independent of the chosen coordinate chart. See \cite{mather-notes} for a coordinate-free restatement of condition $(b)$. 
Note that condition $(b)$ implies condition $(a)$, since given any sequence $\{x_n\} \subset L$ with $T_{x_n} L \to \tau$ and a line $\ell \subset T_x S$, defined in a local chart $(U, x) \cong (\mathbb R^n, 0)$ around $x$, one can construct a pair of sequences $\{x_n\} \subset L$ and $\{y_n\} \subset S$ such that the secants $\overline{x_n y_n}$ converge to $\ell \in T_x M$. Now as $T_{x_n} L \to \tau$, we must have $\ell \subset \tau$ by hypothesis of satisfying condition $(b)$. Therefore, $\ell \subset \tau$. But $\ell \subset T_x S$ was arbitrary, so we conclude $T_x S \subset \tau$. Therefore, condition $(a)$ holds.

\begin{defn}\label{def-whitneyss}  A \emph{Whitney stratified subset} of a smooth manifold $M$ is a subset $X \subset M$ with an $I$-decomposition $\Sigma$ such that every stratum in $\Sigma$ is a {smoothly} embedded submanifold of $M$, and any pair of strata {$(S, L)$ in $\Sigma$ such that $S < L$} satisfies the Whitney condition $(b)$.\end{defn}

Thom and Mather showed that every Whitney stratified subset of a smooth manifold has a canonical local model akin to manifolds being locally modeled by Euclidean spaces. This gives us an intrinsic definition of a topological Whitney stratified set. The cone on a topological space $A$ is denoted as {$cA$}.

\begin{defn}\cite{friedman-notes}\label{def-cs} A {\emph{CS set}}\footnote{{The acronym ``CS" stands for ``conically smooth".}} is an $I$-decomposed space $(X, \Sigma)$ such that for any stratum $S \in \Sigma$ the following holds: \\ For any point $x \in S$ there exists an open neighborhood $U$ of $x$ in $X$, {a chart $V$} around $x$ in $S$, and a stratified space $(A, \Sigma_A)$ called the \emph{link} of $x$, such that there is a stratum-preserving homeomorphism {$\varphi : V \times cA \to U$} where $U$ {is} given the induced stratification from $X$.\end{defn}

\begin{theorem}\cite{mather-notes}\label{mather-cs}
	Every Whitney stratified subset of a smooth manifold is a CS set.
\end{theorem}

Given an $I$-decomposed space $(X, \Sigma)$, a \emph{tubular neighborhood system} or simply, a \emph{tube system} $\mathcal{N}$ on $\Sigma$ is a collection of triples $(N_\alpha, \pi_\alpha, S_\alpha)_{\alpha \in I}$ consisting of (for every $\alpha \in I$) an open neighborhood $N_\alpha$ of the stratum $S_\alpha$ in $X$, called the \emph{tubular neighborhood} of the stratum, a retraction $\pi_\alpha : N_\alpha \to S_\alpha$, called the \emph{tubular projection}, and a continuous function $\rho_\alpha : N_\alpha \to [0, \infty)$ such that $\rho_\alpha^{-1}(0) = S_\alpha$: $\rho_\alpha$ is  called the \emph{radial function}.

We now define an  \emph{abstract stratification} in the sense of Mather \cite{mather-notes}. It provides a notion of a smooth stratification on an $I$-decomposed space $(X, \Sigma)$ independent of the ambient space it is embedded in. This is analogous to the abstract definition of a smooth manifold using a smooth atlas rather than via an embedding.

\begin{defn}\label{def-abtractstratsp} An $I$-decomposed space $(X, \Sigma)$ equipped with a tube system $\mathcal{N}$ on $\Sigma$, denoted by the triple $(X, \Sigma, \mathcal{N})$, defines an \emph{{abstractly} stratified space} if the following holds 
\begin{enumerate}
	\item Each stratum $S_\alpha \in \Sigma$ {is a} smooth manifold. 
	\item For any pair $\alpha, \beta \in I$ of indices such that $\alpha < \beta$, we set $N_{\alpha\beta} = N_\alpha \cap S_\beta$ and the restrictions of $\pi_\alpha$ and $\rho_\alpha$ to $N_{\alpha\beta}$ are denoted by $\pi_{\alpha\beta}$ and  $\rho_{\alpha\beta}$ respectively. {The maps $\pi_{\alpha\beta} : N_{\alpha\beta} \to S_\beta, \rho_{\alpha\beta} : N_{\alpha\beta} \to (0, \infty)$ are smooth\footnote{{The smooth structure on the domain $N_{\alpha\beta}$ of $\pi_{\alpha\beta}, \rho_{\alpha\beta}$ is induced as an open subset of the smooth manifold $S_\alpha$.}}, and the map $(\pi_{\alpha\beta}, \rho_{\alpha\beta}) : N_{\alpha\beta} \to S_\alpha \times (0, \infty)$ is a submersion.}
	\item For all triples $\alpha, \beta, \gamma \in I$ of indices with $\alpha < \beta < \gamma$, the \emph{$\pi$-control condition} $\pi_{\alpha\beta}\pi_{\beta\gamma}(x) = \pi_{\alpha\gamma}(x)$ and the \emph{$\rho$-control condition} $\rho_{\alpha\beta}\pi_{\beta\gamma}(x) = \rho_{\alpha\gamma}(x)$ are satisfied whenever $x \in N_{\beta\gamma} \cap N_{\alpha\gamma} \cap \pi_{\beta\gamma}^{-1}(N_{\alpha\beta})$. 
\end{enumerate}
	\begin{comment}
		content...

	If $(X, \Sigma, \mathcal{N})$ satisfies all the above conditions except the $\rho$-control condition, it is called an \emph{{weakly abstractly} stratified space}.
	\end{comment}
\end{defn}

For simplicity of notation, we will often denote the tubular neighborhood of a stratum $S \in \Sigma$ in an {abstractly} stratified space $(X, \Sigma, \mathcal{N})$ by $N_S$ and the associated tubular function and radial function will be denoted by $\pi_S$ and $\rho_S$. For two strata $S, L \in \Sigma$, $S < L$, the tubular neighborhood of $S$ in $L$ will be defined as $N_{SL} := N_S \cap L$ and the restrictions of $\pi_S$ and $\rho_S$ to $N_{SL}$ will be denoted by $\pi_{SL}$ and $\rho_{SL}$ respectively. Let
$$N_S^\varepsilon := \{x \in N_S : \rho_S(x) < \varepsilon(\pi_S(x))\} \subset N_S,$$ where $\varepsilon : S \to (0, \infty)$ is a {\emph{positive smooth function}}. We  shall often use the shorthand $\varepsilon > 0$ if there is no scope of confusion. $N_S$ shall usually mean $N_S^1$.

{Here, we use $\varepsilon$ to denote a function on $S$ defining a tubular neighborhood of $S$. However, for most applications, we shall need to consider the function $\varepsilon$ only on relatively compact subsets of $S$, where it may be taken to be a small constant, hence this notation.}

Two tube systems $\mathcal{N} = (N_S, \pi_S, \rho_S)_{S \in \Sigma}$ and $\mathcal{N}' = (N'_S, \pi'_S, \rho'_S)_{S \in \Sigma}$ on an $I$-decomposed space $(X, \Sigma)$ {shall be declared} \emph{equivalent} if for any strata $S \in \Sigma$, there exists an open neighborhood $S \subset N''_S \subset N_S \cap N'_S$ such that $\pi_S|N''_S = \pi'_S|N'_S$ and $\rho_S|N''_S = \rho'_S|N''_S$. If $(X, \Sigma_X, \mathcal{N}_X)$ and $(Y, \Sigma_Y, \mathcal{N}_Y)$ are two {abstractly} stratified spaces and $f : X \to Y$ is a stratum-preserving homeomorphism such that the pulled back tube system $f^* \mathcal{N}_Y = (f^* N_S, f^{-1} \circ \pi_S \circ f, \rho_S \circ f)_{S \in \Sigma_Y}$ is equivalent to $\mathcal{N}_X$, then $f$ is said to be an \emph{isomorphism} between $X$ and $Y$.\\

{Examples of stratified spaces include manifolds with corners. A product $X \times Y$ of stratified spaces $X, Y$ is naturally stratified with strata consisting of products of strata of $X$ and $Y$, however there is no canonically defined abstract stratification in general. For instance, consider $I \times I$ where $I = [0, 1]$ is stratified as a manifold with boundary. {We describe the tube system near the corner stratum $S = (0, 0) \in I \times I$. Pick a metric on $[0, 1] \times [0, 1]$, let $N_S$ be a quarter-disk around $(0, 0)$, let $\pi_S : N_S \to S$ be the constant retraction and let $\rho_S$ be the radial distance function from the origin with respect to the metric. This choice constrains the behavior of the tube system for the edge strata $L_1 = (0, 1) \times \{0\}$ and $L_2 = \{0\} \times (0, 1)$ abutting $S$:  $N_{L_1}, N_{L_2}$ must be constructed so that they shrink as they approach $(0, 0)$ and so that $\pi_{L_i}$ preserve the radial lines $\{\rho_S = \mathrm{const}.\}$ in $N_S \cap N_{L_i}$, $i = 1, 2$.} However, if one of $X$ or $Y$ is a manifold, $X \times Y$ does have a canonical abstract stratification. {Indeed, suppose $X$ is a manifold and $S$ is a stratum of $Y$. Then we may immediately set $N_{X \times S} = X \times N_S$, $\pi_{X \times S} = \mathrm{id}_X \times \pi_S$ and $\rho_{X \times S}(x, y) = \rho_S(y)$.}

\begin{theorem}\cite{mather-notes}\label{mather} Any Whitney stratified subset $(X, \Sigma) \subset M$ admits {a tubular neighborhood system consisting of (not necessarily properly) embedded tubular neighborhoods $\nu(S)$, one for each stratum $S \in \Sigma$ in $M$}. Further, there exists a projection $\pi_S : \nu(S) \to S$ and radial function $\rho_S : \nu(S) \to [0, \infty)$ such that $N_S = \nu(S) \cap X$. The restrictions of $\pi_S$ and $\rho_S$ to $N_S$ furnish a tube system $\mathcal{N} = (N_S, \pi_S, \rho_S)_{S \in \Sigma}$ which {makes $(X, \Sigma, \mathcal{N})$ an abstractly stratified space}.\end{theorem}

The following theorem is a version of the Whitney embedding theorem for {abstractly} stratified spaces, essentially saying that every {abstractly} stratified space of dimension $n$ {(i.e., maximal dimension of strata is $n$)} can be embedded in $\mathbb{R}^N$ for $N \geq 2n+1$ as a Whitney stratified space, and any two such embeddings are isotopic if $N \geq 2n+2$.

\begin{theorem}\cite{natsume}\label{natsume} Let $(X, \Sigma, \mathcal{N})$ be an {abstractly} stratified space with $\dim X = n$. Then for any $N \geq 2n+1$ there is a \emph{realization} of $X$ in $\mathbb{R}^N$, i.e.\ there exists an embedding $\iota : X \to \mathbb{R}^N$ such that $X' = \iota(X)$ is a Whitney stratified subset of $\mathbb{R}^N$ with a stratification $\Sigma' = \{\iota(S) : S \in \Sigma\}$ and a tube system $\mathcal{N}' = (N_S, \pi_S, \rho_S)$ {as in Theorem \ref{mather}, such that 
$$\iota : (X, \Sigma, \mathcal{N}) \to (X, \Sigma', \mathcal{N}')$$
is an isomorphism of {abstractly} stratified spaces}. Moreover if $N \geq 2n+2$ any two such embeddings $\iota_0, \iota_1 : X \to \mathbb{R}^N$ are isotopic in the following sense: there is a realization $H : X \times I \to \mathbb{R}^N$ such that $H(x, t) = (H_t(x), t)$ where $H_t : X \to \mathbb{R}^N$ is a realization for all $0 \leq t \leq 1$ and $H_0 = \iota_0, H_1 = \iota_1$.\end{theorem}

\begin{theorem}\cite{goresky-triangulation}\label{goresky-triangulation}
	Any {abstractly} stratified space admits a triangulation by smoothly embedded simplices compatible with the filtration by {stratum-closures.}\end{theorem}

\subsection{Stratified maps}\label{sec-stratfdmaps}

\begin{defn}\label{def-controlledmap} A map $f : (X, \Sigma_X) \to (Y, \Sigma_Y)$ of $I$-decomposed spaces is said to be a {\emph{stratum-preserving map}} if for any $S \in \Sigma_X$, there is a unique $L \in \Sigma_Y$ such that $f(S) \subset L$. Equivalently, for every stratum $L \in \Sigma_Y$, $f^{-1}(L)$ is a disjoint union of strata of $\Sigma_X$.
	
	If $(X, \Sigma_X, \mathcal{N}_X)$ and $(Y, \Sigma_Y, \mathcal{N}_Y)$ are {abstractly} stratified spaces, then a stratum-preserving map $f : X \to Y$ of the underlying $I$-decomposed spaces is said to be a \emph{controlled map} if for any stratum $S \in \Sigma_X$ and the corresponding unique stratum $L \in \Sigma_Y$ such that $f(S) \subset L$, the following conditions are satisfied 
{\begin{enumerate}
\item $f|S : S \to L$ is a smooth map. 
\item There exists $\varepsilon > 0$ such that $f(N_S^\varepsilon) \subset N_L$. 
\item The \emph{$\pi$-control} condition 
$$f(\pi_S(x)) = \pi_L(f(x))$$ 
and the \emph{$\rho$-control} condition 
$$\rho_S(x) = \rho_L(f(x))$$ 
hold for all $x \in N_S^\varepsilon$.
\end{enumerate}}
	
	If all the conditions above except the $\rho$-control condition is satisfied, $f$ is said to be a \emph{weakly controlled map}.\end{defn}

\begin{defn}\label{def-stratfdsubmimm} A controlled map $f : (X, \Sigma_X, \mathcal{N}_X) \to (Y, \Sigma_Y, \mathcal{N}_Y)$ is a \emph{stratified submersion} if   for any stratum $L \in \Sigma_Y$ and any component $S \in \Sigma_X$ of $f^{-1}(L)$, $f|S : S \to L$ is a submersion.
	
\begin{comment}
	content...
	A controlled map $f : (X, \Sigma_X) \to (Y, \Sigma_Y)$ is a {\bf stratified immersion} if for every stratum $S \in \Sigma_X$ 
	and the corresponding unique stratum 
	$L \in \Sigma_Y$ containing $f(S)$, 
	$f|S : S \to L$ is an immersion. If $f$ is moreover injective, we shall call it a {\bf stratified embedding}.
\end{comment}
\end{defn} 

\begin{comment}\label{rmk-term}
	Stratified submersions were called stratified maps in \cite{GM_SMT}; however since we shall have need for stratified immersions as well, we have preferred to use this alternate terminology.
\end{comment}

\subsection{Stratified bundles}\label{sec-stratfdbdl}

\begin{defn}\cite{mather-notes}\label{def-strvf} Let $(X, \Sigma, \mathcal{N})$ be an {abstractly} stratified space. A \emph{stratified vector field} $\eta$ on $X$ is a collection $\{\eta_S : S \in \Sigma\}$ where for each stratum $S \in \Sigma$, $\eta_S$ is a smooth vector field on $S$. The
	stratified vector field $\eta$ will be called a \emph{controlled vector field} if for any pair $S, L \in \Sigma$ of strata with $S < L$, there exists some $\varepsilon > 0$ such that for any $x \in N^\varepsilon_S \cap L$, the following conditions hold:
	\begin{enumerate}
		\item $\eta_L \rho_{SL}(x) = 0$.
		\item $(\pi_{SL})_*\eta_L(x) = \eta_S(\pi_{SL}(x))$.
	\end{enumerate} 
	If we simply drop the first condition, we obtain a \emph{weakly controlled vector field}.
\end{defn}
Thus, a controlled vector field in the higher dimensional  stratum $L$ is  {parallel to the lower dimensional  stratum $S$}, i.e. it does not change along the radial direction $\rho_{SL}$. This is ensured by the
first condition above. It also projects `isomorphically' to the vector field in the lower dimensional  stratum $S$. This is ensured by the
second condition above. A weakly controlled vector field in the higher dimensional  stratum $L$ is not necessarily {parallel to the lower dimensional  stratum $S$}, i.e.\ it is allowed to have a radial component; however, if the radial component is subtracted from a weakly controlled vector field, we obtain a controlled vector field.

We now define higher-dimensional controlled distributions. These will be useful in defining stratified fiber bundles below. Note that we drop the $\rho$-control condition in this case.

\begin{defn}\label{def-strdist}Let $(X, \Sigma, \mathcal{N})$ be an {abstractly} stratified space. A \emph{stratified distribution} $D$ on $X$ is a collection $\{D_S : S \in \Sigma\}$ where for each stratum $S \in \Sigma$, $D_S$ is a smooth subbundle of the tangent bundle $TS$ of $S$.
	The distribution 
	$D$ will be called a \emph{weakly controlled distribution} if
	for any pair $S, L \in \Sigma$ of strata with $S < L$, there exists some $\varepsilon > 0$ such that for any {$x \in N^\varepsilon_S \cap L$}, $$(\pi_{SL})_* D_L(x) = D_S(\pi_{SL}(x)).$$ \end{defn}

Note that the dimensions of $D_S, D_L$ may differ for $S \neq L$. For the next definition, we shall need the local structure of neighborhoods $N_S$ of strata $S$.
By Thom's first isotopy lemma, $N_S$ is a fiber bundle over $S$ with fiber $cA$, where $A$ denotes the link of $S$ in $X$.

\begin{defn}\label{def-strbdl} A triple $(E, X, p)$ consisting of 
	
	\begin{enumerate}
		
		\item {An} {abstractly} stratified {space} $(X, \Sigma, \mathcal{N})$, called the \emph{base space},
		
		\item {An} {abstractly} stratified space $(E, \widetilde{\Sigma}, \widetilde{\mathcal{N}})$, called the \emph{total space},
		
		\item {A} weakly controlled map $p : E \to X$ called the \emph{bundle projection},
		
	\end{enumerate}
	will be called a \emph{stratified fiber bundle} if 
	\begin{enumerate}
		\item[(i)] For every stratum $\widetilde{S} \in \widetilde{\Sigma}$ and the corresponding unique stratum $S \in \Sigma$ such that $p(\widetilde{S}) \subseteq S$, the restriction $p : \widetilde{S} \to S$ is a {smooth} fiber bundle.
		\item[(ii)] The stratified distribution $\ker dp := \{\ker d(p|_{\widetilde{S}}) : \widetilde{S} \in \widetilde{\Sigma}\}$ on $E$ is weakly controlled.
		\item[(iii)]  Let $p|N_{\til{S}}: N_{\til{S}} \to N_{{S}}$ denote the restriction of $p$ to a neighborhood $N_{\til{S}}$ of ${\til{S}}$. Let $B, A$ denote the links of ${\til{S}},S$ respectively, so that $N_{\til{S}}$ (resp.\ $N_{{S}}$) is a bundle over ${\til{S}}$ (resp.\ $S$) with fiber $cB$
(resp.\ $cA$).  Identify $\til S$   with the zero-section of $N_{\til{S}}$. We demand that $(p|N_{\til{S}})^{-1}  ({S}) = \til{S}$.
	\end{enumerate} 
	
	Given a stratified fiber bundle $(E, X, p)$, a \emph{(weakly) controlled section} of $p$ is a (weakly) controlled map $s : X \to E$ such that $p \circ s = \mathrm{id}$.\end{defn}

\begin{rmk}
{It was pointed out to us by the referee that a related notion of a stratified fiber bundle has been defined in algebraic topology \cite{baues}, where the strata are bundles and the attachment of strata is controlled by a structure category of fibers. In this paper, smooth structures are essential,
as we are interested in the differential topology of stratified fiber bundles in the sense of Definition \ref{def-strbdl}. To differentiate these two notions, 
stratified fiber bundles in the sense of Definition \ref{def-strbdl} should, strictly speaking, be called \emph{smooth stratified fiber bundles}. However, since stratified fiber bundle in the sense of  \cite{baues} will never be used in this paper, we 
shall, in what follows, use the terminology `stratified fiber bundles' in the 
sense of Definition \ref{def-strbdl}.}
\end{rmk}

We should point out that a stratified bundle $(E, X, p)$ as defined above is necessarily locally trivial over strata {by Thom's second isotopy lemma, originally formulated in \cite{Thom_stratmaps} (for a detailed proof, see \cite[Proposition 11.2]{mather-notes})}. Hence, we can think of a stratified bundle as a collection $\{(E_S, S, p): S \in \Sigma\}$, where each $p : E_S \to S$ is a genuine topological bundle over the stratum $S$, {where the fiber can a priori be a stratified space}. {The conditions of Definition \ref{def-strbdl} ensure that these bundles patch together consistently. Condition (ii) is known as Thom's condition $(a_p)$ in literature \cite[Section 11]{mather-notes}}. {As a consequence of Condition (ii), we have that whenever $S < L$ is a pair of strata in $X$, the dimension of the fiber of $E_S$ is less than that of $E_L$. Lastly, Condition (iii) ensures that for every stratum $S \in \Sigma$ of $X$, $p^{-1}(S)$ does not contain a pair of strata which abut each other, therefore $p : E_S \to S$ is a fiber bundle such that the connected components of any fiber is a manifold.}

{Let $S$ be any stratum  of $X$ and $\widetilde{S}$ a stratum of $E$ lying above $S$.  Let  $N_S \subset X$ and $N_{\tilde{S}} \subset E$ be the  corresponding tubes. Since, further, the projection $p$ is weakly controlled,  we must have $(p|N_{\tilde{S}})^{-1}(cA) = cB$. Here, $cA$ and $cB$ are the conical fibers of the tubular projection maps $\pi_S : N_S \to S$ and $\pi_{\tilde{S}} : N_{\tilde{S}} \to \widetilde{S}$, respectively. Moreover, Condition (iii) ensures that the preimage of the cone point $c_A$ of $cA$ under $p|N_{\tilde{S}}$ is exactly the cone point $c_B$ of $cB$. This will be useful in Corollary \ref{cor-trivialzn} below.}

\begin{eg}
{Let $U(1)$ act on $S^2$ by rotations about the axis through a pair of antipodal points $\{\pm x\} \subset S^2$. Consider the stratification of $S^2$ by orbit-type, given by $\Sigma = \{\{x\}, \{-x\}, S^2 \setminus \{\pm x\}\}$. The quotient map $p : S^2 \to S^2/U(1) \cong [-1, 1]$ is an example of a stratified bundle.}

{In general, for a connected compact Lie group $G$ acting on a smooth manifold $M$, the quotient map $p : M \to M/G$ is a stratified bundle. For details, see \cite{Davis_orbifib} and \cite{Verona_orbistr}.}
\end{eg}

\begin{comment}
\begin{rmk}\label{rmk-allsxnswc}
Given a stratified fiber bundle $(E, X, p)$,  $p$ is by definition a weakly controlled map. Hence any section  $s : X \to E$ is also necessarily weakly controlled.
\end{rmk}
\end{comment}

We shall sometimes use the suggestive notation $p:E \to X$ for a stratified fiber bundle. The following {lemma} and its consequences (Corollary \ref{strbdl-cone-link} and Corollary \ref{cor-trivialzn}) give the \emph{local structure} of stratified fiber bundles.

\begin{lemma}\label{lem-strbdl-trivialization} Let $(E, X, p)$ be  a stratified fiber bundle.  For any point $\widetilde{x} \in E$ with $p(\widetilde{x}) = x$, there is an open neighborhood $V$ of $\widetilde{x}$ in $E$, and $U$ of $x$ in $X$, equipped with the respective induced stratifications, such that $p(V) = U$, and
{\begin{enumerate}
\item There exist {abstractly} stratified spaces $(A, \Sigma_A, \mathcal{N}_A)$, $(B, \Sigma_B, \mathcal{N}_B)$ and isomorphisms of {abstractly} stratified spaces $\psi : V \to cB \times \mathbb{R}^n$ and $\varphi : U \to cA \times \mathbb{R}^m$ for some $n \geq m$, 
\item There exists a map $f : cB \times \mathbb{R}^n \to cA \times \mathbb{R}^m$ which factors as $f = (g, \mathrm{proj})$ where $\mathrm{proj} : \mathbb{R}^n \to \mathbb{R}^m$ denotes the projection to the first $m$ coordinates, 
\end{enumerate}
making the following diagram commute:
$$
\begin{CD}
V @>>> cB \times \mathbb R^n \\
@VVV @VVV \\
U @>>> cA \times \mathbb R^m
\end{CD}
$$}
\end{lemma}

\begin{proof} Suppose $\widetilde{S}$ is the unique stratum of $E$ containing $\widetilde{x}$. Let $S$ be the unique stratum of $X$ containing $x$. So $p(\widetilde{S}) \subseteq S$. Let $\widetilde{\pi}$ be the tubular projection associated to $\widetilde{S}$ in $E$ and let $\pi$ be the tubular projection associated to $S$ in $X$. Since $p|\widetilde{S} : \widetilde{S} \to S$ is a fiber bundle, we can choose charts $\widetilde{O} \cong \mathbb{R}^n$ around $\widetilde{x}$ in $\widetilde{S}$ and $O \cong \mathbb{R}^m$ around $x$ in $S$ such that $p : \widetilde{O} \to O$ is equivalent to the projection $\mathrm{proj} : \mathbb{R}^n \to \mathbb{R}^m$ with respect to local coordinates. Let {$V = \widetilde{\pi}^{-1}(\widetilde{O}) \cap \mathrm{cl}(N^\varepsilon_{\widetilde{S}})$} and $U = \pi^{-1}(O) \cap \mathrm{cl}(N^\varepsilon_S)$ for some appropriate $\varepsilon > 0$
	(here $\mathrm{cl}(-)$ denotes closure). Observe that $\pi : U \to O$ is a proper stratified submersion. Choose  coordinate vector fields $\partial_1, \cdots, \partial_m$ on $O$ corresponding to local coordinates $t_1, \cdots, t_m$. By \cite[Proposition 9.1]{mather-notes} there exist controlled vector fields $\eta_1, \cdots, \eta_m$ on $U$ which commute stratumwise such that $\pi_* \eta_i = \partial_i$ for all {$1 \leq i \leq m$}.  Let $F = \pi^{-1}(x) \cap \mathrm{cl}(N^\varepsilon_S)$.  
	Let $\Phi_i$ be the local $1$-parameter family of stratum-preserving homeomorphisms on $U$ generated by $\eta_i$, for $0 \leq i \leq m$. For any $u \in U$, there exists a unique $v \in F$ and unique $(t_1, \cdots, t_m)\in O$, such that {$(\Phi_1^{t_1} \circ \Phi_2^{t_2}\cdots \circ \Phi_m^{t_m})(v)=u$}. 
	This gives an inverse homeomorphism $h : U \to F \times O$ defined by {$$h(u) = ((\Phi_m^{-t_m} \circ \cdots \circ \Phi_2^{-t_2} \circ \Phi_1^{-t_1})(u), t_1, \cdots, t_m),$$} where $\pi_S(u) = (t_1, \cdots, t_m)$, so that $t_1, \cdots, t_m$ are (implicitly) functions of $u$.
	
Now consider the commutative square
	{$$
	\begin{CD}
	V @>{\widetilde{\pi}}>> \widetilde{O} \\
	@V{p}VV @V{p}VV \\
	U @>{\pi}>> O
	\end{CD}
	$$}
This gives us a map to the fibered product $(\widetilde{\pi}, p) : V \to \widetilde{O} \times_O U$. Note that since $\ker dp$ is a weakly controlled distribution on $E$ by hypothesis,  $\widetilde{\pi}_* \ker dp_v = \ker dp_{\widetilde{\pi}(v)}$, for any $v \in V$,  i.e.\ $d\widetilde{\pi}$ restricts to a surjection $d\widetilde{\pi} : \ker dp_v \to \ker dp_{\widetilde{\pi}(v)}$. We now use a fact from linear algebra:
	
	\begin{claim}\label{claim-linal} Let $\mathbf{W}_1, \mathbf{W}_2, \mathbf{W}_3, \mathbf{W}_4$ be vector spaces occurring in the following  commutative diagram 
		{$$
	\begin{CD}
	\mathbf{W}_1 @>{f}>> \mathbf{W}_2\\
	@V{p}VV @V{q}VV \\
	\mathbf{W}_3 @>{g}>> \mathbf{W}_4
	\end{CD}
	$$}

\noindent 	where $f, g, p, q$ are all surjective linear maps. {If $f$ restricts to a surjection $f : \ker p \to \ker q$ then the induced map to the fibered product $(f, p) : \mathbf{W}_1 \to \mathbf{W}_2 \times_{\mathbf{W}_4} \mathbf{W}_3$ is surjective.}
	\end{claim}
	
	\begin{proof}[Proof of Claim \ref{claim-linal}] For any $u \in \mathbf{W}_2$ and $v \in \mathbf{W}_3$ such that $q(u) = g(v)$, choose lifts $\widetilde{u}, \widetilde{v} \in \mathbf{W}_1$ such that $f(\widetilde{u}) = u$ and $p(\widetilde{v}) = v$ by surjectivity of $f$ and $p$, respectively. Then observe that $(q \circ f)(\widetilde{u} - \widetilde{v}) = q(u) -  (q \circ f)(\widetilde{v}) = q(u) - (g \circ p)(\widetilde{v}) = q(u) - g(v) = 0$ by commutativity of the diagram. Therefore, $f(\widetilde{u} - \widetilde{v}) \in \ker q$. As $f : \ker p \to \ker q$ is surjective, there must be some $k \in \ker p$ such that $f(\widetilde{u} - \widetilde{v}) = f(k)$. Therefore there must also be some $\ell \in \ker f$ such that $\widetilde{u} - \widetilde{v} = k + \ell$. Let $w = \widetilde{u} - \ell = \widetilde{v} + k \in \mathbf{W}_1$. This is the desired element.\end{proof}
	
	We now return to the proof of Lemma \ref{lem-strbdl-trivialization}. Note that
	$\widetilde{O} \times_O U$ is an abstractly stratified space. Claim \ref{claim-linal} implies $(\widetilde{\pi}, p) : V \to \widetilde{O} \times_O U$ is a stratumwise submersion.  {Therefore, there exist controlled vector fields $\widetilde{\eta}_i$ on $V$ {\textit{over}} $\eta_i$ on $U$ for $1 \leq i \leq m$, see \cite[Proposition 11.5]{mather-notes} (we pause here to record a warning that controlled vector fields \textit{over} controlled vector fields are not controlled vector fields in the usual sense of the word, see \cite[Section 11]{mather-notes} for a careful discussion)}. In particular, $\widetilde{\pi}_* \widetilde{\eta_i} = \widetilde{\partial}_i$ for all $1 \leq i \leq m$, where $\widetilde{\partial}_1, \cdots, \widetilde{\partial}_m$ are the first $m$ coordinate vector fields on $\widetilde{O}$ obtained as lifts of $\partial_1, \cdots, \partial_m$ by the projection $p : \widetilde{O} \to O$. Consider the rest of the coordinate vector fields $\widetilde{\partial}_{m+1}, \cdots, \widetilde{\partial}_n$ on $\widetilde{O}$ and let $\widetilde{\eta}_{m+1}, \cdots, \widetilde{\eta}_n$ be controlled vector fields on $V$ such that $\widetilde{\pi}_* \widetilde{\eta}_i = \widetilde{\partial}_i$ for $m+1 \leq i \leq n$.
	
	Let $\widetilde{\Phi}_i$ denote the local $1$-parameter family of stratum-preserving homeomorphisms of $V$ generated by $\widetilde{\eta_i}$ for $1 \leq i \leq n$. We obtain, as before, two homeomorphisms $\widetilde{h}_1 : V \to \widetilde{F} \times \widetilde{O}$ given by {$\widetilde{h}_1(v) = (\widetilde{\Phi}_n^{-t_n} \circ \cdots \circ \widetilde{\Phi}_1^{-t_1}(v), t_1, \cdots, t_n)$ and $\widetilde{h}_2 : V \to F' \times O$ given by $\widetilde{h}_2(v) = (\widetilde{\Phi}_m^{-t_m} \circ \cdots \circ \widetilde{\Phi}_1^{-t_1}(v), t_1, \cdots, t_m)$} by considering the flow generated by all of $\widetilde{\eta}_1, \cdots, \widetilde{\eta}_n$ in the first case, and the flow generated by the first $m$ of these, namely $\widetilde{\eta}_1, \cdots, \widetilde{\eta}_m$, in the second case. 

{Here $\widetilde{F} = \widetilde{\pi}^{-1}(\widetilde{x}) \cap N^\varepsilon_{\widetilde{S}}$ and $F' = (p \circ \widetilde{\pi})^{-1}(x) \cap N^\varepsilon_{\widetilde{S}}$. Consider the map
\begin{gather*}\phi : \widetilde{F} \times \widetilde{O} \to F' \times O \\
\phi(z, t_1, \cdots, t_n) = (\Phi_{m+1}^{t_{m+1}} \circ \cdots \circ \Phi_n^{t_n}(z), t_{m+1}, \cdots, t_n)\end{gather*}}
	Then the following diagram commutes:
	{$$
	\begin{CD}
	V @>{\widetilde{h}_1}>> \widetilde{F} \times \widetilde{O}\\
	@V{\mathrm{id}}VV @V{\phi}VV \\
	V @>{\widetilde{h}_2}>> F' \times O
	\end{CD}
	$$}
	Since $\widetilde{\eta}_1, \cdots, \widetilde{\eta}_m$ are controlled vector fields on $V$ {\textit{over}} $\eta_1, \cdots, \eta_m$ on $U$, by \cite[Proposition 11.6]{mather-notes} there is also a commutative diagram as follows
	{$$
	\begin{CD}
	V @>{\widetilde{h}_1}>> F' \times O\\
	@V{p}VV @V{(p, \mathrm{id})}VV \\
	U @>{h}>> F \times O
	\end{CD}
	$$}
	By combining the two commutative diagrams above, we obtain (up to change of coordinates) an equivalence of $p : V \to U$ with {a map $\widetilde{F} \times \widetilde{O} \to F \times O$ defined by $(z, \mathbf{t}) \mapsto (g_\mathbf{t}(z), p(\mathbf{t}))$. Now, observe that the map $\widetilde{F} \times \widetilde{O} \to F \times \widetilde{O}$, $(z, \mathbf{t}) \mapsto (g_\mathbf{t}(z), \mathbf{t})$ is also a stratified fiber bundle}. So we can lift coordinate vector fields $\widetilde{\partial}_1, \cdots, \widetilde{\partial}_n$ on $\widetilde{O}$ to controlled vector fields on $F \times \widetilde{O}$ and from there to controlled vector fields on $\widetilde{F} \times \widetilde{O}$ {\textit{over}} the aforementioned controlled vector fields. Once again, using \cite[Proposition 11.6]{mather-notes}, we obtain a commutative diagram as follows
	{$$
	\begin{CD}
	\widetilde{F} \times \widetilde{O} @>{\cong}>> \widetilde{F} \times \widetilde{O}\\
	@V{p}VV @V{(g_{\mathbf{0}}, p)}VV \\
	F \times \widetilde{O} @= F \times \widetilde{O}
	\end{CD}
	$$}

\noindent  where the diagram is compatible with projections of each of the terms to $\widetilde{O}$. Therefore, by conjugating by the isomorphism on the top horizontal arrow we obtain an equivalence of $\widetilde{F} \times \widetilde{O} \to F$, $(z, \mathbf{t}) \mapsto g_{\mathbf{t}}(z)$ with $\widetilde{F} \times \widetilde{O} \to F$, $(z, \mathbf{t}) \mapsto g_{\mathbf{0}}(z)$. This shows that $p : V \to U$ is equivalent (up to reparametrization) to a genuine product 
{$$g_{\mathbf{0}} \times p : \widetilde{F} \times \widetilde{O} \to F \times O$$}
	By an application of Thom's first isotopy lemma {\cite[Proposition 11.1]{mather-notes}}, we identify $F \cong cA$ and $\widetilde{F} \cong cB$ where $A = \pi^{-1}(x) \cap \rho^{-1}(r)$ and $B = \widetilde{\pi}^{-1}(\widetilde{x}) \cap \widetilde{\rho}^{-1}(r')$ for some sufficiently small $r, r' > 0$, with the induced abstract stratification from $E$ and $X$ respectively. The lemma follows.\end{proof}

{We refine the conclusions of Lemma \ref{lem-strbdl-trivialization} by elaborating on the structure of the map $p_c : cB \to cA$ between the conical factors in the local trivialization 
$$(V, U, p|_V) \cong (cB \times \mathbb R^n, cA \times \mathbb R^m, p_c \times \mathrm{proj})$$
of a stratified fiber bundle $(E, X, p)$ furnished by Lemma \ref{lem-strbdl-trivialization}}

\begin{cor}\label{strbdl-cone-link} $p_c : cB \to cA$ is a stratified fiber bundle.\end{cor}
\begin{proof}We know
$$p_c \times \mathrm{proj} : cB \times \mathbb{R}^n \to cA \times \mathbb{R}^m$$
is a stratified fiber bundle, since it is equivalent to the stratified fiber bundle $p|V : V \to U$ by a pair of isomorphisms of abstractly stratified spaces. Then for any stratum $S$ of $cA$ and any stratum $L$ of $cB$ such that $f(L) \subset S$, the restriction $p_c \times \mathrm{proj} : L \times \mathbb{R}^n \to S \times \mathbb{R}^m$ is a smooth fiber bundle. In particular, $p_c|L  : L \to S$ is a surjective smooth submersion.

We can arrange $p_c$ to be a proper map by choosing $r, r' > 0$ appropriately in the last part of Lemma \ref{lem-strbdl-trivialization}. Therefore, $p_c|L$ is a proper surjective submersion and hence a smooth fiber bundle by Ehresmann's fibration theorem. Therefore, $p_c : cB \to cA$ is a weakly controlled map which is a stratumwise smooth fiber bundle. Moreover, it is straightforward to check that $\ker dp_c$ is a controlled distribution on $cB$ since $\ker d(p_c \times \mathrm{proj})$ is a controlled distribution on $cB \times \mathbb{R}^n$ by hypothesis. This establishes that $(cB, cA, p_c)$ satisfies all the hypotheses in Definition \ref{def-strbdl} and is therefore a stratified fiber bundle.
\end{proof}

\begin{cor}\label{cor-trivialzn}
The map $p_c : cB \to cA$ is equivalent to the cone on a map $p_{\ell} : B \to A$ between the links, by a pair of isomorphisms of the abstractly stratified spaces $cA$ and $cB$. That is, there exists a commutative diagram of the form
$$
\begin{CD}
cB @>{\cong}>> cB\\
@VV{p_c}V @VV{c(p_\ell)}V\\
cA @>{\cong}>> cA
\end{CD}
$$
\end{cor}

\begin{proof}
For concreteness, let $cA= A \times [0,1)/A\times \{0\}$ and  $cB= B \times [0,1)/B\times \{0\}$, {and let us indicate the cone points as $\{c_A\}$ and $\{c_B\}$ respectively}. Identify the $[0,1)$ factor in each with radial co-ordinates on $cA, cB$, {using the radial functions $\rho_A$ and $\rho_B$, respectively}. Let $\rho_A: cA \to [0,1)$ denote the projection onto its radial co-ordinate, and let $\Phi=\rho_A \circ p_c$.
	
Then $\Phi: cB \to [0,1)$ is a stratified fiber bundle with compact fibers, where the base $[0,1)$ has exactly two strata $\{0\}$ and $(0,1)$. Let $c_B$ denote the cone-point of $B$. Then condition (iii) of Definition \ref{def-strbdl} ensures that  $\Phi : cB \setminus \{c_B\} \to (0,1)$ is a stratified fiber bundle where the base is a single stratum, and the fibers are compact. By Thom's first isotopy lemma {\cite[Proposition 11.1]{mather-notes}}, $\Phi : cB \setminus \{c_B\} \to (0,1)$ is a product fibration, i.e.\ $cB \setminus \{c_B\}$ is isomorphic to $B \times (0,1)$ {as abstractly stratified spaces, by an isomorphism which preserves the projection to $(0, 1)$}. Reparametrizing the radial co-ordinates furnishes the conclusion.
\end{proof}

\begin{comment}\label{rmk-induct}\marginpar{{how relevant is this anymore}}
	Lemma \ref{lem-strbdl-trivialization} and Corollary \ref{cor-trivialzn} are potentially useful in inductive arguments inducting on dimension of stratified spaces. Lemma \ref{lem-strbdl-trivialization} gives a local structure for neighborhoods of $x \in S$, where $S$  is a stratum of a stratified space $X$. Let $\R^m \times cA$ denote such a local neighborhood. Then $\R^m \times cA$ factors into the manifold $(\R^m)$ factor and a compact stratified $(cA)$ factor. For $m > 0$, this directly allows us to apply induction to the $cA$ factor. Even for $m=0$, Corollary \ref{cor-trivialzn} allows us a further decomposition into the radial co-ordinate of $cA$ and the link $A$. Hence, induction may be applied to the link $A$. In any such argument, the problem thus comes down to assembling the information along the three co-ordinates:
	\begin{enumerate}
		\item the manifold $\R^m-$factor,
		\item the radial co-ordinate in $cA$, and
		\item the link $A$ in $cA$.
	\end{enumerate}
	
\end{comment}

{Let $(A, \Sigma_A, \mathcal{N}_A)$ and $(B, \Sigma_B, \mathcal{N}_B)$ be abstractly stratified spaces. The same argument as in Lemma \ref{lem-strbdl-trivialization} and Corollary \ref{cor-trivialzn} can be used to establish lifting of \emph{stratified homotopies} $H: A \times I \to B$ i.e.\ a homotopy where, for every stratum $S \in \Sigma_A$, there exists a unique stratum $L \in \Sigma_B$ such that $H(I \times S) \subset I \times L$.}

\begin{prop}\label{strathomotlift}Let $(E, B, p)$ be a stratified bundle, and $H : A \times [0, 1] \to B$ be a stratified homotopy. Let $h_0 := H|A \times \{0\}$ and let $\til{h}_0 : A \to E$ be a lift of $h_0$. Then there exists a lift $\til{H} : A \times [0, 1] \to E$ such that $\til{H}$ is a stratified homotopy, $\til{H}|A \times \{0\} = \til{h}_0$ and $\til{H}$ covers $H$, i.e. \ $\til{H} \circ p = H$.\end{prop}

\begin{proof}We modify the homotopy by enlarging $[0, 1]$ slightly to $(-\ep, 1+\ep)$ and defining $H : A \times (-\ep, 1+\ep) \to B$ by declaring $H$ to be constant on $A \times (-\ep, 0]$ and $A \times [1, 1+\ep)$. It is now possible to choose $\ep > 0$ such that $H$ is a stratified mapping, where $A \times (-\ep, 1+\ep)$ is stratified by $S \times (-\ep, 1+\ep)$; $S \in \Sigma_A$ being the strata of $A$.

Let $H^* E \subset A \times (-\ep, 1+\ep) \times E$ denote the pullback of $E$ over $A \times (-\ep, 1+\ep)$. Since $H$ is a stratified map, $H^* E$ is a stratified bundle over $A \times (-\ep, 1+\ep)$. By projecting first to $A \times (-\ep,1+\ep)$ and then to $(-\ep,1+\ep)$ as in Corollary \ref{cor-trivialzn}, we obtain a commutative diagram
$$
\begin{CD}
H^* E @>{\cong}>> E_0 \times (-\ep, 1+\ep)\\
@VVV @VVV \\
A \times (-\ep,1+\ep) @= A \times (-\ep,1+\ep)
\end{CD}
$$
Where $E_0 = h_0^* E$ is the pullback of the stratified fiber bundle $(E, B, p)$ over $A$ under $h_0 : A \to B$. The map $\til{h}_0 : A \to E$ induces a map to the fibered product $H^*(\til{h}_0) : A \to H^* E$. 

Let $\Phi : H^* E \to E_0 \times (-\ep, 1+\ep)$ denote the isomorphism above. Then, the product homotopy with coordinates changed by $\Phi$, i.e.
$$H = \Phi^{-1} \circ (\Phi \circ H^*(\til{h}_0), t)$$ gives the required lift.
\end{proof}

 A slight generalization of stratified fiber bundles is sometimes useful:

\begin{defn}\label{def-stratumwisebdl}
A \emph{stratumwise bundle} $P: E \to B$ consists of
\begin{enumerate}
\item an abstractly stratified space $E$ {called} the \emph{total space},
\item an abstractly stratified space $B$ {called} the \emph{base space},
\item a stratum-preserving map $P$, such that for every stratum $S$  of $B$,
$$P|P^{-1}(S) \, : P^{-1}(S)  \to S$$ is a topological fiber bundle, with fiber 
a stratified space $F_S$.
\end{enumerate}
\end{defn}

\begin{eg}\label{eg-pdkt}
A product of stratified spaces is a stratumwise bundle, but not necessarily a 
stratified fiber bundle. {For example, the projection map $p : I \times I \to I$ to the first coordinate, with the stratification on $I \times I$ as discussed before Theorem \ref{mather}, is not a stratified bundle: it fails Condition (iii) in Definition \ref{def-strbdl}. Indeed, the fiber of $p$ over $\{0\}$ is the stratified space $(\{0\} \times [0, 1])$. This is a manifold with boundary but not a manifold.}
\end{eg}

\begin{comment}
\begin{rmk}\label{rmk-strathomotlift}
The same argument as in 	Lemma \ref{lem-strbdl-trivialization} and Corollary \ref{cor-trivialzn} can be used to establish lifting of stratified homotopies. Let $H: A \times [0,1] \to B$ be a stratified homotopy, i.e.\ a homotopy through stratified maps. Let $h_0= H|A \times \{0\}$. Let $P:E \to B$ be a stratified bundle, and let $\til{h_0}: A \to E$ be a lift of $h_0$. Then there exists a lift  $\til{H}: A \times [0,1] \to E$ covering $H$ so that $\til{H}|A \times \{0\}=\til{h_0}$.

To see this, let $H^* E$ denote the pullback of $E$ over $A \times [0,1]$. Since $H$ is a stratified
homotopy, $H^* E$ is a stratified bundle over $A \times [0,1]$. To ensure that the strata of $A \times [0,1]$ are in one-to-one correspondence with those of $A$, we enlarge $ [0,1]$ slightly to $(-\ep,1+\ep)$ and declare that the homotopy $H$ is constant on  $ A \times (-\ep, 0] $ and on
$ A \times [1,1+\ep) $. Then, using the Thom isotopy Lemma as in 	Lemma \ref{lem-strbdl-trivialization} and Corollary \ref{cor-trivialzn} (by projecting first to $A \times (-\ep,1+\ep)$ and then to $(-\ep,1+\ep)$) guarantees that there exists a stratified bundle 
$E_0$ over $A$ such that $H^* E$ is bundle isomorphic to $E_0 \times (-\ep,1+\ep)$ over $A \times (-\ep,1+\ep)$ with the identity in the second factor. Let $\Phi: H^*E \to E_0 \times (-\ep,1+\ep)$ denote this bundle isomorphism.
Then, the product homotopy with coordinates changed by $\Phi$, i.e.\ $$\Phi^{-1} \circ (\Phi \circ H^*(\tilde{h_0}),t)$$ gives the required lift.
\end{rmk}
\end{comment}

\begin{defn}\label{str-tbl}Let $(X, \Sigma, \mathcal{N})$ be an abstract $C^\infty$-stratified space. By Theorem \ref{natsume} there is a realization $X' \subset \mathbb{R}^N$ such that $(X', \Sigma', \mathcal{N}')$ is an abstract stratified set, where $\mathcal{N}'$ is induced from a tubular neighborhood system $(\nu(S), \pi_S, \rho_S)_{S \in \Sigma'}$ on the Whitney stratified set $X' \subset \mathbb{R}^n$. We define the {\bf tangent bundle} $TX$ to be the union $\bigcup_{S \in \Sigma'} TS \subset T\mathbb{R}^N$ of tangent bundles to each strata of $X'$. This inherits a topology from $T\mathbb{R}^N = \mathbb{R}^{2N}$ and an $I$-decomposition $\Sigma^{(1)} = \{TS_\alpha : S_\alpha \in \Sigma'\}_{\alpha \in I}$. 
	
	Let $p : TX \to X$  be the {\bf projection} obtained by restricting the projection $T\mathbb{R}^N \to \mathbb{R}^N$ to $TX$ and composing with the inverse of the realization homeomorphism $X \to X'$. Define $$N^{(1)}_\alpha := T\nu(S_\alpha) \cap TX$$ to be the tube around the stratum $TS_\alpha$ of $TX$, i.e.\ $N^{(1)}_\alpha$ is the intersection of $TX$ with the tangent bundle to the tubular neighborhood $\nu(S_\alpha) $.
	
	The associated tubular projection $\pi^{(1)}_\alpha := d\pi_\alpha$ is defined by  the restriction of the differential $$d\pi_\alpha : T\nu(S_\alpha) \to TS_\alpha$$ to $N^{(1)}_\alpha.$ 
	
	Finally, consider the differential of the radial function $d\rho_\alpha : T\nu(S_\alpha) \to \mathbb{R}$ as a map to the fibers of $T[0, \infty) \cong [0, \infty) \times \mathbb R$. Then we define the radial function associated to the stratum $TS_\alpha$ as 
	$$\rho^{(1)}_\alpha := \rho_\alpha \circ p + (d\rho_\alpha)^2.$$ 
	
	We denote the tube system defined by these functions as $\mathcal{N}^{(1)} = (N^{(1)}_\alpha, \pi^{(1)}_\alpha, \rho^{(1)}_\alpha)_{\alpha \in I}$.\end{defn}

\begin{lemma}\label{lem-txstratfd} The triple $(TX, \Sigma^{(1)}, \mathcal{N}^{(1)})$ is an abstract $C^\infty$-stratified space.\end{lemma}

\begin{proof}For any pair of indices $\alpha, \beta \in I$ with $\alpha < \beta$, $N_{\alpha\beta} = N_\alpha \cap S_\beta$ is a submanifold of $S_\beta$. Hence it inherits a $C^\infty$-structure. Consider $\pi_{\alpha\beta} : N_{\alpha\beta} \to S_\alpha$ and $\rho_{\alpha\beta} : N_{\alpha\beta} \to (0, \infty)$ $-$ both $C^\infty$-maps. If $\pi^{(1)}_{\alpha\beta}$ and $\rho^{(1)}_{\alpha\beta}$ denote the restrictions of $\pi^{(1)}_\alpha$ and $\rho^{(1)}_\alpha$ to $$N^{(1)}_{\alpha\beta} = N^{(1)}_\alpha \cap TS_\beta = TN_\alpha \cap TS_\beta = T(N_\alpha \cap S_\beta) = TN_{\alpha\beta},$$
	then observe that $\pi^{(1)}_{\alpha\beta} = d\pi_{\alpha\beta}$ and $\rho^{(1)}_{\alpha\beta} = \rho_{\alpha\beta} \circ p + (d\rho_{\alpha\beta})^2$. 
	
	We know that for any triple of indices $\alpha, \beta, \gamma \in I$, the following control conditions hold: 
	\begin{equation} 
		\pi_{\alpha\beta} \circ \pi_{\beta\gamma} = \pi_{\alpha\gamma} \label{control-1}\end{equation} 
	\begin{equation}
		\rho_{\alpha\beta} \circ \pi_{\beta\gamma} = \rho_{\alpha\gamma} \label{control-2}
	\end{equation} 
	Differentiating and using the chain rule on \ref{control-1} and \ref{control-2}, we have 
	\begin{equation}
		d\pi_{\alpha\beta} \circ d\pi_{\beta\gamma} = d\pi_{\alpha\gamma}\label{control-1'}
	\end{equation} 
	\begin{equation}d\rho_{\alpha\beta} \circ d\pi_{\beta\gamma} = d\rho_{\alpha\gamma}\label{control-2'}
	\end{equation} 
	whenever both sides of the equations are defined. Equation
	\ref{control-1'} implies that $$\pi^{(1)}_{\alpha\beta} \circ \pi^{(1)}_{\beta\gamma} = \pi^{(1)}_{\alpha\gamma},$$ hence $(TX, \Sigma^{(1)}, \mathcal{N}^{(1)})$ has $\pi$-control. 
	
From  Equation \ref{control-2'} we also obtain 
	\begin{equation}
		(d\rho_{\alpha\beta})^2 \circ d\pi_{\beta\gamma} = (d\rho_{\alpha\gamma})^2\label{rho-1}
	\end{equation} 
	Since $p \circ d\pi_{\beta\gamma} = \pi_{\beta\gamma}$, we see that Equation \ref{control-2} also implies 
	\begin{equation}(\rho_{\alpha\beta} \circ p) \circ d\pi_{\beta\gamma} = \rho_{\alpha\gamma} \circ p\label{rho-2}
	\end{equation} 
	Adding Equations \ref{rho-1} and \ref{rho-2}, we see that $\rho^{(1)}_{\alpha\beta} \circ \pi^{(1)}_{\beta\gamma} = \rho^{(1)}_{\alpha\beta}$. This is the desired $\rho$-control.
	
	Note that $\rho^{(1)}_\alpha(x, v) = 0$ if and only if $\rho_\alpha(x) = d\rho_\alpha(x, v) = 0$. Since $\rho_\alpha(x) = 0$, $x \in S_\alpha$. Next write $v = u + w \in T\mathbb{R}^N$ where $u$ is the orthogonal projection of $v$ to $T_x S_\alpha$ under a fiberwise inner product defined on $\nu(S_\alpha)$. Then 
	$$0 = d\rho_\alpha(x, v) = d\rho_\alpha(x, u) + d\rho_\alpha(x, w).$$
	But $d\rho_\alpha(x, u) = 0$ since $\rho_\alpha \equiv 0$ on $S_\alpha$. This forces $d\rho_\alpha(x, w) = 0$. Since $\rho_\alpha$ is the radial function on $\nu(S_\alpha)$, $d\rho_\alpha$ is strictly increasing in any direction orthogonal to $S_\alpha$, forcing $w = 0$. Hence $v \in T_xS_\alpha$. Therefore $(x, v) \in TS_\alpha$, i.e., $(\rho^{(1)}_\alpha)^{-1}(0) = TS_\alpha$. Finally it is straightforward to check that $(\pi_{\alpha\beta}^{(1)}, \rho_{\alpha\beta}^{(1)}) : N^{(1)}_{\alpha\beta} \to S_\beta \times (0, \infty)$ is a submersion. 
	
	We have checked that $\pi^{(1)}$ and $\rho^{(1)}$ are valid projection and radial functions, and $(TX, \Sigma^{(1)}, \mathcal{N}^{(1)})$ satisfies both the control conditions. The lemma follows.\end{proof}

\begin{lemma}\label{lem-cvf} (Weakly) controlled sections of the projection $p : TX \to X$ are (weakly) controlled vector fields on $X$  (cf.\ Definition \ref{def-strvf}).  \end{lemma}

\begin{proof} For concreteness, we prove the lemma for controlled sections and controlled vector fields. The same proof works for weakly controlled sections and weakly controlled vector fields. 
	Let $\eta : X \to TX$ be a controlled section of $p$. Then for any stratum $S \subset X$, $\eta|S : S \to TS$ is a $C^\infty$-section of the tangent bundle of $S$; let us denote this vector field as $\eta_S$. Then $\{\eta_S : S \in \Sigma\}$ is a stratified vector field on $X$. Note that since $\eta$ is $\pi$-controlled, it follows that for any pair of strata $S, L \subset X$ with $S < L$, we have $\pi^{(1)}_{SL}(\eta(x)) = \eta(\pi_{SL}(x))$. Hence $(\pi_{SL})_*(\eta_L(x)) = \eta_S(\pi_{SL}(x))$. Moreover $\eta$ is $\rho$-controlled, hence $\rho^{(1)}_{SL}(\eta(x)) = \rho_{SL}(x)$. Now $\rho^{(1)}_{SL}(\eta(x)) = \rho_{SL}(x) + d\rho_{SL}(\eta(x))^2$, therefore $d\rho_{SL}(\eta(x)) = 0$.  Equivalently, $(\eta_L)_*\rho_{SL} = 0$. This verifies that $\{\eta_S : S \in \Sigma\}$ is indeed a controlled vector field on $X$.\end{proof}

\begin{prop}\label{str-derv}Let $(X, \Sigma_X, \mathcal{N}_X)$ and $(Y, \Sigma_Y, \mathcal{N}_Y)$ be abstract stratified spaces and $f : X \to Y$ be a controlled (resp.\ weakly controlled) map. Then the stratum-wise differential induces a controlled (resp.\ weakly controlled)  map $df : TX \to TY$.\end{prop}

\begin{proof}Suppose $\alpha, \beta \in I$ is a pair of indices such that $\alpha < \beta$, and $S_\alpha, S_\beta \in \Sigma_X$ be the corresponding pair of strata of $X$. Let $L_\alpha, L_\beta \in \Sigma_Y$ be the unique strata of $Y$ such that $f(S_\alpha) \subset L_\alpha$ and $f(S_\beta) \subset L_\beta$. Let us denote the tube system on $X$ associated to the pair of strata $(S_\alpha, S_\beta)$ by $(N^X_{\alpha\beta}, \pi^X_{\alpha\beta}, \rho^X_{\alpha\beta})$ and similarly the tube system on $Y$ associated to the pair of strata $(L_\alpha, L_\beta)$ by $(N^Y_{\alpha\beta}, \pi^Y_{\alpha\beta}, \rho^Y_{\alpha\beta})$. As $f$ is a controlled mapping, we have 
	\begin{equation}f \circ \pi^X_{\alpha\beta} = \pi^Y_{\alpha\beta} \circ f\label{pi-control}\end{equation} 
	\begin{equation}\rho^X_{\alpha\beta} = \rho^Y_{\alpha\beta} \circ f \label{rho-control}\end{equation} 
	on $N^X_{\alpha\beta} \cap f^{-1}(N^Y_{\alpha\beta})$. Differentiating Equation \ref{pi-control}, it is immediate that $df \circ (\pi^X)^{(1)}_{\alpha\beta} = (\pi^Y)^{(1)}_{\alpha\beta} \circ df$. Differentiating Equation \ref{rho-control} we obtain
	\begin{equation}d\rho^X_{\alpha\beta} = d\rho^Y_{\alpha\beta} \circ df\label{rho-control'}\end{equation}
	Since $(p \circ df)(x, v) = f(x)$ for all $(x, v) \in TX$, we can rewrite Equation \ref{rho-control} as 
	\begin{equation}\rho^X_{\alpha\beta} \circ p = \rho^Y_{\alpha\beta} \circ p \circ df \label{rho-control''}\end{equation}
	Squaring both sides of Equation \ref{rho-control'} and adding to Equation \ref{rho-control''} gives $(\rho^X)^{(1)}_{\alpha\beta} = (\rho^Y)^{(1)}_{\alpha\beta} \circ df$. This verifies that $df : (TX, \Sigma_X^{(1)}, \mathcal{N}_X^{(1)}) \to (TY, \Sigma_Y^{(1)}, \mathcal{N}_Y^{(1)})$ is a controlled map. It is clear from the proof that if $f$ is only weakly controlled, $df$ is also weakly controlled.\end{proof}

\subsection{Stratified jets}\label{sec-stratfdjet}

Let $(X, \Sigma, \mathcal{N})$ be an abstract stratified space. Then the tangent bundle $(TX, \Sigma^{(1)}, \mathcal{N}^{(1)})$ is also an abstractly stratified space by Lemma \ref{lem-txstratfd}. We can now iterate this construction to define the {\bf $k$-fold iterated tangent bundle} $T^{(k)} X$. However, the radial functions become rather unwieldy to work with.  We redefine the control structure on $T^{(k)} X$ as follows:

\begin{defn}[Iterated tangent bundle]\label{def-itdtb} Let $(X, \Sigma, \mathcal{N})$ be an abstract stratified space and we choose a realization $X' \subset \mathbb{R}^N$ such that $(X', \Sigma', \mathcal{N}')$ is an abstract stratified set. Here, $\mathcal{N}'$ is induced from a tubular neighborhood system $(\nu(S), \pi_S, \rho_S)_{S \in \Sigma'}$ on the Whitney stratified set $X' \subset \mathbb{R}^N$. 
	
	Let $T^{(k)} X$ be the union $\bigcup_{S \in \Sigma'} T^{(k)} S \subset T^{(k)} \mathbb{R}^N$ of the $k$-fold iterated tangent bundles to each stratum of $X'$. Then $T^{(k)} X$
	 inherits a topology from $T^{(k)}\mathbb{R}^N = \mathbb{R}^{2^k N}$ and an $I$-decomposition $\Sigma^{(k)} = \{T^{(k)} S_\alpha : S_\alpha \in \Sigma'\}_{\alpha \in I}$. Let $$p^{(k)}: T^{(k)} X \to X$$
	be the projection obtained from restricting $T^{(k)}\mathbb R^N \to \mathbb{R}^N$ to $X'$ and composing with the inverse of the realization homeomorphism $X \to X'$. 
	
	Define $$N^{(k)}_\alpha := T^{(k)} \nu(S_\alpha) \cap T^{(k)} X$$ to be the tube around the stratum $T^{(k)} S_\alpha$ of $T^{(k)} X$. Let the restriction of the $k$-th derivative 
	$$d^{(k)} \pi_\alpha : T^{(k)}\nu(S_\alpha) \to T^{(k)}S_\alpha$$ to $N^{(k)}_\alpha$ be the associated tubular projection $\pi^{(k)}_\alpha := d^{(k)}\pi_\alpha$. Let $d^{(i)}\rho_\alpha : T^{(i)} \nu(S_\alpha) \to \mathbb{R}$ be the $i$-th derivative of the radial function $\rho_\alpha$. We define the  radial function associated to the stratum $T^{(k)} S_\alpha$ of $T^{(k)} X$ to be $$\rho^{(k)}_\alpha := \rho_\alpha \circ p^{(k)} + (d\rho_\alpha)^2 + \cdots + (d^{(k)}\rho_\alpha)^2$$ restricted to $T^{(k)}_\alpha$. This defines a tube system $\mathcal{N}^{(k)} = (T^{(k)}_\alpha, \pi^{(k)}_\alpha, \rho^{(k)}_\alpha)_{\alpha \in I}$ and $(T^{(k)} X, \Sigma^{(k)}, \mathcal{N}^{(k)})$ is abstractly stratified. \end{defn}

\begin{defn}\label{def-contEvaldjet}   Let $(E, X, p)$ be a stratified fiber bundle and $U \subset X$ be an open subset with the canonical abstract stratification inherited from $X$. We define a {\bf formal (weakly) controlled $E$-valued $r$-jet} over $U$ to be an $(r+1)$-tuple $(s_0, s_1, \cdots, s_r)$ such that $s_k : T^{(k)} U \to T^{(k)} E$ is a  (weakly)  controlled section of $(T^{(k)} E, T^{(k)} X, d^{(k)} p)$ over $T^{(k)} U \subset T^{(k)} X$ for all $0 \leq k \leq r$ and each $s_{k+1}$ {\bf covers} $s_k$ in the sense that the following diagram commutes: 
	$$\begin{array}{ccccccccc}
		E & \xla{} & TE & \xla{} & T^{(2)}E & \xla{} & \cdots & \xla{} & T^{(r)} E \\
		\;\;\xua{s_0} &  & \xua{s_1} &  & \xua{s_2} &  & \cdots &  & \xua{s_r} \\
		U & \xla{} & TU & \xla{} & T^{(2)}U & \xla{} & \cdots & \xla{} & T^{(r)} U \\
	\end{array}$$\end{defn}

An extended and modified notion of stratified jets will be given later in Definition \ref{def-sjr} before we prove the $h-$principle for jet sheaves.
The {\bf sheaf of formal  (weakly)  controlled $E$-valued $r$-jets}, denoted as $\mathcal{J}^r_E$ ($\JJ^r_{E,w})$),  assigns to each open subset $U \subset X$ the set $\mathcal{J}^r_E(U)$ ($\JJ^r_{E,w}(U))$) of all formal (weakly) controlled $E$-valued $r$-jets over $U$. Let $\Gamma_E$ ($\Gamma_{E,w}$) denote the sheaf of  (weakly)  controlled local sections of $(E, X, p)$. Then there is similarly a morphism of sheaves 

\begin{equation}\label{eqn-jr}
	J^r :  \Gamma_E \to \mathcal{J}^r_E  
\end{equation}
\begin{equation*}
	\big(J^r_w : \Gamma_{E,w} \to \mathcal{J}^r_{E,w}\big)
\end{equation*}
which sends a (weakly) controlled local section $s \in \Gamma_E(U)$ of $(E, X, p)$ over an open subset $U \subset X$ to the formal (weakly) controlled $E$-valued $r$-jet $$J^r s:= (s, ds, d^{(2)}s, \cdots, d^{(r)} s) \in \mathcal{J}_E^r(U)$$
$$\big(J^r_w s:= (s, ds, d^{(2)}s, \cdots, d^{(r)} s) \in \mathcal{J}^r_{E,w}(U)\big).$$ The image of this morphism is a subsheaf 
$\mathcal{H}^r_{E}$ of $\mathcal{J}^r_{E}$ 
($\mathcal{H}^r_{E,w}$ of $\mathcal{J}^r_{E,w}$) which we shall call the {\bf (weakly)  controlled sheaf of holonomic $E$-valued $r$-jets on $X$}.

Let $T^{(k)} E$ be equipped with metrics $\mathrm{dist}^{(k)}_E$ respecting the topology for all $0 \leq k \leq r$. Then for any open subset $U \subset X$ we can equip $\mathcal{J}^r_E(U)$ (or $\mathcal{J}^r_{E,w}(U)$) with a metric topology: we shall call two (weakly) controlled $E$-valued $r$-jets $J = (s_0, s_1, \cdots, s_r)$ and $J' = (s_0', s_1', \cdots, s_r')$ over $U$ are {\bf $\varepsilon$-close} if 
$$\sup_{x \in U} \mathrm{dist}^{(k)}_E(s_k(x), s_k'(x)) < \varepsilon \; \text{for all} \; 0 \leq k \leq r$$

\section{Flexibility, Diff-invariance}\label{sec-sheafflexdiff} Two crucial notions that come into play in Gromov's sheaf-theoretic $h-$principle \cite[Section 2.2]{Gromov_PDR} over manifolds are
\begin{enumerate}
	\item Flexibility,
	\item $\mathrm{Diff}$-invariance.
\end{enumerate}
The purpose of this section is to extend these two notions  both at the level of the base space as well as 
that of the nature of the sheaf. Thus, we shall
\begin{enumerate}
	\item replace the base manifold by a stratified space,
	\item replace the sheaf of quasi-topological spaces in \cite{Gromov_PDR} by stratified sheaves,
\end{enumerate}
and extend Gromov's notions of flexibility and $\mathrm{Diff}$-invariance to this setup.

\subsection{Flexibility of  sheaves}\label{sec-flexsimp} 
Following Gromov \cite[Ch. 2]{Gromov_PDR}, we shall refer to sheaves of {quasitopological spaces}\footnote{{A quasitopological space consists of a set $A$ and for every topological space $X$ a subset $A(X)$ of the collection of all set-theoretic maps $X \to A$ such that: 
		\begin{itemize}
	\item  if $f : X \to X'$, $g \in A(X')$, then $g \circ f \in A(X)$, 
		\item  for an open cover $\{U_i\}$ of $X$, $f \in A(X)$ if $f|U_i \in A(U_i)$, and 
			\item  if $X = X_1 \cup X_2$, where $X_1, X_2$ are closed sets, $f \in A(X)$ if $f|X_i \in A(X_i)$.
		\end{itemize} We call maps $f : X \to A$ for which $f \in A(X)$ as ``continuous". A set-theoretic map $f : A \to A'$ between quasitopological spaces is said to be continuous if for any $g \in A(X)$, $f \circ g \in A'(X)$.}} as \emph{continuous sheaves}. We collect together in this subsection, some basic notions from \cite[Ch. 2]{Gromov_PDR}, and facts about continuous sheaves.

\begin{defn}\label{def-micfibn}\cite[p. 40]{Gromov_PDR}
Let $\alpha:A \to A'$ be a continuous map of quasitopological spaces. Consider a continuous map $\phi: P \to A$ of a compact polyhedron $P$ into $A$. Let $\phi'=\alpha \circ \phi$. Let $\Phi': P \times [0,1] \to A'$ be such that  
$\Phi'| P \times \{0\}=\phi'$.

The map $\alpha$ is called a \emph{{(Serre)} fibration} if {for all such polyhedra $P$, maps $\phi: P \to A$ and homotopies $\Phi'$ of $\phi'$}, $\Phi'$ lifts to a map $\Phi: P \times [0,1] \to A$ such that $\Phi| P \times \{0\}= \phi$ and $\alpha \circ \Phi = \Phi'$.

{The map $\alpha$ is called a} \emph{{(Serre)} microfibration} if {for all such polyhedra $P$, maps $\phi: P \to A$ and homotopies $\Phi'$ of $\phi'$}, there exists ${0 < \ep \leq 1}$ (where $\ep$ may depend on $P, \phi, \Phi'$) and a map $\Phi : P \times [0,\ep] \to A$, such that $\Phi|P \times \{0\} = \phi$ and $\alpha \circ \Phi = \Phi'|P \times [0,\ep]$.
\end{defn}							

Henceforth, by fibration (resp.\ microfibration), we shall mean a Serre fibration (resp.\ microfibration) of quasitopological spaces. {Given a continuous sheaf $\mathcal{F}$ on a space $X$ and a subset $K \subset X$, we shall denote 
$\mathcal{F}(K) := \lim_{U \supset K} \mathcal{F}(U)$.}

\begin{defn}\label{def-kan}
	Let $X$ be locally compact Hausdorff.
	A continuous sheaf $\FF$ on $X$ is 
	\emph{flexible} (resp.\ \emph{microflexible}) if for all {pairs of compact subsets 
	$K \subset K'$ of X}, $\FF(K') \to \FF(K)$ is a  fibration (resp. microfibration).
\end{defn}

\begin{eg}\label{lem-surj}
	Let $f: Y \to X$ be surjective, and $\FF$ be the continuous sheaf of sections associated to $f$, {equipped with the quasitopology on mapping spaces}. Then $\FF$ is flexible.
\end{eg}

\begin{theorem}\label{thm-gromov-whe2he}\cite[Theorem B, p. 77]{Gromov_PDR} Let ${\Phi : \FF \to \GG}$ be a morphism of flexible  sheaves over a finite dimensional locally compact Hausdorff space $X$. Then $\Phi$ is a local weak homotopy equivalence if and only if $\Phi$ is a  weak homotopy equivalence, i.e.\ $\Phi_x: \FF_x \to \GG_x$ is a weak homotopy equivalence for all $x \in X$ if and only if ${\Phi_U : \FF(U) \to \GG(U)}$ is a weak homotopy equivalence for all $U \subset X$ open.
\end{theorem}

\subsection{Stratified spaces and flexibility conditions}\label{sec-flexdefs}
Let $(X, \Sigma)$ be a Whitney stratified space in the sense of Definition \ref{def-whitneyss}. Then, by the Whitney conditions (a), (b) of Definition \ref{def-whitneyab} and Thom's isotopy {lemma} \cite[Section 1.5]{GM_SMT} (see also Lemma \ref{lem-strbdl-trivialization}) we have:

\begin{lemma}\label{lem-nbhd}
	Let $x \in X$ and let $S$ denote the unique stratum in which $x$ lies. Then there exists an open neighborhood $U$ of $x$ and a stratum-preserving homeomorphism $\phi: U \to \R^i \times cA$, where 
	\begin{enumerate}
		\item $S$ has dimension $i$
		\item $A$ is a compact stratified space admitting a stratum-preserving homeomorphism with the link of $S$ 
		in $X$.
	\end{enumerate}
\end{lemma}

We shall now define a notion of sheaves over stratified spaces $(X, \Sigma)$. This is a finer notion than that of a sheaf over the underlying topological space $X$. It associates data to open subsets of each stratum-closure of $X$. To formulate this, we introduce the \emph{stratified site} associated to $(X, \Sigma)$:

\begin{defn}\label{def-stratfdsite}
A stratified space $(X,\Sigma)$ comes equipped with the canonical filtered collection of topological spaces $\{\bbar{S}\}$, where
	\begin{enumerate}
		\item $S$ is a stratum of $(X,\Sigma)$.
		\item  $\bbar{S}$ is equipped with the subspace topology inherited from $X$.
	\end{enumerate}
The \emph{stratified site} $\str(X, \Sigma)$ is {the category where}
	\begin{enumerate}
		\item Objects of $\str(X, \Sigma)$ are open subsets $U \subset \bbar{S}$ of some stratum-closure. 
		\item Morphisms of $\str(X, \Sigma)$ are inclusions $U \hookrightarrow V$ between such subsets.
	\end{enumerate}
\end{defn}

\begin{rmk}\label{rmk-site}
The term stratified site in Definition \ref{def-stratfdsite} is borrowed from algebraic geometry. For example, it can be checked that $\str(X,\Sigma)$ comes equipped with a natural Grothendieck topology (namely, sieves in $\str(X, \Sigma)$ are covers consisting of objects in $\str(X, \Sigma)$), and forms an example of a site.
\end{rmk}

\begin{defn}\label{def-sssheaf}
{A \emph{stratified continuous sheaf} $\FF$ on $(X, \Sigma)$ is a quasitopological space-valued sheaf on the site $\str(X, \Sigma)$ (Definition \ref{def-stratfdsite}), i.e., it assigns a quasitopological space $\FF(U)$ to every object $U$ of $\str(X,\Sigma)$ in a way that the gluing axiom is satisfied.
	In particular, restricting to open subsets in the stratum closures, it consists of a collection of sheaves $\{\FF_{\bbar{L}}\}$, one for every stratum $L$ of $X$, such that for every pair $S<L$, there is a morphism of sheaves over $\bbar{S}$
$$\res^L_S: i_{\bbar{S}}^*\FF_{\bbar{L}}\to \FF_{\bbar{S}}$$
{that we call the \emph{restriction map from $L$ to $S$}}. (Note that 
$\FF_{\bbar{L}}$ is a sheaf over ${\bbar{L}}$.) }
\end{defn}

{Open subsets of $X$ contained entirely in the closure of a top-dimensional stratum are certainly elements of $\str(X,\Sigma)$. An arbitrary open subset $U$ of $X$, on the other hand, can be written as a union of finitely many elements of $\str(X, \Sigma)$. This becomes relevant when $X$ has more than one top-dimensional stratum. The value of the sheaf on such a $U$ can be defined using fiber products, see Section \ref{sec-glueacross}.
  We shall denote by $\FF_S$ the restriction of the sheaf $\FF_{\bbar{S}}$ over 
  ${\bbar{S}}$
  to the open stratum $S$}. 

\begin{eg}\label{eg-stratbdlstratsheaf}\rm{
{
Recall (Definition \ref{def-contEvaldjet} and the discussion in Section \ref{sec-stratfdjet}) that for any stratified fiber bundle $(E, X, p)$, there are natural sheaves $\Gamma_E, \Gamma_{E, w}$ over $X$ consisting of controlled (or weakly controlled) sections of $E$. There are also sheaves $\JJ^r_E, \HH^r_E, \JJ^r_{E,w},  \HH^r_{E,w}$ over $X$ consisting of controlled (or weakly controlled) formal and holonomic $E$-valued $r$-jets.}

{Let $S$ be a stratum of $X$. Let $E_{\bbar{S}} = p^{-1}(\bbar{S})$. We shall refer to $E_{\bbar{S}}$ as the \emph{restriction} of $E$ to  $\bbar{S}$. Also, let $p_{\bbar{S}}: E_{\bbar{S}}\to {\bbar{S}}$ denote the \emph{restriction} of $p$ to 	$E_{\bbar{S}}$.		
Then $(E_{\bbar{S}}, \bbar{S}, p_{\bbar{S}})$ is a stratified bundle as well. Each member of
the family $\{(E_{\bbar{S}}, \bbar{S}, p_{\bbar{S}}): S \in \Sigma\}$ has naturally 
associated sheaves over $\bbar{S}$:
$$\Gamma_{E, \bbar{S}}, \Gamma_{E, w, \bbar{S}}, \JJ^r_{E,{\bbar{S}}}, \HH^r_{E,{\bbar{S}}}, \JJ^r_{E,w,{\bbar{S}}},  \HH^r_{E,w,{\bbar{S}}}.$$
The family $\{\Gamma_{E, \bbar{S}} : S \in \Sigma\}$ is an example of a stratified continuous sheaf over $X$. For every pair of strata $S < L$, the restriction map
$$\mathrm{res}^L_S : i_{\bbar{S}}^* \Gamma_{E, \bbar{L}} \to \Gamma_{E, \bbar{S}}$$
is given at the level of stalks over a point $x \in \bbar{S}$ as follows. For a germ of a controlled section $\sigma$ of $(E_{\bbar{L}}, \bbar{L}, p_{\bbar{L}})$ defined over a neighborhood $U \subset \bbar{L}$ of $x$, we define $\mathrm{res}^L_S(\sigma)$ to be the restriction of $\sigma$ to the subset $U \cap \bbar{S} \subset \bbar{S}$.  
Similarly, each of the collections
\begin{gather*}\{\Gamma_{E,w,{\bbar{S}}} : S \in \Sigma\}, \{\JJ^r_{E,{\bbar{S}}} : S \in \Sigma\}, \{\HH^r_{E,{\bbar{S}}} : S \in \Sigma\}, \\
\{\JJ^r_{E,w,{\bbar{S}}} : S \in \Sigma\}, \{\HH^r_{E,w,{\bbar{S}}}: S \in \Sigma\}\end{gather*}
constitutes a stratified continuous sheaf over $X$.
}
}
\end{eg}

\begin{rmk}\label{rmk-stratshcnstrble}
\rm{
{The notion of stratified continuous sheaves on a stratified space $(X, \Sigma)$ can be considered as a ``non-abelian" analogue of constructible sheaves in algebraic geometry: a constructible sheaf on $(X, \Sigma)$ is a sheaf of $k$--modules on $X$ which is locally constant restricted to each stratum $S \in \Sigma$ (here $k$ is a commutative ring). Therefore, stratified local systems give rise to constructible sheaves, whereas stratified fiber bundles give rise to stratified continuous sheaves.}

{An important basic example of a constructible sheaf is as follows: Let $\pi : X \to \Bbb{CP}^1$ be a Lefschetz fibration. Then, the derived pushforward of the constant sheaf $\mathbf{R}\pi_*\underline{k}_X$ defines a constructible sheaf on the stratification $\Sigma = \{S, L\}$ of $\Bbb{CP}^1$, where $S$ is the set of singular values of $\pi$, and $L = \Bbb{CP}^1 \setminus S$. Homological information pertaining to the map $\pi$ (e.g., vanishing cycles of $\pi$) are recorded by the constructible sheaf $\mathbf{R}\pi_*{\underline{k}}_X$ (e.g., the monodromy of the local system $i_L^*\mathbf{R}\pi_*{\underline{k}}_X$).}

{We can likewise construct a stratified continuous sheaf on $\Bbb{CP}^1$ associated to $\pi$ as in Example \ref{eg-stratbdlstratsheaf}, as $\pi$ is a stratumwise fiber bundle (Definition \ref{def-stratumwisebdl}). This records \emph{homotopical} information pertaining to the map $\pi$ (e.g.\ whether sections of $\pi$ solving certain differential relations satisfy the $h-$principle or not).}
}
\end{rmk}

\begin{defn}\label{def-stratflex}
	A stratified  sheaf on $(X,\Sigma)$ is \emph{flexible} (resp. \emph{microflexible}) if 
	for any stratum $S$, {the continuous sheaf $\FF_{\bbar{S}}$ on $\bbar{S}$} is 
	flexible (resp.\ microflexible).

	A stratified  sheaf on $(X,\Sigma)$ is \emph{stratumwise flexible} (resp. \emph{stratumwise microflexible}) if 
	for any stratum $S$, {the continuous sheaf $\FF_S$ on $S$} is 
flexible (resp.\ microflexible).
\end{defn}

Note that the latter is a condition on the sheaves $\{\FF_{\bbar{S}}\}$ comprising the stratified sheaf $\FF$ \emph{after restricting each $\FF_{\bbar{S}}$ to the open stratum $S$.}

\begin{eg}\label{eg-stratflexnotflex}\rm{
{Let $X, Y$ be fixed topological spaces and $f : X \to Y$ be a fixed map which is not a fibration. On the interval $[0, \infty)$, we define a presheaf $\FF$ as follows. Set $\FF(U) = X$ whenever $U \subset [0, \infty)$ is an open subset containing $0$, and $\FF(U) = Y$ otherwise. For a pair of open subsets $U \subset V \subset [0, \infty)$, we define the restriction map $\FF(V) \to \FF(U)$ on a case-by-case basis:
\begin{enumerate}
\item Suppose $V$ does not contain $0$, so neither does $U$. Then, we define the restriction map $\FF(V) \to \FF(U)$ to be the identity map $\mathrm{id}_Y : Y \to Y$.
\item Suppose $V$ contains $0$ but $U$ does not. Then, we define the restriction map $\FF(V) \to \FF(U)$ to be $f : X \to Y$.
\item Suppose $V, U$ both contain $0$. Then, we define the restriction map $\FF(V) \to \FF(U)$ to be the identity map $\mathrm{id}_X : X \to X$.
\end{enumerate}
We obtain a continuous sheaf on $[0, \infty)$ by sheafifying $\FF$. Let us continue to denote this sheaf as $\FF$, by a slight abuse of notation.}

{Let us stratify $[0, \infty)$ as $\Sigma = \{S, L\}$ where $S = \{0\}$ and $L = (0, \infty)$. We may turn $\FF$ into a stratified sheaf by defining $\FF_{\bbar{L}} = \FF$ and $\FF_{\bbar{S}} = i_{\bbar{S}}^*\FF$. Then $\FF_L = i_L^*\FF_{\bbar{L}}$ is the constant sheaf on $L$ taking values in $Y$, and $\FF_S = i_S^*\FF_{\bbar{S}} = \FF_{\bbar{S}}$ is the constant sheaf on $S$ taking values in $X$. These are both flexible sheaves.}

{However, as $f : X \to Y$ is not a fibration, $\FF_{\bbar{L}} = \FF$ is not a flexible sheaf. This gives an example of a stratified sheaf $\FF = \{\FF_{\bbar{S}}, \FF_{\bbar{L}}\}$ which is stratumwise flexible, but not flexible.}
}
\end{eg}

We recall a construction from \cite[p. 77]{Gromov_PDR}.  Let $\mathcal{F}, \mathcal{G}$ be continuous sheaves on $X$ and $q : \mathcal{F} \to \mathcal{G}$ be a morphism of continuous sheaves. Consider the continuous sheaf $\widetilde{\mathcal{F}}$ defined by assigning to every open set $U \subset X$ the set
$$\widetilde{\mathcal{F}}(U) := \{(s, \gamma) \in \mathcal{F}(U) \times \mathrm{Maps}(I, \mathcal{G}(U)) :  q(s) = \gamma(0)\}.$$
Equip  $\widetilde{\mathcal{F}}(U)$ with a quasitopology  as follows: for any topological space $W$, a map $W \to \widetilde{\mathcal{F}}(U)$ is continuous if and only if the projections $W \to \mathcal{F}(U)$ and $W \to \mathrm{Maps}(I, \mathcal{G}(U))$ are continuous. There is a morphism of continuous sheaves $\widetilde{q} : \widetilde{\mathcal{F}} \to \mathcal{G}$ given by $\widetilde{q}(s, \gamma) = \gamma(1)$. Then  $$\til q:\widetilde{\mathcal{F}}(U) \to \mathcal{G}(U)$$  is a fibration. Let $\psi \in \mathcal{G}(X)$ be a global section.

\begin{defn}\label{def-hofib}  We shall call the fiber of $\widetilde{q}$ over $\psi$ the \emph{homotopy fiber} of $q$ over $\psi$:
	$$\mathrm{hofib}(q; \psi)(U) := \widetilde{q}^{-1}(\psi|_U) \subset \widetilde{\mathcal{F}}(U).$$
{If the choice of $\psi \in \mathcal{G}(X)$ is understood, we simply denote the homotopy fiber sheaf as $\mathrm{hofib}(q)$.} \end{defn}

Let $\mathcal{F}$ be a stratified continuous sheaf on $(X, \Sigma)$. For ease of exposition, we assume the existence of and  fix a global section $\psi \in \mathcal{F}(X)$.

\begin{defn}\label{def-hh} For $S<L$, 
	define the \emph{closed homotopy fiber sheaf} of $\FF$ from $L$ to $S$ by $$\bH^L_S = \hofib(\mathrm{res}^L_S : i_{\bbar{S}}^*\FF_{\bbar{L}} \to \FF_{\bbar{S}}).$$
	The corresponding  \emph{open homotopy fiber sheaf} is defined to be $\HH^L_S =i_S^*\bH^L_S $
\end{defn}

Note that  $$\HH^L_S = \hofib(i_{{S}}^*\FF_{\bbar{L}} \to \FF_{{S}}).$$

\begin{defn}\label{def-infstratflex}
	A stratified continuous sheaf on $(X,\Sigma)$ is \emph{infinitesimally flexible across strata}  if 
	for any  $S<L$ in $\Sigma$,  $\HH^L_S$ is  flexible.
	
\end{defn}

\begin{lemma}\label{lem-mapsI2F}
	Let $\FF$ be a  continuous sheaf over $X$ and  $Z\subset K \subset X$ be compact subsets. 
	\begin{enumerate}
		\item If $\FF(K) \to \FF(Z)$
		is a  fibration, then so is $\maps(I^n, \FF(K)) \to \maps(I^n, \FF(Z))$.
		\item If $\FF$ is a stratumwise flexible stratified sheaf, so is $\maps(I^n,\FF)$.
		%	\item If $\FF$ is a stratumwise locally trivial stratified sheaf, so is $\maps(I^n,\FF)$.
		\item If $\FF$ is a stratified sheaf which is infinitesimally flexible across strata, so is $\maps(I^n,\FF)$.
	\end{enumerate}
\end{lemma}

%\mcomment
\begin{proof}\leavevmode
\begin{enumerate}
\item We are given that $\mathcal{F}(K) \to \mathcal{F}(Z)$ is a  fibration. Let $C$ be a CW-complex; then every map $C \times I^n \times I \to \mathcal{F}(Z)$ with a chosen initial lift $C \times I^n \times \{0\} \to \mathcal{F}(K)$ admits a lift. Thus, $\maps(I^n, \mathcal{F}(K)) \to \maps(I^n, \mathcal{F}(Z))$ satisfies the homotopy lifting property with respect to homotopies of maps from $C$. As $C$ was arbitrary, this proves $\maps(I^n, \mathcal{F}(K)) \to \maps(I^n, \mathcal{F}(Z))$ satisfies the homotopy lifting property and thus is a Serre fibration.
\item This is an immediate corollary of $(1)$.
\item Infinitesimal flexibility across strata is the statement that for all $S, L \in \Sigma_X$, $S < L$, $\HH^L_S$ is flexible. Applying $\maps(I^n, -)$  and using $\iota_S^* \maps(I^n, \mathcal{F}_{\bbar L}) = \maps(I^n, \iota_S^* \mathcal{F}_{\bbar L})$, we obtain the desired claim.\qedhere
\end{enumerate}
\end{proof}

\subsection{Diff-invariance}\label{sec-diffinv} We shall say that  $U \subset S$ is a \emph{relatively compact embedded open ball}, if
$U$ is an open ball and
$\bbar{U} \subset S$ is a compact (smoothly) embedded  ball in $S$.
Following Gromov \cite{Gromov_PDR}, we shall say that a sheaf $\FF$ over a manifold $V$ is \emph{$\diffc$-invariant} if {it is acted on by the pseudogroup of compactly supported diffeomorphisms of $V$} in the following sense: for every pair of relatively compact open balls $U, U' \subset V$ (i.e., $\bbar{U}, \bbar{U'}$ are embedded compact balls in $V$), and a diffeomorphism $\phi: U' \to U$, {there is an isomorphism of sheaves $\psi : \phi^*(\FF\vert_U) \to \FF\vert_{U'}$ such  that $\psi$ is  functorial in $\phi, U, U'$}. Recall that the pseudogroup $\mathrm{Diff}_c(V)$ of diffeomorphisms is the set of all pairs $(U,f)$, where $U \subset V$ is an open set and {$f$ is a compactly supported diffeomorphism of $M$ carrying $U$ onto another open set $U' = f(U) \subset V$}.

Finally, if $\FF^1, \FF^2$ are $\diffc$-invariant sheaves over a manifold $V$, then a morphism of sheaves $\Phi: \FF^1
\to \FF^2$ is said to be  \emph{$\diffc$-invariant} if it is natural with respect to the $\diffc(V)$-action, i.e.\ if
$\psi_i: \phi^*(\FF^i\vert_U)\to\FF^i\vert_{U'}$, for $i=1,2$, denote the isomorphisms above, then the following diagram commutes:
$$
	\begin{CD}
		\phi^*(\FF^1\vert_U)@>{\psi_1}>>\FF^1\vert_{U'} \\
		@V{\Phi}VV @V{\Phi}VV\\
		\phi^*(\FF^2\vert_U)	@>{\psi_2}>>\FF^2\vert_{U'}
	\end{CD}
$$

\begin{defn}\label{def-stratdiff}
	A stratified continuous sheaf on a stratified space $(X,\Sigma)$ is \emph{$\sdiff$-invariant}  if 
	\begin{enumerate}
		\item for any  $S<L$ in $\Sigma$,  $i_S^*\FF_{\bbar{L}}$ is  $\diffc(S)-$invariant.
		\item for any  $S\in\Sigma$, $\FF_{S}$ is  $\diffc(S)-$invariant.
		\item for any  $S<L$ in $\Sigma$, $\res^L_S$ is  $\diffc(S)-$invariant.
	\end{enumerate}
\end{defn}

We observe the following.

\begin{lemma}\label{lem-diffinvconststalk}
	Let $\FF$ be a $\diffc-$invariant sheaf over a connected manifold $M$. Then $\FF$ has constant stalks, {i.e.\ for any pair of points $x, y \in M$, $\FF_x \cong \FF_y$.}
\end{lemma}

\begin{proof}
	Let $x, y \in M$. Let $\{U_i :i \in \natls\}$ be a family of nested open balls around $x$ such that
	$\cap_i U_i = \{x\}$. There exists a homeomorphism $\phi: U_1 \to V_1$ such that $V_1$ is a neighborhood of $y$, and $\phi(x) = y$. Let $V_i = \phi(U_i)$. Then  $\{V_i :i \in \natls\}$ is a family of nested open balls around $y$ such that
	$\cap_i V_i = \{y\}$. By $\diffc-$invariance, {$\FF\vert_{U_i} \cong \phi^* \FF\vert_{V_i}$}, and hence (by passing to limits), $\FF_x \cong \FF_y$.
\end{proof}

Let $\FF$ be a stratified (continuous) sheaf over a stratified space $(X,\Sigma)$ such that
\begin{enumerate}
	\item  $\FF$  is infinitesimally flexible across strata.
	\item $\FF$ is $\sdiff-$invariant.
\end{enumerate}

Recall that 
for any  $S<L$ in $\Sigma$,  $\res^L_S : i_S^*\FF_{\bbar{L}}  \to \FF_{S}$ is  a   morphism of sheaves, and
 $\HH^L_S=\hofib (\res^L_S)$ denotes the homotopy fiber sheaf.

\begin{lemma}\label{lem-trivialhofibbdl} $\HH^L_S$ is  $\diffc(S)-$invariant. In particular,
	$\HH^L_S$ has constant stalks over $S$. 
\end{lemma}

\begin{proof}

	By $\sdiff-$invariance (Definition \ref{def-stratdiff})
	of $\FF$, it follows that 
	
	\begin{enumerate}
		\item for any  $S<L$ in $\Sigma$,  $i_S^*\FF_{\bbar{L}}$ is  $\diffc(S)-$invariant.
		\item for any  $S\in\Sigma$, $\FF_{S}$ is  $\diffc(S)-$invariant.
		\item for any  $S<L$ in $\Sigma$, $\res^L_S$ is  $\diffc(S)-$invariant.
	\end{enumerate}
	By functoriality of the homotopy fiber construction, $\HH^L_S$ is $\diffc(S)-$invariant.
	Hence, by Lemma \ref{lem-diffinvconststalk}, $\HH^L_S$ has constant stalks.
\end{proof}

\section{The sheaf-theoretic $h$-principle}\label{sec-hprin}

\subsection{The (Gromov) diagonal normal sheaf}\label{sec-gromovformalfn}
Let $\FF$ be  a continuous sheaf  over a locally compact countable polyhedron $X$ (e.g.\ a stratified space). Define a sheaf $\PP$ over $X \times X$ by {assigning to every basic open set $U \times V \subset X \times X$, the quasitopological space}
$$\PP(U \times V ):= {\mathrm{Maps}(U, \FF(V))},$$ 

\begin{defn}\label{def-formalfn}
The (Gromov) diagonal normal sheaf $\FF^\bullet$ associated to $\FF$ is defined by $$\FF^\bullet = \diag^* \PP,$$ where $\diag: X \to X \times X$ is the diagonal embedding.
\end{defn}

When $\FF$ {is a} subsheaf of the sheaf of sections of a surjective map $P: E \to X$ between topological spaces, an alternate description of the (Gromov) diagonal normal sheaf may be given in terms of (a slight relaxation of) Milnor's construction of microbundles \cite{milnor-microb,kister-mic}. 

\begin{defn}\label{def-mic} Let $X$ be a topological space.
	The \emph{tangent microbundle} $(U_X,X,p)$ to $ X$ is defined to be the germ of a neighborhood $U_X$ of $\diag(X) \subset  X \times X$ along with the  projection $p : U_X \to X$ to the first coordinate.
\end{defn}

\begin{rmk} In Milnor's definition \cite{milnor-microb}, a microbundle is required to always be locally trivial, whereas we relax this condition.\end{rmk}

Let $P: E \to X$ be surjective and $\Gamma (U, E)$ denote the space of sections over $U \subset X$ equipped with the compact open topology. Let $\FF$ denote a subsheaf of the sheaf of sections {$\Gamma(-, E)$} satisfying some property $\AAA$, i.e.\ $\FF(U)$ consists of sections $s \in \Gamma (U, E)$ satisfying the property $\AAA$.

Then $\PP(U \times V)$ consists of continuous maps from $U$ to {$\FF(V) \subset \Gamma(V, E)$} (where the latter has the inherited compact open topology). The following are two equivalent descriptions:
\begin{enumerate}
	\item  $\PP(U \times V )$ consists of $U-$parametrized families of sections over $V$ {satisfying property $\mathcal{A}$}
	\item $\PP(U \times V )$ consists of continuous maps from $U \times V$ to $E$ such that for each $x\in U$, {it restricts on $\{x\} \times V$ to a section $\sigma_x :V \to E$ satisfying property $\mathcal{A}$}
\end{enumerate}

It is more convenient to think of $\FF^\bullet$ as the restriction of $\PP$ to the {tangent microbundle} $(U_X,X,p)$ (Definition \ref{def-mic}) rather than as the restriction to $\diag (X)$. This is because $(U_X,X,p)$ is defined as a germ of open neighborhoods of ${\diag (X) \subset X \times X}$, and restriction of the sheaf $\PP$ to any representative of $U_X$ makes sense \emph{without passing to limits}. We proceed to describe this in some more detail.

Let $\{U_i \times U_i\}$ be a collection of basic open sets in $X \times X$. We shall say that a collection of elements $\phi_i \in \PP(U_i \times U_i)$ are \emph{consistent}, if for all $i \neq j$, 
$$\phi_i=\phi_j \text{ on } {(U_i \cap U_j) \times (U_i \cap U_j)}$$
Let $W \subset \diag (X) \subset X \times X$ be an open subset. Then any element of $\FF^\bullet(W)$ is represented by a family of consistent 
elements $\phi_i \in \PP(U_i \times U_i)$, where $\{U_i \times U_i\}$ covers $W$. {Using the equivalence described in the preceding paragraphs, we may treat $\phi_i$ as maps $\phi_i : U_i \times U_i \to E$. The consistency condition therefore allows us to glue these to a well-defined map $\phi : (W \times W) \cap U_X \to E$. We list  the properties of this map as a characterization of sections of $\FF^\bullet$ over $W$:}

\begin{rmk}\label{rmk-formalasgerms}\rm{$\phi: (W\times W) \cap U_X \to E$ is an element of $\FF^\bullet(W)$ if and only if
	\begin{enumerate}
		\item The restriction ${\phi\vert{\diag(W)}}$ of $\phi$ to $\diag(W)$ is a section of $E$ over $p(\mathrm{diag}(W)) = W \subset X$, and
		\item {For any $w \in W$,} the restriction of $\phi$ to $(\{w\}\times W)\cap U_X$ is the germ of a section of $E$ over the open subset $p_2((\{w\}\times W)\cap U_X) \subset W \subset X$ of $w$ {where $p_2: X \times X \to X$ defines the projection of $X \times X$ to the second coordinate.}
	\end{enumerate}

Thus, the first coordinate in $X \times X$ defines the base space of the tangent microbundle $(U_X,X,p)$, and the second gives  germs of neighborhoods of points $x \in X$. Hence, 
	an element of $\FF^\bullet(W)$ is given by a section of $E$ over 
	$W$ (in the first coordinate)  decorated with germs of sections $\{s_w: w \in W\}$ ({in} the second coordinate).
	
A caveat is in order. {The preceding paragraph suggests that elements of $\FF^\bullet(X)$ correspond to maps $U_X \to E$ from the total space of the tangent microbundle $(U_X, X, p)$ to the total space of the surjective map $(E, X, P)$ whose sections define the sheaf $\FF$. However, this map is {\bf not} a fiber-preserving map; in fact, the situation is completely orthogonal. This is because the fibers of the tangent microbundle $p : U_X \to X$ are subspaces of the second factor in the square $X \times X$, which map to germs of \emph{sections} of the surjection $P : E \to X$, and are therefore ``transverse" to the fibers of $P$.}}
\end{rmk}

\begin{defn}\label{defn-sheafmaps2ff}
	Let $\FF$ denote a {continuous} sheaf $X$. Let $W$ be a fixed topological space.
	Define a new sheaf 
	$\maps(W,\FF)$ over $X$ as follows. For any $U \subset X$ open, set $$\maps(W,\FF) (U) = \maps(W,\FF(U)), $$ where $ \maps(W,\FF(U))$ is equipped with {the natural quasitopology on mapping spaces\footnote{{For a topological space $X$ and a quasitopological space $A$, the natural quasitopology on $\maps(X, A)$ is defined by setting $\maps(X, A)(Z) = A(X \times Z)$ for any topological space $Z$.}}.}
\end{defn}

\begin{lemma}\label{lemw2ff*}
	For $\FF, X, W$ as in Definition \ref{defn-sheafmaps2ff} above, the Gromov diagonal normal sheaves satisfy
	$$(\maps(W,\FF))^\bullet = \maps(W,\FF^\bullet). $$
\end{lemma}

\begin{proof} Consider the sheaf $\PP$ on $X \times X$ given by 
$$\PP(U \times V) = \maps (U, \maps(W, \FF) (V)).$$
Also, let $\PP_1$ be the sheaf on $X \times X$ given by 
$$\PP_1(U \times V) = \maps (U,  \FF (V)).$$
Then,
	\begin{eqnarray*}
		\PP(U \times V) &= \maps (U, \maps(W, \FF(V)) ) &= \maps (U\times W, \FF(V)) )\\
		&=\maps (W, \maps(U, \FF(V)) ) &= \maps (W,\PP_1(U \times V)).
	\end{eqnarray*}
Hence,
	\begin{eqnarray*}
		(\maps(W,\FF))^\bullet =& \diag^*\PP  &= \diag^* (\maps (W, \PP_1) ) \\ =&  \maps (W, \diag^*\PP_1)&=  \maps (W, \FF^\bullet).\qedhere
	\end{eqnarray*}
\end{proof}

There exists a sheaf over $W \times X$  closely related  to the sheaf $\maps(W,\FF)$ over $X$ (Definition \ref{defn-sheafmaps2ff}). This is defined below.

\begin{defn}\label{defn-sheafmaps2ff2}
	Let $\GG$ denote a {continuous sheaf} over $X$. Let $W$ be a fixed topological space.
	Define a new sheaf of $W-$parametrized sections $\FF=\maps^p(W,\GG)$ over $W\times X$ as follows. For any $U \subset X$ and $V\subset W$ open, set $$\FF(V \times U)=\maps^p(W,\FF) (V\times U) = \maps(V,\GG(U)), $$ where $ \maps(V,\GG(U))$
	is equipped with the {natural quasitopology on mapping spaces}.
\end{defn}

\begin{eg}\label{eg-surj3}{\rm
		A natural example of a sheaf of $W-$parametrized sections may be given by the following.
	Let $P: Y \to X$ be a continuous surjective map and let $\GG$ denote the sheaf of continuous sections of $P$. Let $W$ be a fixed topological space. Define
		a continuous surjective map $P_W: W \times Y \to W \times X$ such that $P_W(w,y)=(w, P(y))$. Then the sheaf of continuous sections of $P_W$ is given precisely by 
		$\FF=\maps^p(W,\GG)$.}
\end{eg}

\begin{lemma}\label{lemw2ff*p}
	For $\FF, \GG, X, W$ as in Definition \ref{defn-sheafmaps2ff2} above {with $W$ locally contractible}, the Gromov diagonal normal sheaves satisfy the following for open $V \subset W$ and $U \subset X$: {There is a homotopy equivalence,
	$$\FF^\bullet(V \times U) \simeq \maps(V,\GG^\bullet(U)). $$}
\end{lemma}

\begin{proof} As in Definition \ref{def-formalfn}, the Gromov diagonal normal sheaf is constructed by first  constructing
	a sheaf $\PP$ on $(W \times X)\times (W \times X)$ as follows.
	\begin{eqnarray*}
		{\PP((V \times U)\times (V \times U))} =& \maps((V \times U), \FF (V \times U)) \\ 
		=&  \maps((V \times U),\maps(V,\GG(U)))  \\
		=& \maps((V \times V),\maps(U,\GG(U)))  \\
	\end{eqnarray*}
	
	Restricting $\PP$ to the diagonal, we have the following.
{
	\begin{eqnarray*}
		\FF^\bullet(V \times U) = &	\diag^*\PP((V \times U)\times (V \times U))\\
		= & \mathrm{diag}_V^* \maps((V \times V), \mathrm{diag}_U^*\maps(U,\GG(U)))  \\
		= & \mathrm{diag}_V^* \maps((V \times V), \GG^\bullet(U))
	\end{eqnarray*}
The last expression is the space of germs of maps from the tangent microbundle of $V$ to $\GG^\bullet(U)$. As $W$ is locally contractible, the tangent microbundle of any open subset $V \subset W$ deformation retracts to its zero section, i.e., $V$ itself. Thus, there is a homotopy equivalence
$$\mathrm{diag}_V^* \maps((V \times V), \GG^\bullet(U)) \simeq \maps(V, \GG^\bullet(U)).$$}
	This completes the proof.
\end{proof}

There is a tautological inclusion $$\De: \FF \stackrel{\iota}\longrightarrow \FF^\bullet$$ sending $s \in \FF(U)$ to 
$(s,\{s_x: x \in U\})$, where $s_x$ denotes the germ of the section $s$ at $x \in U$.

\begin{defn}\label{def-sheafh}\cite[p. 76]{Gromov_PDR}
	A (continuous) sheaf $\FF$ satisfies the sheaf theoretic $h$-principle {if the morphism $\Delta_U : \mathcal{F}(U) \to \FF^\bullet(U)$ is a surjection in $\pi_0$ for all open $U \subset X$, i.e., every section $\phi \in \FF^\bullet(U)$ can be homotoped to lie in $\Delta_U(\mathcal{F}(U))$}. Further, $\FF$ satisfies the 
	parametric sheaf theoretic $h$-principle if the morphism 
	$ \Delta_U: \FF(U) \to \FF^\bullet(U)$ is a weak homotopy equivalence for all open $U \subset X$.
\end{defn}

\begin{prop}\label{prop-formalisflex}\cite[p. 76]{Gromov_PDR}
	Let $\FF$ be a continuous sheaf over a locally compact finite dimensional Hausdorff space $X$. Then $\FF^\bullet$ is flexible.
\end{prop}

\begin{theorem}\label{thm-fleximphprin}\cite[p. 76]{Gromov_PDR}
	Let $\FF$ be a  continuous 
	sheaf  over a locally compact countable polyhedron  $X$ (e.g.\ a manifold or a stratified space). If $\FF$ is flexible, it satisfies the parametric $h$-principle.
\end{theorem}

\begin{rmk}\label{rmk-hprinvsflex}\rm{
	For sheaves, the notion of flexibility is strictly stronger than that of an $h-$principle. Suppose $X$ is locally contractible. By definition of the stalks $\FF_x$ and $\FF^\bullet_x$ of a continuous sheaf $\FF$  and its diagonal normal form $\FF^\bullet$, {respectively, we see that the tautological inclusion $\De_x : \FF_x \to \FF^\bullet_x$ is a weak homotopy equivalence at the level of stalks. That is, the germ of a formal section at $x$ is homotopic to the germ of a holonomic section at $x$. Therefore, given a formal section in an open neighborhood $U_x$ of $x$, we may homotope it to a holonomic section, possibly in a small open neighborhood $U_x' \subset U_x$.}
	Flexibility allows us to glue these local holonomic sections together to obtain a (global) holonomic section over a large open set. Thus, flexibility may be thought of as an analog of a Mayer-Vietoris principle used to glue homotopy equivalences (see for instance Theorem \ref{thm-bh} below).
	
	The existence of the $h-$principle is invariant under homotopy equivalence of sheaves (essentially by definition), i.e.\  if $\FF$ satisfies the $h-$principle and $\GG$ is 
	homotopy equivalent to $\FF$, then $\GG$ satisfies the $h-$principle.
	The same is not true for flexibility. This is the raison d'etre behind the existence of Section 2.2.7 of \cite{Gromov_PDR}.}
\end{rmk}

{Suppose $(X, \Sigma)$ is a stratified space and $\FF = \{\FF_{\bbar{S}} : S \in \Sigma\}$ is a stratified continuous sheaf on $(X, \Sigma)$. If $\FF$ is flexible (Definition \ref{def-stratflex}), then the individual sheaves $\FF_{\bbar{S}}$ over $\bbar{S}$ are flexible continuous sheaves. Thus, by Theorem \ref{thm-fleximphprin}, $\FF_{\bbar{S}}$ satisfies the parametric $h-$principle for all $S \in \Sigma$. However, in many natural examples (such as Example \ref{eg-stratflexnotflex}), we only have \emph{stratumwise flexibility} (Definition \ref{def-stratflex}) of $\FF$. The goal of the subsequent sections is to deduce the parametric $h-$principle for all the individual sheaves $\FF_{\bbar{S}}$ while only assuming stratumwise flexibility of $\FF$ and flexibility of the open homotopy fiber sheaves $\HH^L_S$ (Definition \ref{def-hh}).}

\subsection{{Gromov diagonal normal construction for  stratified continuous sheaves}}\label{sec-stratfdgromov}  We now extend the notion of a sheaf-theoretic $h-$principle (Definition \ref{def-sheafh}) to stratified sheaves.
The essential difference between a stratified continuous sheaf and a continuous sheaf over $X$ is that a stratified sheaf $\FF$ assigns a quasitopological space $\FF(U)$ to  an open subset of $U \subset \bbar{L}$ for \emph{every} stratum $L$ of $X$, whereas an ordinary sheaf does so only for open subsets of $X$. Hence, for every pair $S<L$, and for every $x \in S$, there are two stalks:
\begin{enumerate}
	\item $(\FF_S)_x$, which we shall refer to as the \emph{intrinsic stalk}. More generally, for 
	any $U \subset S$, $(\FF_S)\lvert_{U}$ will be referred to as the \emph{intrinsic sheaf} over $U$.
	\item $(i_S^*\FF_{\bbar{L}})_x$, which we shall refer to as the \emph{extrinsic stalk}. More generally, for 
	any $U\subset S $, $(i_S^*\FF_{\bbar{L}})\lvert_{U}$ will be referred to as the \emph{extrinsic sheaf} over $U$.
\end{enumerate}

\begin{defn}\label{def-diagstratsheaf}
For a stratified sheaf $\FF =\{\FF_{\bbar{S}} \}$, over a stratified space $X$, the Gromov diagonal normal 
stratified sheaf is given by $\FF^\bullet =\{\FF^\bullet_{\bbar{S}} \}$.
\end{defn}

We check below that  the restriction morphisms of $\FF^\bullet$ are the expected ones, and that there is a canonical map from $\FF$ to $\FF^\bullet$.

\begin{lemma}\label{lem-f2f*}
	{$\FF^\bullet$ is a stratified sheaf, and} $\Delta: \FF \to \FF^\bullet$ is a morphism of stratified sheaves.
\end{lemma}

\begin{proof}
{To check the first claim, on $X \times X$ we define 
$$\PP_{\bbar{L}}(U \times V) = \maps(U, \FF(V))$$
where $U, V$ are any pair of open subsets in some stratum-closure $\bbar L$. The collection 
$$\{\PP_{\bbar{L}}: L \text{ a stratum of } X\}$$ 
defines a stratified sheaf $\PP$ on $(X \times X, \Sigma \times \Sigma)$. We define the Gromov diagonal normal stratified sheaf $\FF^\bullet$ on $X$ by restricting $\PP$ to $\diag(X)$. This agrees with the description $\FF^\bullet = \{\FF^\bullet_{\bbar{L}}\}$ in Definition \ref{def-diagstratsheaf} by construction. For a pair of strata $S < L$, we have the intrinsic sheaf $\FF^\bullet_S$ on the open stratum and the extrinsic sheaf $i_S^* \FF^\bullet_{\bbar{L}}$. The restriction morphism $(\mathrm{res}^L_S)^\bullet : i_S^*\FF^\bullet_{\bbar{L}} \to \FF^\bullet_S$ can be explicitly described as follows.}

{For any open subset $U \subset S$, let  $U_L$ be an open subset of $\bbar{L}$ containing $U$ such that $U_L \cap S = U$. Let $\mathcal{B}(U)$ denote a collection of basic open subsets $\{V \times W\}$ in the product topology on $S \times S$ covering the diagonally embedded copy of $U \subset S$ in $S \times S$. Let $\mathcal{B}(U_L)$ denote the collection of basic open subsets $\{V_L \times W_L\}$ of $\bbar{L} \times \bbar{L}$ covering the diagonally embedded copy of $U_L \subset \bbar{L}$ in $\bbar{L} \times \bbar{L}$.  Then, 
\begin{gather*}
\FF^\bullet_S(U) = \varinjlim_{\mathcal{B}(U)} \varprojlim_{V \times W \in \mathcal{B}(U)} \mathcal{P}_S(V \times W) \\ i_S^*\FF^\bullet_{\bbar{L}}(U) = \varinjlim_{U_L \supset U} \varinjlim_{\mathcal{B}(U_L)} \varprojlim_{V_L \times W_L \in \mathcal{B}(U_L)} \mathcal{P}_L(V_L \times W_L)
\end{gather*}
where by the direct limits indexed by $\mathcal{B}(U), \mathcal{B}(U_L)$ we mean direct limit as the union of the elements of the open covers decrease to $U$, $U_L$, respectively. Since $V \times W \subset V_L \times W_L$, we have the restriction maps $\mathcal{P}_L(V_L \times W_L) \to \mathcal{P}_S(V \times W)$ as $\mathcal{P}$ is a stratified sheaf. Passing to limits furnishes $(\mathrm{res}^L_S)^\bullet$. }

{For the second claim, we observe that for any stratum $\bbar{L}$, 
$$\Delta : \mathcal{F}_{\bbar{L}} \to \FF^\bullet_{\bbar{L}}$$
is a morphism of sheaves. To conclude that these constitute a morphism of stratified sheaves, we need only verify that the following diagram commutes:
\begin{center}
		$
		\begin{CD}
			i_S^*	\FF_{\bbar{L}} @>\Delta>>i_S^*	\FF^\bullet_{\bbar{L}} \\
			@V{\res^{L}_S}VV @VV{(\res^{L}_S)^\bullet}V\\
			\FF_S  	@>\Delta>>\FF^\bullet_S
		\end{CD}
		$
	\end{center}
This is an exercise in unwinding definitions. Under the top horizontal arrow, a section $\phi$ of $i_S^*\FF_{\bbar{L}}(U)$ maps to an element in the triple limit defining $i_S^*\FF^\bullet_{\bbar{L}}(U)$. This element is  given 
 by the collection of constant $\phi|_{W_L}$-valued maps from $V_L$ to $\mathcal{F}_{\bbar{L}}(W_L)$. Here, we recall that $\mathcal{P}_L(V_L \times W_L) = \maps(V_L, \mathcal{F}_{\bbar{L}}(W_L))$. By the right vertical arrow, this
 collection of constant $\phi|_{W_L}$-valued maps  is sent to an element in the double limit defining $\FF^\bullet_S(W)$. This element is
 given by the collection of  constant $\mathrm{res}^L_S(\phi|_{W_L})-$ valued maps
 from $V$ to $\mathcal{F}_S(W)$.  Here, we recall that  $\mathcal{P}_S(V \times W) = \maps(V, \mathcal{F}_S(W))$. Since $\mathrm{res}^L_S(\phi|_{W_L}) = \phi \vert_W$,
 this element of $\FF^\bullet_S(U)$ is exactly the image of $\mathrm{res}^L_S(\phi)$ under the bottom horizontal arrow, concluding the proof.}
\end{proof}

{\begin{rmk}\label{rmk-diagstratsheaf}
The essence of the proof above is that the definition of $\FF^\bullet$ only considers the strata $S \times S \subset X \times X$, even though $X \times X$ also has mixed strata of the form $S \times L$ for $S, L \in \Sigma$.
\end{rmk}}

We now have the following analog of Definition \ref{def-sheafh}:

\begin{defn}\label{def-ssheafh}
	{A stratified continuous sheaf $\FF$ on $(X, \Sigma)$ satisfies the stratified sheaf theoretic $h$-principle, if for any object $U$ of $\str(X, \Sigma)$, the morphism $\Delta_U : \mathcal{F}(U) \to \FF^\bullet(U)$ (furnished by Lemma \ref{lem-f2f*}) is a surjection in $\pi_0$, i.e., for every stratum $L \in \Sigma$ of $X$ and open set $U \subset \bbar{L}$, any section $\phi \in \FF^\bullet_L(U)$ can be homotoped to lie in $\Delta(\FF_L(U)) \subset  \FF^\bullet_L(U)$.}
	
	{Further, $\FF$ satisfies the 
	parametric stratified sheaf theoretic $h$-principle if the morphism 
	$\Delta_U: \FF(U) \to \FF^\bullet(U)$ is a weak homotopy equivalence for all objects $U$ of $\str(X, \Sigma)$.}
\end{defn}

The following is an analog of Proposition \ref{prop-formalisflex} for stratified sheaves:

\begin{prop}\label{prop-formalisflexs}
	Let $\FF$ be a stratified continuous sheaf over a stratified space {$(X, \Sigma)$}. {Then $\FF^\bullet_{\bbar{L}}$ is flexible for every stratum $L \in \Sigma$}.
\end{prop}

\begin{proof}
	Flexibility of {$\FF^\bullet_{\bbar{L}}$} for every stratum $L$ follows from Proposition \ref{prop-formalisflex} and naturality of restriction maps from Lemma \ref{lem-f2f*}. 
\end{proof}

We are now in a position to state the the stratified analog of Theorem \ref{thm-fleximphprin}.

\begin{theorem}\label{thm-flex2shprin}
	Let $\FF$ be a stratified  (continuous) 
	sheaf  over a   stratified space {$(X, \Sigma)$}. If {$\FF_{\bbar{L}}$ is flexible for every $L \in \Sigma$}, it satisfies the parametric sheaf-theoretic stratified $h$-principle.
\end{theorem}

\begin{proof}
	By Lemma \ref{lem-f2f*}, $\Delta$ is a morphism of stratified sheaves.
	The weak homotopy equivalence property for $\Delta: \FF_{\bbar{L}} \to \FF^\bullet_{\bbar{L}}$ 
	for every stratum $L$ follows from Theorem \ref{thm-fleximphprin}. 
\end{proof}

\subsection{Topological properties}\label{sec-topprop} This subsection is rather general in flavor and sets up some basic homotopy theoretic properties of {continuous sheaves} that will be useful later. All topological spaces in this subsection are locally compact $\sigma$-compact finite dimensional locally contractible spaces.

\begin{defn}\label{def-ses}
	Let $\mathcal{F}_1, \mathcal{F}_2, \mathcal{F}_3$ be continuous sheaves on a topological space $X$. We say that 	$$\mathcal{F}_1 \stackrel{p}{\to} \mathcal{F}_2 \stackrel{q}{\to} \mathcal{F}_3$$ is a \emph{homotopy fiber sequence} if 
	\begin{enumerate}
		\item there exists some $\psi \in \mathcal{F}_3(X)$ such that $q_U \circ p_U : \mathcal{F}_1(U) \to \mathcal{F}_3(U)$ is the constant map to $\psi\vert_U$, for all $U \subset X$, and
		\item $\mathcal{F}_1 \to \mathrm{hofib}(q; \psi)$ is a weak homotopy equivalence.
	\end{enumerate}
\end{defn}

The following was observed by Gromov \cite[p.77]{Gromov_PDR} (see the paragraph preceding Theorem $B'$ there).
\begin{rmk}\label{lem-sessheaffibn}
	Let $\mathcal{F}, \mathcal{G}$ be continuous sheaves on $X$ and $q : \mathcal{F} \to \mathcal{G}$ be a morphism of continuous sheaves. If $\mathcal{F}, \mathcal{G}$ are flexible, then for any $\psi \in \mathcal{G}(X)$, $\mathrm{hofib}(q; \psi)$ is flexible. \end{rmk}

\begin{lemma}\label{lem-sessheafflex}
	Let 
	$$\mathcal{F}_1 \stackrel{p}{\to} \mathcal{F}_2 \stackrel{q}{\to} \mathcal{F}_3$$
	be a homotopy fiber sequence as in Definition \ref{def-ses}.
	If $\mathcal{F}_1, \mathcal{F}_3$ satisfy the parametric $h$-principle, then so does $\mathcal{F}_2$. \end{lemma}

\begin{proof} From the homotopy fiber sequence,
	we obtain a sequence of morphisms $\mathcal{P}_1 \to \mathcal{P}_2 \to \mathcal{P}_3$ of continuous sheaves over $X \times X$ (see Definition \ref{def-formalfn} and the preceding discussion for notation). Restricting to $\diag(X) \subset X \times X$, we obtain a sequence of morphisms $\mathcal{F}_1^\bullet \to \mathcal{F}_2^\bullet \to \mathcal{F}_3^\bullet$ of sheaves over $X$. 
	By functoriality of the diagonal normal construction (see for instance Lemma \ref{lem-f2f*}), we obtain a commutative diagram
	\begin{center}
		$\begin{CD}
			\mathcal{F}_1 @>>> \mathcal{F}_2 @>>> \mathcal{F}_3\\
			@VVV @VVV @VVV\\
			\mathcal{F}_1^\bullet @>>> \mathcal{F}_2^\bullet @>>> \mathcal{F}_3^\bullet
		\end{CD}$
	\end{center}
	As $\mathcal{F}_1, \mathcal{F}_3$ satisfies the parametric $h$-principle, the first and third vertical arrows are weak homotopy equivalences. We may evaluate the diagram of sheaves on any open set $U \subset X$, and use naturality of  homotopy long exact sequences corresponding to the rows to conclude, by an application of the 5-lemma, that $\mathcal{F}_2(U) \to \mathcal{F}_2^\bullet(U)$ is a weak homotopy equivalence. Thus, $\mathcal{F}_2 \to \mathcal{F}_2^\bullet$ is a weak homotopy equivalence of continuous sheaves. This demonstrates the parametric $h$-principle for $\mathcal{F}_2$.\end{proof}

\begin{rmk}\label{rmk-sessheafflex}
	The proof of Lemma \ref{lem-sessheafflex} goes through mutatis  mutandis to show that if any two of  $\mathcal{F}_1,\mathcal{F}_2, \mathcal{F}_3$ satisfy the parametric $h$-principle, then so does the third. 
\end{rmk}

\begin{conv}
Henceforth, we adopt Gromov's convention \cite[Section 1.4.1]{Gromov_PDR} of referring to an arbitrarily small but non-specified neighborhood of a set $K \subset X$
by $\op_X K$, or simply $\op K$ if there is no scope for confusion. Thus, $\op K$ refers to a small neighborhood of $K$ 
which may become even smaller in the course of the argument \cite[p. 35]{Gromov_PDR}
(see the table on \cite[p. 36]{Gromov_PDR} for further details about this convention/notation).
\end{conv}

\begin{lemma}\label{lem-formalfnrestcomm}
	Let $\mathcal{F}$ be a continuous sheaf on a topological space $X$, and $Z \subset X$ be a closed subspace. Then the diagonal normal construction commutes with restriction, i.e.\ there is a weak homotopy equivalence of continuous sheaves $\iota_Z^* (\FF^\bullet) \to (\iota_Z^* \mathcal{F})^\bullet$.\end{lemma}

\begin{proof} Recall that the diagonal normal construction applied to any sheaf yields a flexible sheaf (Proposition \ref{prop-formalisflex}). Therefore $\mathcal{G} := \FF^\bullet$ is a flexible sheaf on $X$. For $K \subset Z$  compact, we have the following:
	$$(\iota_Z^* \mathcal{G})(K) = \lim\limits_{\substack{V \supset K \\ V \subset Z \text{ open}}} \iota_Z^*\mathcal{G}(V) = \lim\limits_{\substack{V \supset K \\ V \subset Z \text{ open}}} \ \lim\limits_{\substack{U \supset V \\ U \subset X \text{ open}}} \mathcal{G}(U) = \lim\limits_{\substack{U \supset K \\ U \subset X \text{ open}}} \mathcal{G}(U) = \mathcal{G}(K)$$
	Thus, flexibility of $\mathcal{G}$ implies flexibility of $\iota_Z^* \mathcal{G} = \iota_Z^*(\FF^\bullet)$. Moreover, $(\iota_Z^* \mathcal{F})^\bullet$ is flexible by flexibility of the diagonal normal construction as mentioned above. Consider the sheaf morphism 
	$$(-)|_{Z} : \iota_Z^* (\FF^\bullet) \to (\iota_Z^* \mathcal{F})^\bullet$$
	defined on the stalk over $z \in Z$ by sending a germ of a mapping $\psi : \op_X(z) \to \mathcal{F}_z$ to its restriction $\psi|\op_Z(z) : \op_Z(z) \to \mathcal{F}_z$. We check that $(-)|_{Z}$ is  a sheaf morphism. Indeed, given any open set $U \subset Z$ consider an open cover $\{V_i\}$ of $U$ in $X$, and a collection $\{\phi_i : V_i \to \mathcal{F}(V_i)\}$ which is \emph{consistent}, i.e., 
	$$\mathrm{res}_{V_i\cap V_j, V_i} \circ \phi_i|_{V_i \cap V_j} \equiv \mathrm{res}_{V_i\cap V_j, V_j} \circ \phi_j|_{V_i \cap V_j}.$$
	Thus $\{\phi_i : V_i \to \mathcal{F}(V_i)\}$ represents an element $\phi \in \iota_Z^*(\FF^\bullet)(U)$. The restrictions $\{\phi_i : V_i \cap Z \to \mathcal{F}(V_i)\}$ are also  consistent, simply by restricting the above equality to $Z$. Therefore $\{\phi_i : V_i \cap Z \to \mathcal{F}(V_i)\}$ represents an element $\phi|_Z \in (\iota_Z^*\mathcal{F})^\bullet(U)$. Observe that $(-)|_{Z}$ is a \emph{stalkwise} weak homotopy equivalence as $X, Z$ are locally contractible. As both the domain and target sheaves are flexible, we  conclude that $(-)|_{Z}$ is a weak homotopy equivalence by appealing to a theorem of Gromov  \cite[Theorem B, p. 77]{Gromov_PDR} which says that local  weak homotopy equivalence implies  weak homotopy equivalence for flexible sheaves. \end{proof}

\begin{rmk}\label{rmk-formalfnrestcomm}\rm{
	Suppose $Z \subset X$ is a neighborhood deformation retract, and let $\pi : N_Z \to Z$ be a choice of such a retract. Then we can write down an explicit homotopy-inverse 
	$$(-) \circ \pi : (\iota_Z^* \mathcal{F})^\bullet \to \iota_Z^* (\FF^\bullet)$$
	defined on the stalk over $z \in Z$ by sending a germ of a mapping $\psi : \op_Z(z) \to \mathcal{F}_z$ to $\psi \circ \pi : \op_X(z) \to \mathcal{F}_z$. Let $\phi \in (\iota_Z^*\mathcal{F})^\bullet(U)$ be a section represented by a consistent family $\{\phi_i : V_i \cap Z \to \mathcal{F}(V_i)\}$ for a $\pi$-saturated open cover $\{V_i\}$ (i.e., $V_i = \pi^{-1}(V_i \cap Z)$) of $U$ in $N_Z$, satisfying the consistency relations:
	$$\mathrm{res}_{V_i\cap V_j, V_i} \circ \phi_i|_{V_i \cap V_j \cap Z} \equiv \mathrm{res}_{V_i\cap V_j, V_j} \circ \phi_j|_{V_i \cap V_j \cap Z}$$
	By taking a sufficiently fine open cover, we may assume $V_i \subset N_Z$. By composing with $\pi$ on both sides we obtain a consistent family $\{\phi_i \circ \pi : V_i \to \mathcal{F}(V_i)\}$ representing $\phi \circ \pi \in \iota_Z^*(\FF^\bullet)(U)$, proving that it is a well-defined sheaf homomorphism. Gromov's theorem  \cite[Theorem B, p. 77]{Gromov_PDR} once again demonstrates that it is a weak homotopy equivalence.}\end{rmk}

\begin{defn}\label{def-deldgerm}
	Let $(A, B)$ be a pair of topological spaces where $B \subset A$ is closed. Let $\mathcal{F}$ be a continuous sheaf on $A$. We define the space of sections on a \emph{deleted germinal neighborhood of $B$ in $A$} as 
	$$\mathcal{F}(\op(B) \setminus B) := \injlim_{U \supset B} \mathcal{F}(U \setminus B)$$ 
	There is a restriction map $\mathcal{F}(U) \to \mathcal{F}(U \setminus B)$ for every open neighborhood $U$ of $B \subset A$ which is compatible with the associated directed system indexed by the poset of open neighborhoods $\{U \subset B : A \subset U\}$. Hence we get a restriction map $\mathcal{F}(B) \to \mathcal{F}(\op(B) \setminus B)$ by applying direct limits $\varinjlim_U$ to both sides. Moreover, we have a restriction map $\mathcal{F}(A \setminus B) \to \mathcal{F}(\op(B) \setminus B)$ by restricting a section on $A \setminus B$ to a deleted germinal neighborhood of $B$ in $A$.\end{defn}

\begin{lemma}\label{lem-pastedeld}
	Let $(A, B)$ be a pair of topological spaces where $B \subset A$ is closed. Let $\mathcal{F}$ be a continuous sheaf on $A$. Then the following is a fiber square of quasitopological spaces:
	\begin{center}
		$\begin{CD}
			\mathcal{F}(A) @>>> \mathcal{F}(B)\\
			@VVV @VVV\\
			\mathcal{F}(A \setminus B) @>>> \mathcal{F}(\op(B) \setminus B)
		\end{CD}$
	\end{center}
\end{lemma}

\begin{proof} Suppose $\psi_1 \in \mathcal{F}(B)$ and $\psi_2 \in \mathcal{F}(A \setminus B)$ such that $\psi_1|({\op(B) \setminus B}) = \psi_2|({\op(B) \setminus B})$. Pick a representative $\widetilde{\psi}_1 \in \mathcal{F}(U)$ for some $U \supset B$ open neighborhood. Then we must have $\widetilde{\psi}_1|({V\setminus B}) = \psi_2|({V\setminus B})$ for some deleted open neighborhood $V \setminus B \subset A$. Next, we know that the following is a fiber square by the gluing axiom:
	\begin{center}
		$\begin{CD}
			\mathcal{F}(A) @>>> \mathcal{F}(V)\\
			@VVV @VVV\\
			\mathcal{F}(A \setminus B) @>>> \mathcal{F}(V \setminus B)
		\end{CD}$
	\end{center}
	We may then glue $\widetilde{\psi}_1|V$ and $\psi_2$ to obtain a $\psi \in \mathcal{F}(A)$. The element $\psi$ is independent of the choice of $\widetilde{\psi}_1$. We therefore obtain a well-defined map 
	$$\Psi : \mathcal{F}(A \setminus B)\times_{\mathcal{F}(\op(B) \setminus B)} \mathcal{F}(B) \to \mathcal{F}(A),$$ given  by $\Psi (\psi_1, \psi_2) = \psi.$
	For any topological space $X$, consider  a continuous map $f : X \to \mathcal{F}(A \setminus B) \times_{\mathcal{F}(\op(B)\setminus B)} \mathcal{F}(B)$ with respect to the quasitopology on the codomain. Then $f$ is equivalent  to a pair of continuous maps $f_1 : X \to \mathcal{F}(A \setminus B)$ and $f_2 : X \to \mathcal{F}(B)$ which agree when composed with the restriction to $\mathcal{F}(\op(B) \setminus B)$, by definition of quasitopology of fiber products. By definition of quasitopology on limits, there exists an open neighborhood $U$ and a deleted open neighborhood $V \setminus B$ of $B$ contained in $U$, such that $f_2$ factors through a continuous map $\widetilde{f}_2 : X \to \mathcal{F}(U)$ and $f_1, \widetilde{f}_2$ agree when restricted to $\mathcal{F}(V \setminus B)$. Therefore, we may paste $\widetilde{f}_2|V$ and $f_1$ to a continuous map $g : X \to \mathcal{F}(A)$. We see that $g = f \circ \Psi$, therefore $\Psi$ preserves the quasitopologies on the domain and codomain, i.e., $\Psi$ is continuous.
	
	Finally, $\Psi$ is inverse to the natural continuous map going in the opposite direction obtained from the universal property of fiber products. Thus, $\Psi$ is a homeomorphism of quasitopological spaces.\end{proof}

\begin{lemma}\label{lem-pastecpts} Let $\mathcal{F}$ be a flexible sheaf on a topological space on $X$. Let $A, K \subset X$ be a pair of subsets such that $K, A\cap K$ are both compact. Then $\mathcal{F}(A \cup K) \to \mathcal{F}(A)$ is a fibration.
\end{lemma}

\begin{proof}  The following is a fiber square of quasitopological space
	\begin{center}
		$\begin{CD}
			\mathcal{F}(A \cup K) @>{\mathrm{res}_{K, A\cup K}}>> \mathcal{F}(K)\\
			@V{\mathrm{res}_{A, A\cup K}}VV @V{\mathrm{res}_{A\cap K, K}}VV\\
			\mathcal{F}(A) @>{\mathrm{res}_{A\cap K, A}}>> \mathcal{F}(A \cap K)
		\end{CD}$
	\end{center}
	Suppose $\psi : W \times I \to \mathcal{F}(A)$ is a homotopy with an initial lift $\widetilde{\psi}_0 : W \times 0 \to \mathcal{F}(A \cap K)$. As $A\cap K, K$ are compact, $\mathrm{res}_{A\cap K, K} : \mathcal{F}(K) \to \mathcal{F}(A \cap K)$ is a fibration. Therefore, we may lift $\mathrm{res}_{A \cap K, A} \circ \psi : W \times I \to \mathcal{F}(A \cap K)$ with initial condition $\mathrm{res}_{K, A\cup K} \circ \widetilde{\psi}_0 : W \times 0 \to \mathcal{F}(K)$ to a homotopy $\varphi : W \times I \to \mathcal{F}(K)$. Finally, using the fact that the diagram is a fiber square, $\psi, \varphi$ provide a lift $\widetilde{\psi} : W \times I \to \mathcal{F}(A \cup K)$ of $\psi$, finishing the proof.
\end{proof}

\begin{lemma}\label{lem-ltsfibns}  Let $\{X_n\}$ be an inverse system and $\{Y_n\}$ be a directed system of quasitopological spaces. Let $Z, W$ be quasitopological spaces. Let $\{f_n : X_n \to Z\}$ and $\{g_n : W \to Y_n\}$ be a collection of maps compatible with the systems $\{X_n\}$ and $\{Y_n\}$ respectively. Let $X = \varprojlim X_n$, $Y = \varinjlim Y_n$, and $f : X \to Z$, $g : W \to Y$ be the canonical maps from, and to, the respective limits.
	\begin{enumerate}
		\item If $f_n$ are fibrations and the structure maps in $\{X_n\}$ are also fibrations, then $f$ is a fibration.
		\item If $g_n$ are fibrations, then $g$ is a fibration.
	\end{enumerate}
\end{lemma}

\begin{proof}
	Let $Q$ be an auxiliary topological space. 
	Let $\psi : Q \times I \to Z$ be a homotopy with an initial lift $\widetilde{\psi}_0 : Q \times \{0\} \to X$ of $\psi_0 := \psi|Q \times \{0\}$. Choose a lift of $\psi$ to a homotopy $Q \times I \to X_1$ along the fibration $f_1 : X_1 \to Z$, given initial condition $\pi_1 \circ \widetilde{\psi}_0 : Q \times \{0\} \to X_1$. Then, since the structure maps of the inverse system are fibrations, we may lift $Q \times I \to X_1$ to a homotopy $Q \times I \to X_n$ using as initial condition the maps $\pi_n \circ \widetilde{\psi}_0$ for all $n \geq 2$. This gives a collection of homotopies $\{Q \times I \to X_n\}$ compatible with the structure maps of the inverse system. By the universal property of inverse limits, this provides a homotopy $\widetilde{\psi} : Q \times I \to X$, such that $\widetilde{\psi}$ is a lift of $\psi$ with initial condition $\widetilde{\psi}_0$ as desired. This proves (1).
	
	 Let $\psi : Q \times I \to Y$ be a homotopy with an initial lift $\widetilde{\psi}_0 : Q \times \{0\} \to W$ of $\psi_0 := \psi|Q \times \{0\}$. By definition of the quasitopology of the direct limit, $\psi$ factors through a homotopy $Q \times I \to Y_n$. Since $g_n : W \to Y_n$ is a fibration, we may lift $\psi$ to $Q \times I \to W$, using $\widetilde{\psi}_0 : Q \times \{0\} \to W$ as the initial condition. This is a lift of $\psi$ with initial condition $\widetilde{\psi}_0$, as desired.
	This proves (2).
\end{proof}

We shall need the following theorem to `coglue' {weak} homotopy equivalences. {We refer the reader to \cite{brown-heath} by Brown and Heath, and the  notes  \cite{frankland-notes,francis-notes} for a proof.
\begin{theorem}\label{thm-bh}
	Consider a commutative diagram of maps of {quasitopological spaces} as in Figure \ref{homotopyco}, where  
	\begin{enumerate}
		\item the front and back squares are \emph{{(strict)}} pullback diagrams (equivalently, $Q$ and $P$ are fiber-products),
		\item $p, q$ are fibrations
	\end{enumerate}
	If the diagonal arrows, labeled $\phi_1,\phi_2, \phi$ are {weak} homotopy equivalences, so is $\Phi$.
\end{theorem}
\begin{comment}[H]
	\begin{center}
		\includegraphics[width=0.75\linewidth]{homotopyco.png}
	\end{center}
	\caption{Homotopy Co-gluing}
	%\label{homotopyco}
\end{comment}

\begin{figure}[H]
	\begin{tikzcd}
		Q \arrow{rr}{\til{g}} \arrow{dd}{\til{q}} \arrow{rrrrd}[near end]{ \Phi} &  & E \arrow{dd}{q} \arrow{rrrrd}{\phi_1}                         &  &                                                        &  &                    \\
		&  &                                                                    &  & P \arrow{rr}[pos=0.3]{\til{f}} \arrow{dd}{\til{p}}&  & D \arrow{dd}{p} \\
		Y \arrow{rr}{g} \arrow{rrrrd}[pos=0.6]{ \phi_2 }                                                     &  & B \arrow{rrrrd}[near start]{\phi  } &  &                                                        &  &                    \\
		&  &                                                                    &  & X \arrow{rr}[near start]{f}                                     &  & A                 
	\end{tikzcd}
	\caption{Homotopy Co-gluing}
	\label{homotopyco}
\end{figure}
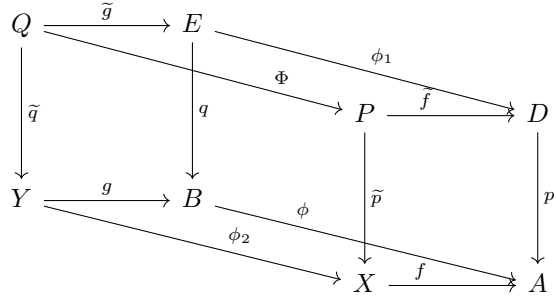

We note here for later use, a fact about homotopy fibers of fiber-products:

\begin{lemma}\label{lem-hofib-fiberpdkt}
	Let $f: X \to Z$, and $g: Y \to Z$ be continuous maps, furnishing the following pullback diagram:
	\begin{center}

		$
		\begin{CD}
			X \times_Z Y @>F>>Y \\
			@VVV  @VgVV\\
			X@>f>>Z \\
		\end{CD}
		$
	\end{center}
	Then, the homotopy fibers of
	$X \times_Z Y \to Y$ and $f:X\to  Z$ are homotopy equivalent. 
\end{lemma}

\begin{proof}
	There are two homotopy fiber bundles that can be constructed over $Y$ as follows:
	\begin{enumerate}
		\item Let $\PP(X \times_Z Y,F,Y)\to Y$ denote the path space fibration construction applied to $F: X \times_Z Y \to Y$, and let $P_1:E_1 \to Y$ denote the resulting fibration. Then $\hofib (P_1)$ is homotopy equivalent to $\hofib (F)$.
		\item Let $\PP(X,f,Z)\to Z$ denote the path space fibration construction applied to $f: X \to Z$, and let $P_2: E_2 \to Y$ denote the pullback fibration (under $g$) of $\PP(X,f,Z)$.
		Then $\hofib (P_2)$ is homotopy equivalent to $\hofib (f)$.
	\end{enumerate}
	Then there exists
	a homotopy equivalence $E_1 \to X \times_Z Y$ covering the identity map over $Y$. 
	The same holds for $E_2$. Hence, there exists a homotopy equivalence
	$\phi:E_1 \to E_2$ of total spaces covering the identity map over $Y$. Since
	$P_1: E_1\to Y$, and  $P_2: E_2 \to Y$ are fibrations, $\hofib (P_1)$ 
	is homotopy equivalent to $\hofib (P_2)$. The lemma follows.
\end{proof}

\subsection{Sheaf-theoretic $h$-principle for stratified spaces}\label{sec-sheafh}
Let $(X, \Sigma, {\mathcal{N}})$ be an abstractly stratified set, and $\mathcal{F}$ be a stratified continuous sheaf on $(X, \Sigma)$.
For every stratum $S \in \Sigma$, we denote the associated sheaf on the closure $\bbar{S} \subset X$ by $\mathcal{F}_{\bbar{S}}$  so that $\mathcal{F}_S := \iota_S^*\mathcal{F}_{\bbar{S}}$. For every pair of strata $S, L \in \Sigma$, $S < L$, we also have the restriction map from the sheaf on $\bbar{L}$ to that on $\bbar{S} \subset \bbar{L}$ given by ${\res}^L_S: \iota_{\bbar{S}}^*\mathcal{F}_{\bbar{L}} \to \mathcal{F}_{\bbar{S}}$. Recall also the restriction map ${\res}^L_S : \iota_S^* \mathcal{F}_{\bbar{L}} \to \mathcal{F}_S$.\\

{
\noindent {\bf Scheme:} The main aim of this subsection is to establish Theorem \ref{thm-sheafh-intro} (Theorem \ref{thm-hofibsflexg}), which is one of the main theorems of this paper. The book-keeping is a little involved; so we summarize the scheme here in order to provide a road-map to the reader:
\begin{enumerate}
\item Section \ref{sec-prellts} starts off by dealing with the simplest situation, where $\Sigma$ consists  of a pair of strata
$S<L$, so that $X= L \cup S$. For $\FF$ a stratified continuous sheaf on $X$, Proposition \ref{prop-sheafh} provides sufficient conditions to establish that 
the single continuous sheaf $\FF_{\bbar{L}}$ satisfies the parametric $h-$principle.
\item Section \ref{sec-glueacross} deals with a kind of Mayer-Vietoris principle, allowing us to glue sheaves on two maximal strata \emph{of the same height} intersecting along their closures. More generally, it allows us to construct a single continuous sheaf out of a stratified continuous sheaf.
\item Section \ref{sec-main} proves Theorem \ref{thm-hofibsflexg}, first under the hypothesis that the strata are \emph{well-ordered}, and then in the general case, using the results of Sections \ref{sec-prellts} and \ref{sec-glueacross}.
\item Section \ref{sec-ctreg} provides counter-examples to show that  the hypotheses of
 Theorem \ref{thm-hofibsflexg} are crucial.
\end{enumerate}
Item (2) above allows us to `coarsen' the partial order on strata using height. Corollary \ref{cor-3glue} in Section \ref{sec-glueacross} allows us to treat \emph{all} strata at a common maximal height on the same footing. Assuming that the parametric $h-$principle has already been established on the boundary of the union of maximal height strata, it gives sufficient conditions to extend the parametric $h-$principle  to the whole stratified space. Theorem \ref{thm-hofibsflexg} then uses Corollary \ref{cor-3glue} repeatedly to establish the inductive hypothesis. The main book-keeping problem arises due to the fact that there is no canonical total order associated to a partial order, and we have to make a choice. Thus, while Corollary \ref{cor-3glue} can be applied directly to establish Theorem \ref{thm-hofibsflexg} inductively when the strata are totally ordered, a second induction has to be applied when 
 the strata are only partially ordered. This second induction is performed in the general case of the proof of Theorem \ref{thm-hofibsflexg}.}

\subsubsection{Limits and preliminary gluing}\label{sec-prellts}
\begin{lemma}\label{lem-topstrat}
	Let $U \subset \bbar{L}$ be an open subset and $U' := U \cap S$. 
	Let $\mathcal{G} = \mathcal{F}_{\bbar{L}},$ or $ \FF^\bullet_{\bbar{L}}$. Suppose that $\GG$ is flexible. Then $\mathcal{G}(U \setminus U') \to \mathcal{G}(\op_{U}(U') \setminus U')$ is a fibration.
\end{lemma}

\begin{proof}
	Let $\{K_n\}$ be an ascending sequence of compact subsets of $U \setminus U'$ exhausting $U \setminus U'$. Let   $\{C_n\}$ be closures (in $U$) of a descending sequence of open neighborhoods of $U'$ in $U$  such that $\cap_n C_n = U'$, i.e.\ 
	\begin{gather*}
		U' \subset \cdots \subset C_3 \subset  C_2 \subset C_1 \subset U \\
		K_1 \subset K_2 \subset K_3 \subset \cdots \subset U \setminus U'
	\end{gather*}
	
	Consider restriction maps $r_{m, n} : \mathcal{G}(K_m \cup (C_n \setminus U')) \to \mathcal{G}(C_n \setminus U')$. Since both $K_m$ and $K_m \cap (C_n \setminus U') = K_m \cap C_n$ are compact,  flexibility of $\mathcal{G}|L$ on $L$ then  shows that for all $m, n$, $r_{m, n} $ is a  fibration by Lemma \ref{lem-pastecpts}. We observe now that:
	\begin{claim}\label{claim1}
		$\{\mathcal{G}(K_m \cup (C_n \setminus U'))\}_m$ is an inverse system with structure maps given by fibrations.
	\end{claim} 
	\begin{proof}[Proof of Claim \ref{claim1}:]
		Observe that for all $m \geq 1$, $K_{m+1} \cup (C_n \setminus U') = (K_m \cup C_n \setminus U') \cup K_{m+1}$. Moreover, $(K_m \cup C_n \setminus U') \cap K_{m+1} = K_m \cup (K_{m+1} \cap C_n)$ and $K_{m+1}$ are both compact. Therefore, Lemma \ref{lem-pastecpts} applies, and we obtain
		$$\mathcal{G}(K_{m+1} \cup (C_n \setminus U')) \to \mathcal{G}(K_m \cup (C_n \setminus U'))$$
		is a fibration, for all $m \geq 1$.\end{proof}
	
	\begin{claim}\label{claim2} $\varprojlim_m \mathcal{G}(K_m \cup (C_n \setminus U')) \cong \mathcal{G}(U \setminus U')$. \end{claim} 
	
	\begin{proof}[Proof of Claim \ref{claim2}:] This is true in complete generality. Let $V$ be an arbitrary open set (here $V=(U \setminus U')$). Let  $\{C_n\}$ (resp.\  $\{U_n\}$) be a sequence  of closed (resp.\ open) subsets of $U$ such that 
		\begin{enumerate}
			\item $C_n \subset U_n \subset C_{n+1}$,
			\item $\cup_n C_n = U = \cup_n  U_n$.
		\end{enumerate}
		There exist restriction maps $\mathcal{G}(V) \to \mathcal{G}(C_n)$ which are compatible with the inverse system $\{\mathcal{G}(C_n)\}$. Hence by the universal property, there exists a continuous map $\mathcal{G}(V) \to \varprojlim \mathcal{G}(C_n)$. On the other hand, we have a continuous map $\Theta: \varprojlim \mathcal{G}(C_n) \to \varprojlim \mathcal{G}(U_{n-1})$ given by  restriction. Then $\Theta$ is a homeomorphism of quasitopological spaces, with inverse $\Theta^{-1}: \varprojlim \mathcal{G}(U_{n-1}) \to \varprojlim \mathcal{G}(C_{n-1})$ given by  restriction again. 
		
		The composition $\mathcal{G}(V) \to \varprojlim \mathcal{G}(C_n) \to \varprojlim \mathcal{G}(U_{n-1})$ is also a homeomorphism, by the gluing property of continuous sheaves applied to the open cover $\{U_{n-1}\}$ of $U$. Hence, $\mathcal{G}(V) \cong \varprojlim \mathcal{G}(C_n)$, as desired. \end{proof}
	
	\begin{claim}\label{claim3} $\varinjlim_n \mathcal{G}(C_n \setminus U') \cong \mathcal{G}(\op_U(U') \setminus U')$. \end{claim} 
	
	\begin{proof}[Proof of Claim \ref{claim3}:] Consider the following chain of homeomorphisms
		$$\varinjlim_n \mathcal{G}(C_n \setminus U') \cong \varinjlim_n \varinjlim_{V \supset C_n} \mathcal{G}(V \setminus U') \cong \varinjlim_{V \supset U'} \mathcal{G}(V \setminus U') \cong \mathcal{G}(\op_U(U') \setminus U')$$
		where $V \supset C_n$ varies over all open neighborhoods of $C_n$ in $U$.
	\end{proof}
	
	We now proceed to take limits of $r_{m, n}$, first as $m \to \infty$ and then as $n \to \infty$. By Lemma \ref{lem-ltsfibns}, this gives that $\mathcal{G}(U \setminus U') \to \mathcal{G}(\op_{U}(U') \setminus U')$ is also a fibration.
\end{proof}

\begin{prop}\label{prop-sheafh} Let $(X, \Sigma_X)$ be a stratified space with exactly two strata, $\Sigma_X = \{S, L\}$, with $S<L$, so that $X = \bbar{L}$. Let $\mathcal{F}$ be a stratified continuous sheaf on $X$. Assume the following:
	\begin{enumerate}
		\item {$\mathcal{F}_S=i_S^* \FF_{\bbar{S}}$} satisfies the parametric $h-$principle, 
		\item {$\mathcal{F}_L=i_L^* \FF_{\bbar{L}}$} is flexible, and
		\item for {some} $\psi \in \mathcal{F}_S(S)$, $\HH^L_S=\hofib(\res^L_{S}; \psi)$ satisfies the parametric $h$-principle.
	\end{enumerate}
	Then $ \mathcal{F}_{\bbar{L}}$ satisfies the parametric $h-$principle. (Note that here $\mathcal{F}_{\bbar{L}}$ is a single continuous sheaf, not a general stratified sheaf.)
\end{prop}

\begin{proof} Let $\psi \in \mathcal{F}_S(S)$ be as in Definition \ref{def-ses}. Then by hypothesis there exists a homotopy fiber sequence of continuous sheaves over $S$,  given by
	$$\mathrm{hofib}(\res^L_{S}; \psi) \to \iota_S^* \mathcal{F}_{\bbar{L}} \to \mathcal{F}_S.$$
	By hypothesis, $\mathcal{F}_S$ and $\mathrm{hofib}(\mathrm{res}^L_S; \psi)$ satisfy the parametric $h-$principle. Hence, by Lemma \ref{lem-sessheafflex} so does $\iota_S^*\mathcal{F}_{\bbar{L}}$, i.e.\ the natural inclusion  map $\iota_S^* \mathcal{F}_{\bbar{L}} \to (\iota_S^* \mathcal{F}_{\bbar{L}})^\bullet$ is a weak homotopy equivalence. From Lemma \ref{lem-formalfnrestcomm} and Remark \ref{rmk-formalfnrestcomm}, we know that the map $(-) \circ \pi_S : (\iota_S^* \mathcal{F}_{\bbar{L}})^\bullet \to \iota_S^* (\FF^\bullet_{\bbar{L}})$ is a weak homotopy equivalence. This implies that the resulting composition  map 
	$$\iota_S^* \mathcal{F}_{\bbar{L}} \to (\iota_S^* \mathcal{F}_{\bbar{L}})^\bullet  \to \iota_S^* (\FF^\bullet_{\bbar{L}})$$ 
is a weak homotopy equivalence. By construction, this is simply the restriction of the natural inclusion map of the diagonal normal construction $\mathcal{F}_{\bbar{L}} \to \FF^\bullet_{\bbar{L}}$ to the stratum $S$.

	Let $\Phi : \mathcal{F}_{\bbar{L}} \to \FF^\bullet_{\bbar{L}}$ 	denote the natural morphism in the diagonal normal construction. Then  $\Phi|L : \mathcal{F}_L \to \FF^\bullet_L$ and $\Phi|S : \iota_S^* \mathcal{F}_{\bbar{L}} \to \iota_S^* (\FF^\bullet_{\bbar{L}})$ are both weak homotopy equivalences. We would like to ``glue" these to a {weak} homotopy equivalence $\Phi$. To this end, let $U \subset \bbar{L}$ be an open subset and $U' := U \cap S$. Using Lemma \ref{lem-pastedeld}, we have fiber squares:
	$$
		\begin{CD}
			\mathcal{F}_{\bbar{L}}(U) @>>> \mathcal{F}_{\bbar{L}}(U')\\
			@VVV @VVV\\
			\mathcal{F}_{\bbar{L}}(U \setminus U') @>>> \mathcal{F}_{\bbar{L}}(\op_{U}(U') \setminus U')
		\end{CD}
		\qquad 
		\begin{CD}
			\FF^\bullet_{\bbar{L}}(U) @>>> \FF^\bullet_{\bbar{L}}(U')\\
			@VVV @VVV\\
			\FF^\bullet_{\bbar{L}}(U \setminus U') @>>> \FF^\bullet_{\bbar{L}}(\op_{U}(U') \setminus U')
		\end{CD}
	$$
	The diagonal normal construction map $\Phi$ gives natural maps from each corner of the first diagram to the corresponding corner of the second diagram. Note first that $\Phi$ is a weak homotopy equivalence on the top-right and bottom-left corners as $\Phi|L$ and $\Phi|S$ are weak homotopy equivalences of continuous sheaves.

	Next, note that
	$$\Phi : \mathcal{F}_{\bbar{L}}(\op_U(U')\setminus U') \to \FF^\bullet_{\bbar{L}}(\op_U(U') \setminus U')$$ 
	is a direct limit of the weak homotopy equivalences $\Phi|(V \setminus U') : \mathcal{F}_{\bbar{L}}(V \setminus U') \to \FF^\bullet_{\bbar{L}}(V \setminus U')$ indexed by open neighborhoods $V \subset U$ of $U'$ in $U$. Since homotopy groups commute with direct limit of quasitopological spaces, hence the limiting map is also a weak homotopy equivalence. Thus, $\Phi$  is a weak homotopy equivalence on the bottom-right corner as well.  
	
	We shall use Theorem \ref{thm-bh} to conclude the proof. To set up  the situation such that the hypotheses of Theorem \ref{thm-bh} are satisfied, we need to establish that  
	\begin{enumerate}
		\item $\Phi$  is a weak homotopy equivalence on the top-right, bottom-left and bottom-right corners. The above paragraph did precisely this.
		\item The right vertical arrows in  both commutative diagrams above are fibrations. Lemma \ref{lem-topstrat} gives this.
	\end{enumerate}
	
	Finally, we apply Theorem \ref{thm-bh} to ``coglue" to a homotopy equivalence $\Phi : \mathcal{F}_{\bbar{L}}(U) \to \FF^\bullet_{\bbar{L}}(U)$. This proves that $\mathcal{F}_{\bbar{L}}$ satisfies the parametric $h$-principle.\end{proof}

The  above proof also gives us a criterion for verifying the parametric $h-$principle for a continuous sheaf on the underlying topological space of a stratified space $(X, \Sigma)$. We emphasize that by this we mean a \emph{single, genuine sheaf} on the topological space $X$, as opposed to a stratified sheaf which consists of a collection of sheaves, one on the closure of each stratum in $\Sigma$. {Nevertheless, we may treat $\mathcal{F}$ as the stratified sheaf $\{\mathcal{F}_{\bbar{S}} := i_{\bbar{S}}^*\mathcal{F} : S \in \Sigma\}$ on $(X, \Sigma)$. For any pair of strata $S, L \in \Sigma$ with $S < L$, we then have an \emph{equality} of sheaves $i_S^*{\mathcal{F}_{\bbar{L}}} = \mathcal{F}_S = i_S^*\FF$. Consequently, the homotopy fiber sheaf $\HH^L_S$, being homotopy fiber of the identity map $i_S^*\FF_{\bbar{L}} \to \FF_S$, is homotopically trivial.}

\begin{lemma}\label{lem-exstratfleximplieshprin}Let $X$ be a stratified space. Let $\FF$ be a continuous sheaf on $X$, 
such that $i_{{S}}^* \FF$ is  flexible for all (open) strata $S$, then $\FF$ satisfies the parametric $h-$principle.
\end{lemma}

\begin{proof}We argue by induction on depth of $X$. If the stratified space $X$ has depth one, i.e.\ it is a manifold, then the result follows from Gromov's Theorem \ref{thm-fleximphprin}. Otherwise, Let $L$ denote the disjoint union of maximal strata of $X$, i.e.\ strata that do not lie on the boundary of other strata. Let $\partial L=X\setminus L$. Note that $\partial L$ is a stratified space of height strictly less than the height of $X$. Applying the inductive hypothesis to $\partial L$, we conclude that $i_{\partial L}^*\FF$ satisfies the parametric $h-$principle. By hypothesis, we also have flexibility of {$i_L^*\FF$}, since $L$ is a disjoint union of open strata in $X$. In particular, $i_L^*\FF$ satisfies the parametric $h-$principle as well.

As in the proof of Proposition \ref{prop-sheafh}, we use Theorem \ref{thm-bh} to co-glue the homotopy equivalences $i_{\partial L}^*\FF \to (i_{\partial L}^*\FF)^\bullet$ and $\iota_L^* \FF \to (\iota_L^* \FF)^\bullet$ to obtain a homotopy equivalence $\FF \to \FF^\bullet$. Therefore, $\FF$ satisfies the parametric $h-$principle.
\end{proof}

{Notice that we did not require the fact that $\HH^L_S$ is homotopically trivial in the proof of Lemma \ref{lem-exstratfleximplieshprin}. Indeed, the only place in Theorem \ref{thm-bh} we use $\HH^L_S$ is to deduce that $i_S^*\FF_{\bbar{L}}$ satisfies the $h-$principle from the fact that $\FF_S$ and $\HH^L_S$ both satisfy the $h-$principle. But in this specific case, this is automatic from the equality $i_S^*\FF_{\bbar{L}} = i_S^*\FF$, as $i_S^*\FF$ is flexible by hypothesis, and flexible sheaves satisfy the $h-$principle by Gromov's Theorem \ref{thm-fleximphprin}.}

\begin{rmk}\label{rmk-stratflexextvsint}
\rm{
{It is important to be mindful of the relationship between the hypothesis of Lemma \ref{lem-exstratfleximplieshprin} and the hypothesis of stratumwise flexibility in Definition \ref{def-stratflex}. The former is a hypothesis on a \emph{single} sheaf on the underlying topological space $X$ of the stratified space $(X, \Sigma)$, whereas the latter is a hypothesis on a \emph{stratified} sheaf on the stratified space $(X, \Sigma)$. As discussed prior to Lemma \ref{lem-exstratfleximplieshprin}, an individual sheaf $\mathcal{F}$ on $X$ can be treated as a stratified sheaf $\{i_{\bbar{S}}^*\mathcal{F} : S \in \Sigma\}$ on $(X, \Sigma)$. Then the hypothesis of Lemma \ref{lem-exstratfleximplieshprin} does translate to stratumwise flexibility of the stratified sheaf $\{i_{\bbar{S}}^*\mathcal{F} : S \in \Sigma\}$.}

{Notice that, given a general stratified sheaf $\{\mathcal{F}_{\bbar{L}} : L \in \Sigma\}$ on $(X, \Sigma)$, we would be able to ensure the parametric $h-$principle (Definition \ref{def-ssheafh}) if the hypothesis of Lemma \ref{lem-exstratfleximplieshprin} were to be true for each of the individual sheaves $\FF_{\bbar{L}}$ on $\bbar{L}$. This is equivalent to demanding the flexibility of the \emph{extrinsic} sheaves $i_S^*\FF_{\bbar{L}}$ for all $S, L \in \Sigma$, $S < L$. But for most naturally occurring examples, we would only have the luxury of asserting the flexibility of the \emph{intrinsic} sheaves $\FF_S$, $S \in \Sigma$. The difference between these sheaves lie in the homotopy fiber sheaf $\HH^L_S := \hofib(\res^L_S)$.}
}
\end{rmk}

\begin{comment}
It is important to distinguish the hypothesis of Lemma \ref{lem-exstratfleximplieshprin} from the hypothesis of stratumwise flexibility in Definition \ref{def-stratflex} which applies to \emph{stratified sheaves}. 
		{First off, Lemma  \ref{lem-exstratfleximplieshprin} deals with a \emph{single}
		continuous sheaf $\FF$, not a general stratified sheaf.
	Further, note that}	stratumwise flexibility is a condition on the \emph{intrinsic sheaves} $\FF_S$ of a stratified sheaf, whereas the hypotheses above are related to flexibility of the \emph{extrinsic sheaves} $\iota_S^* \FF_{\bbar{L}}$, for $S < L$. The difference between these sheaves lies in the homotopy fiber sheaf $\HH^L_S := \hofib(\res^L_S;\psi)$.
\end{comment}

\begin{lemma}\label{lem-bendright}
{Let $(X, \Sigma)$ be a stratified space, and $\mathcal{F}$ be a stratified sheaf on $X$. For any pair of strata $S < L$ in $X$, recall the associated closed and open homotopy fiber sheaves $\bH^L_S$ and $\HH^L_S$ from Definition \ref{def-hh}. For any triad of strata $P < S < L$ in $X$, there exist homotopy equivalences of sheaves}
\begin{gather*}i_{\bbar{P}}^*\bH^L_S \simeq \hofib\left(\bH^L_{P} \to \bH^S_{P}\right), \\
i_P^*\bH^L_S \simeq \hofib\left(\HH^L_P \to \HH^S_P\right).\end{gather*}
\end{lemma}

\begin{proof}  Let $f: X\to Y$ and $g: Y\to Z$. Then there exists a  homotopy fiber sequence $$\hofib(f)\to \hofib (g \circ f) \to \hofib (g)$$ by \cite[Lemma 1.2.7]{may-ponto}.
	To translate this into the context of continuous sheaves, consider the following diagram:

	\medskip
	\begin{tikzcd}
		&  &  & \hofib(i_{\bbar{P}}^*{\mathcal{F}}_{\bbar{L}} \to {\mathcal{F}}_{\bbar{P}}) \arrow{dd}{3} \arrow{rr}{2} &  & \hofib(i_{\bbar{P}}^*{\mathcal{F}}_{\bbar{S}} \to {\mathcal{F}}_{\bbar{P}}) \arrow{dd}{4} \\
		&  &  &                                                                                         &  &                                                                         \\
		\hofib(i_{\bbar{P}}^*{\mathcal{F}}_{\bbar{L}} \to i_{\bbar{P}}^*{\mathcal{F}}_{\bbar{S}}) \arrow{rrr}{5} \arrow{rrruu}[bend left]{1}  &  &  & i_{\bbar{P}}^*{\mathcal{F}}_{\bbar{L}} \arrow{rr}{6} \arrow{dd}{7}                            &  & i_{\bbar{P}}^*{\mathcal{F}}_{\bbar{S}} \arrow{lldd}{8}                          \\
		&  &  &                                                                                         &  &                                                                         \\
		&  &  & {\mathcal{F}}_{\bbar{P}}                                                                         &  &                                                                        
	\end{tikzcd}
	\bigskip

	Here, 
	the arrows 6,7,8 take the place of $f, g\circ f, g$ respectively.
	The above homotopy theoretic fact (\cite[Lemma 1.2.7]{may-ponto}) then shows that the arrows 1,2 give a  homotopy fiber sequence 
	$$\hofib(i_{\bbar{P}}^*\FF_{\bbar{L}} \to i_{\bbar{P}}^*\FF_{\bbar{S}})\stackrel{1}\longrightarrow \hofib(i_{\bbar{P}}^*\FF_{\bbar{L}} \to \FF_{\bbar{P}}) \stackrel{2}\longrightarrow  \hofib(i_{\bbar{P}}^*\FF_{\bbar{S}}  \to \FF_{\bbar{P}}).$$ 

The first term in the sequence above can be identified with $i_{\bbar{P}}^* \bH^L_S$. Indeed,
$$\hofib(i_{\bbar{P}}^*\FF_{\bbar{L}} \to i_{\bbar{P}}^*\FF_{\bbar{S}}) = \hofib(i_{\bbar{P}}^* i_{\bbar{S}}^* \FF_{\bbar{L}} \to i_{\bbar{P}}^* \FF_{\bbar{S}}) = i_{\bbar{P}}^* \hofib(i_{\bbar{S}}^*\FF_{\bbar{L}} \to \FF_{\bbar{S}}) = i_{\bbar{P}}^* \bH^L_S$$
The first statement of the Lemma follows. The second statement follows by replacing $\bbar{P}$ by $P$ throughout.
\end{proof}

\subsubsection{Gluing sheaves across strata that intersect along their closure}\label{sec-glueacross}
In the proof of Theorem \ref{thm-hofibsflexg} below, we shall need to glue sheaves over different strata $S_1, S_2$ to obtain a sheaf over $\bbar{S_1} \cup \bbar{S_2}$ when $\bbar{S_1} \cap \bbar{S_2} \neq \emptyset$. Note that Definition \ref{def-sssheaf} does not directly furnish such a sheaf. For concreteness, and to illustrate the construction, suppose $S_1, S_2$ have height one (Definition 
\ref{def-Idec}), and $\bbar{S_1} \cap \bbar{S_2} = P$, where $P$ has height zero. Thus, $P = \bbar{P}$. Then, there are two extrinsic sheaves $i_P^* \FF_{\bbar{S_1}}$ and $i_P^* \FF_{\bbar{S_2}}$ over $P$, and restriction maps 
$$\res^{S_j}_P : i^*_P \mathcal{F}_{\bbar{S_j}} \to \mathcal{F}_P, \ j = 1, 2.$$
{There exists a {\emph{single, continuous}} sheaf $ \FF_{\bbar{S_1} \cup \bbar{S_2}}$ such that:
\begin{enumerate}
	\item Restricted to a germinal neighborhood of $P$ in $\bbar{S_1} \cup \bbar{S_2}$, $\FF_{\bbar{S_1} \cup \bbar{S_2}}$ is the fibered product of $i_P^*\FF_{\bbar{S_1}}$ and $i_P^*\FF_{\bbar{S_2}}$ over $\FF_P$. 
	(Here $P$ is equipped with the subspace topology from $X$.)
	In other words,
$$i^*_P \FF_{\bbar{S_1} \cup \bbar{S_2}} = i_P^*\FF_{\bbar{S_1}} \times_{\FF_P} i_P^*\FF_{\bbar{S_2}}.$$
	\item Restricted to $S_1$ (resp. $S_2$), $\FF_{\bbar{S_1} \cup \bbar{S_2}}$ is $\FF_{S_1}$ (resp. $\FF_{S_2})$. In other words, 
$$i^*_{S_1} \FF_{\bbar{S_1} \cup \bbar{S_2}} = \FF_{S_1}, \ i^*_{S_2}\FF_{\bbar{S_1} \cup \bbar{S_2}} = \FF_{S_2}.$$
\end{enumerate}
We may construct the sheaf explicitly as follows. Let $U \subset \bbar{S_1} \cup \bbar{S}_2$ be open.
\begin{enumerate}
\item If $U \subset S_j$ for $j = 1, 2$, we define,
$$\mathcal{F}_{\bbar{S_1}\cup\bbar{S_2}}(U) = \mathcal{F}_{\bbar{S_j}}(U).$$ 
\item Otherwise, $U \cap P \neq \emptyset$. Then, we define 
$$\mathcal{F}_{\bbar{S_1}\cup\bbar{S_2}}(U) = \mathcal{F}_{\bbar{S_1}}(U \cap \bbar{S_1}) \times_{\mathcal{F}_P(U \cap P)} \mathcal{F}_{\bbar{S_2}}(U \cap \bbar{S_2}).$$
\end{enumerate}
It remains to define the restriction maps for $\mathcal{F}_{\bbar{S_1} \cup \bbar{S_2}}$. Let $U, V \subset \bbar{S_1} \cup \bbar{S_2}$ be a pair of open sets such that $U \subset V$.
\begin{enumerate}
\item Suppose $V \subset S_j$ for $j = 1, 2$. Then, $U \subset S_j$. The restriction map is simply defined to be the restriction map $\FF_{\bbar{S_j}}(V) \to \FF_{\bbar{S_j}}(U)$ for the sheaf $\FF_{\bbar{S_j}}$. 
\item Suppose $V \cap P \neq \emptyset$ and $U \cap P \neq \emptyset$. Then, we have restriction maps $\FF_{\bbar{S_j}}(V \cap \bbar{S_j}) \to \FF_{\bbar{S_j}}(U \cap \bbar{S_j})$ and $\FF_P(V \cap P) \to \FF_P(U \cap P)$. By naturality of fibered products, we obtain a map,
$$\FF_{\bbar{S_1} \cup \bbar{S_2}}(V) \to \FF_{\bbar{S_1} \cup \bbar{S_2}}(U).$$ 
We define this to be the restriction map, in this case.
\item Suppose $V \cap P \neq \emptyset$ and $U \subset S_j$ for $j = 1, 2$. Then, $U \subset V \cap S_j$. The restriction map is then defined to be the composition,
$$\FF_{\bbar{S_1} \cup \bbar{S_2}}(V) \to \FF_{\bbar{S_j}}(V \cap \bbar{S_j}) \to \FF_{\bbar{S_j}}(U).$$
Here, the first arrow is projection of the fibered product to one of the constituent factors, and the second arrow is the restriction map for the sheaf $\FF_{\bbar{S_j}}$. 
\end{enumerate}
We give a general construction in the following definition.}

\begin{defn}\label{defn-cs}
{Let $(X, \Sigma)$ be a stratified space, $\FF = \{\FF_{\bbar{S}} : S \in \Sigma\}$ be a stratified continuous sheaf on $(X, \Sigma)$ and $Y \subset X$ be a closed stratified subspace. We define a single, continuous sheaf $\FF_Y$ on $Y$ as follows.}

{For an open subset $U \subset Y$, we define a poset in the stratified site $\str(X, \Sigma)$ (Definition \ref{def-stratfdsite}), given by the collection of objects
$$\DD_Y(U) = \{U \cap \bbar{S} : S \in \Sigma, S \subset Y, S \cap U \neq \emptyset\}.$$
Notice that for a pair of objects $U \cap \bbar{S}, U \cap \bbar{L}$ of $\DD_Y(U)$, if $S < L$, there is a restriction map $\mathcal{F}_{\bbar{L}}(U \cap \bbar{L}) \to \mathcal{F}_{\bbar{S}}(U \cap \bbar{S})$ from the definition of stratified continuous sheaf as a sheaf over the stratified site (Definition \ref{def-sssheaf}). We define,
$$\mathcal{F}_Y(U) := \varprojlim_{\DD_Y(U)}\FF_{\bbar{S}}(U \cap \bbar{S}).$$
For a pair of open subsets $U, V \subset Y$ such that $U \subset V$, we define another poset in the stratified site $\str(X, \Sigma)$, given by the collection of objects
$$\DD_Y(V, U) = \{V \cap \bbar{S} : S \in \Sigma, S \subset Y, S \cap U \neq \emptyset\}.$$
The restriction map $\FF_Y(V) \to \FF_Y(U)$ is defined by the composition,
$$\FF_Y(V) = \varprojlim_{\DD_Y(V)} \FF_{\bbar{S}}(V \cap \bbar{S}) \to \varprojlim_{\DD_Y(V, U)} \FF_{\bbar{S}}(V \cap \bbar{S}) \to \varprojlim_{\DD_Y(U)} \FF_{\bbar{S}}(U \cap \bbar{S}) = \FF_Y(U).$$
The first arrow is obtained using the universal property of limits, as $\DD_Y(V, U) \subset \DD_Y(V)$ is a sub-poset. The second arrow is obtained by collecting together the restriction maps $\FF_{\bbar{S}}(V \cap \bbar{S}) \to \FF_{\bbar{S}}(U \cap \bbar{S})$, using the naturality of limits. Here, $V \cap \bbar{S}$ is an object of $\DD_Y(V, U)$, and the map $\FF_{\bbar{S}}(V \cap \bbar{S}) \to \FF_{\bbar{S}}(U \cap \bbar{S})$ is the restriction map for the continuous sheaf $\FF_{\bbar{S}}$ on $\bbar{S}$.}

{We refer to the continuous sheaf $\FF_Y$ as the \emph{consolidated sheaf} on $Y$ associated to the stratified continuous sheaf $\FF$ on $(X, \Sigma)$.}
\end{defn}

{The stalks of the consolidated sheaf $\mathcal{F}_Y$ admit a simple description. Let $y \in Y$ and $S_0 \subset Y$ be the unique stratum of $Y$ containing $y$. Also, let $S_1, \cdots, S_k$ be the finitely many strata of $Y$ different from $S_0$ such that $S_0 \subset \bbar{S_i}$. Notice that these are exactly the strata of $Y$ which contains $y$ in the boundary, by Condition (3) of Definition \ref{def-Idec}. Then we have the intrinsic stalk $(\mathcal{F}_{S_0})_y$ at $y$, as well as the extrinsic stalks $(i_{S_j}^*\FF_{\bbar{S_j}})_y$, $j \in \{1, \cdots, k\}$ at $y$. Whenever $S_\ell < S_j$, $j, \ell \in \{0, \cdots, k\}$, there is a restriction morphism between the stalks,
$$\mathrm{res}^j_\ell : (i_{S_j}^*\FF_{\bbar{S_j}})_y \to (i_{S_\ell}^*\FF_{\bbar{S_\ell}})_y.$$ 
This gives a diagram of quasitopological spaces. The limit of this diagram is the stalk $(\FF_Y)_y$ of the consolidated sheaf at $y$.}

{We also note that if $(X, \Sigma)$ is a stratified space with a unique top-dimensional stratum $L$, such that $\bbar{L} = X$, then the consolidated sheaf $\FF_X$ is simply the constituent sheaf $\FF_{\bbar{L}}$ of the stratified continuous sheaf $\FF = \{\FF_{\bbar{S}} : S \in \Sigma\}$.}

%%%%%
\begin{comment}
{Let  $Y$  be a closed stratified subspace of  $X$. Then we construct a single continuous sheaf $\FF_{Y}$ on $Y$ such that for any $y \in Y$, the stalk $(\FF_Y)_y$ of $\FF_Y$ is given as follows. Let $S=S_0$ be the unique stratum  of $Y$ containing $y$. Also, let $S_1, \cdots S_k$ be the finitely many strata of $Y$ different from $S_0$ such that $S_0 \subset \bbar{S_i}$. Let
$\res^j_{0}$, $j=0, \dots, k$ be the restriction maps from $\bbar{S_j}$ to $S_0$.
Then $(\FF_Y)_y$ is given by the fiber-product of the restriction maps 
$\res^j_{0}: i_{S_0}^* (\FF_{\bbar{S_j}})_y \to (\FF_{\bbar{S_0}})_y$ at the level of stalks at $y$. We refer to the sheaf $\FF_{Y}$ as the \emph{consolidated sheaf} on $Y$
associated with the stratified sheaf $\FF$.}

{We draw attention to the fact that the indexing in Definition \ref{defn-cs} above is for $j=0, \dots, k$ instead of $j=1, \dots, k$. This is, in particular, to take care of the possibility that $k=0$.
Note also that
Definition \ref{defn-cs} furnishes a {\emph{single} well-defined continuous sheaf} $\FF_Y$ for any  closed stratified subspace $Y$ of $X$.}
\end{comment}
%%%%%
 
\begin{rmk}
\rm{
{Let $\FF = \{\FF_{\bbar{S}} : S \in \Sigma\}$ be a stratified continuous sheaf on a stratified space $(X, \Sigma)$. Taking $Y = X$ in Definition \ref{defn-cs}, we obtain a consolidated sheaf $\FF_X$ over the underlying topological space $X$ of the stratified space. For any closed stratified subspace $Y \subset X$, we also have a consolidated sheaf $\FF_Y$ over $Y$. For any open set $U \subset X$ intersecting $Y$, observe that $\DD_Y(U \cap Y)$ is a sub-poset of $\DD_X(U)$. Therefore, by the universal property of limits, there is a map
\begin{equation}\label{eq-resXY}\FF_X(U) = \varprojlim_{\DD_X(U)}\FF_{\bbar{S}}(U \cap \bbar{S}) \to \varprojlim_{\DD_Y(U \cap Y)} \FF_{\bbar{S}}(U \cap \bbar{S}) = \FF_Y(U \cap Y).\end{equation}
Let $V \subset Y$ be an open subset. Consider the collection of open subsets $U \subset X$ such that $U \cap Y = V$. Taking a direct limit $\varinjlim_U$ in Equation \ref{eq-resXY} as $U$ shrinks to $V$, we obtain a map $i_Y^*\FF_X(V) \to \FF_Y(V)$. As this map is natural in $V$, we thus have a morphism of continuous sheaves $i_Y^*\FF_X \to \FF_Y$ on $Y$.
}
}
\end{rmk}

We note for use below, the following corollary of the proof of Proposition \ref{prop-sheafh}. 

\begin{cor}\label{cor-3glue}
	Let $(X, \Sigma)$ be a stratified space with {unique top} dimensional stratum $L$, {so that $X=\bbar L$. Let $Y=\partial L$. Let ${\FF= \{\FF_{\bbar{S}}: S \in \Sigma\}}$ be a stratified sheaf on $X$. Let $\FF_{\bbar{L}}=\FF_X$, and let $\FF_Y$ denote the consolidated sheaf on $Y$. } Suppose that
	\begin{enumerate}
		\item $\FF_Y$ satisfies the parametric $h-$principle. 
		\item ${\bH^X_Y := \hofib(i_Y^* \FF_X \to \FF_Y)}$ satisfies the parametric $h-$principle.
		\item {$\iota^*_L\FF_X$} is flexible.
	\end{enumerate}
	Then {(the single continuous sheaf)} $\FF_X$ satisfies the parametric $h-$principle.
\end{cor}

\begin{proof}
	Lemma \ref{lem-sessheafflex} along with the first two conditions ensure that {$i_{Y}^* \FF_X$} 
	satisfies the parametric $h-$principle.  Now, flexibility of {$\iota^*_L \FF_X$} allows us to co-glue 
	$\iota^*_L\FF_X$  and  {$i_{Y}^* \FF_X$}  as in the proof of Proposition \ref{prop-sheafh} to conclude that  $\FF_X$ satisfies the parametric $h-$principle.
\end{proof}

{ Note  that $\partial L$ may have many strata in Corollary \ref{cor-3glue}. This is the difference with the content of Proposition \ref{prop-sheafh}.}

\subsubsection{$h-$principle from stratumwise conditions}\label{sec-main}
We are now in a position to prove the main theorem of this Section. It says roughly that flexibility of homotopy fibers for pairs of strata guarantees parametric $h-$principle for the stratified sheaf provided the latter is stratumwise flexible. We shall first prove this assuming a  total order of strata.
{As mentioned at the beginning of Section \ref{sec-sheafh}, Corollary \ref{cor-3glue} can be applied directly and inductively to establish Theorem \ref{thm-hofibsflexg} below  when the strata are totally ordered. However, a second induction has to be applied when 
	the strata are only partially ordered. This second induction is performed in the general case of the proof of Theorem \ref{thm-hofibsflexg} below.}

\begin{theorem}\label{thm-hofibsflexg} 
	Let $(X, \Sigma)$ be a stratified space and $\FF$ be a stratified sheaf on $X$ such that 
	\begin{enumerate}
	\item $\FF$ is stratumwise flexible, i.e.\ for every stratum $S \in \Sigma$, $\FF_S := i^*_{S}\FF_{\bbar{S}}$ is flexible on the stratum $S$.
	\item  $\FF$ is infinitesimally flexible across strata, i.e.\ for all $S<L$, the open homotopy fiber sheaf  $ {\HH_S^L}$ is flexible.
	\end{enumerate}
Then  the stratified sheaf $ \FF$ satisfies the parametric $h-$principle. 
\end{theorem}

\begin{proof}[Proof of Theorem \ref{thm-hofibsflexg} under the assumption of total ordering:] We first assume that the strata of $X$ are totally-ordered. Thus,
	\begin{enumerate}
		\item $(X, \Sigma)$ is a stratified space with strata $\Sigma = \{S_1 > \cdots > S_n\}$ ordered so that $S_i < S_j$ {(i.e., $S_i \subset \bbar{S_j}$)} if $i > j$.
		\item $\FF = \{\FF_{\bbar{i}} : 1 \leq i \leq n\}$ is a stratified continuous sheaf on $X$, where $\mathcal{F}_{\bbar{i}}$ is a continuous sheaf on the stratum-closure $\bbar{S}_i$.
	\end{enumerate}  
	For any pair of indices $i > j$, let ${\bH_i^j} = \hofib (i_{\bbar{S}_i}^*\FF_{\bbar{j}}\to  \FF_{\bbar{i}})$, and $\HH_i^j = i^*_{S_i} \bH_i^j$ be the closed and open homotopy fibers, respectively, as in Definition \ref{def-hh}.
	Then the hypotheses of the theorem translate to the following:
	\begin{enumerate}
		\item For all $j$, the sheaf  $\FF_j := i^*_{S_j}\FF_{\bbar{j}}$ is flexible on the stratum $S_j$.
		\item For all $i > j$, the open homotopy fiber sheaf  $ {\HH_i^j}$ is flexible.
	\end{enumerate}
We shall show that	 the sheaves  $ \FF_{\bbar{j}} $ satisfy the parametric $h-$principle. 
We proceed by induction on $n$. For $n=1$, i.e.\ for manifolds, this is due to Gromov (see the Main Theorem on pg.\ 76 of \cite{Gromov_PDR}). 

By induction, we can assume that $\FF_{\bbar{2}}$ satisfies the parametric $h-$principle (using the chain of $n-1$ strata $S_2>\cdots>S_n$). We first prove that the closed homotopy fiber sheaves $\bH^j_i$, $i > j$ satisfy the parametric $h-$principle.

\begin{claim}\label{claim-bhhprin}
With notation and hypothesis as in Theorem \ref{thm-hofibsflexg}, $\bH^j_i$ 
satisfies the parametric $h-$principle for all $i > j$.
\end{claim}

\begin{proof}[Proof of Claim \ref{claim-bhhprin}] We prove this by downward induction on $i$. First, let $i=n$.
Since $S_n$ is the deepest stratum and therefore $S_n =\bbar{S}_n$, $\bH^j_n = \HH^j_n$, and $\HH^j_n$ is flexible by hypothesis. Hence, by \cite[p. 76]{Gromov_PDR}, $\bH^j_n =\HH^j_n$  satisfies the parametric $h-$principle for all $j <n$.

Suppose now that the claim is true for $i = k+1 \leq n$ and all $j < k+1$. We shall show that $\bH^j_k$ satisfies the parametric $h-$principle as a sheaf on $\bbar{S}_k$, for all $j < k$. Observe that $i^*_{S_k} \bH^j_k = \HH^j_k$ is an open homotopy fiber {sheaf}, hence flexible by hypothesis. Next, by Lemma \ref{lem-bendright},
$$i^*_{\bbar{S}_{k+1}} \bH^j_k \simeq \hofib \left ( \bH^j_{k+1} \to \bH^k_{k+1} \right ).$$
By the induction hypothesis, $\bH^j_{k+1}$ and $\bH^k_{k+1}$ satisfy the parametric $h-$principle. Therefore, by Lemma \ref{lem-sessheafflex} and Remark \ref{rmk-sessheafflex}, we conclude that the sheaf $i^*_{\bbar{S}_{k+1}} \bH^j_k$ satisfies the parametric $h-$principle as well. Thus, by Corollary \ref{cor-3glue}, $\bH^j_k$ satisfies the parametric $h-$principle. \end{proof}

In particular $\bH^1_2$ satisfies the parametric $h-$principle. By Corollary \ref{cor-3glue}, we conclude that $\mathcal{F}_{\bbar{1}}$ satisfies the parametric $h-$principle, as required.\end{proof}

\begin{proof}[Proof of Theorem \ref{thm-hofibsflexg}, the general case:] 
	We now show how to remove the hypothesis of total orderability of strata in the proof of Theorem \ref{thm-hofibsflexg}. We shall proceed by induction on height and apply Corollary \ref{cor-3glue}. 	
	We use the notation of Corollary \ref{cor-3glue}. The proof of Corollary  \ref{cor-3glue} in fact shows that $\FF_X$ satisfies the parametric $h-$principle 
provided that for \emph{any} maximal (with respect to height) stratum  $L$, and $Y=\partial L$, we can show that $\FF_Y$, and $ \bH^X_Y$ satisfy the parametric $h-$principle. The gluing needs flexibility  of $\FF_L$ in Corollary \ref{cor-3glue},
but here our concern will be with the parametric $h-$principle. Let $m+1$ denote the height of $L$, so that $Y$ has height $m$. Assume by induction that 
\begin{enumerate}
\item $\FF_{A}$ satisfies the parametric $h-$principle for any stratified subspace $A \subset X$ of height $< m$.
\item Further, for any stratified subspace $B \subset X$ with $B>A$, $ \bH^B_A$ satisfies the parametric $h-$principle.
\end{enumerate}

Set $Y = \cup_1^k Y_i$, where $Y_i$'s denote the \emph{closures} of the maximal  (with respect to height)  strata of $Y$. By induction on $k$, it suffices to prove
the theorem for $k=2$. This is because the proof in the totally-ordered case (coupled with the above inductive hypothesis) allows us to conclude that the statement is true for $k=1$, and we can take the union $ \cup_1^{k-1} Y_i$ as a single stratified space in the inductive step. Hence, assume that $Y=Y_1 \cup Y_2$. Also,
note that $Z=Y_1 \cap Y_2$ has less height than (at least one of) $Y_1, Y_2$.

We refer to the sheaf $\FF_Y$ (constructed from $\FF_{Y_1}$ and $\FF_{Y_2}$ as in Section \ref{sec-glueacross}) as the \emph{intrinsic sheaf on $Y$}.
Similarly, we refer to $i_Y^* \FF_X$ as the \emph{extrinsic sheaf on $Y$}. It suffices, by Corollary \ref{cor-3glue} to prove that
\begin{enumerate}
\item The intrinsic sheaf $\FF_Y$ satisfies the parametric $h-$principle, and
\item $ \bH^X_Y=\hofib(i_Y^* \FF_X \to \FF_Y)$  satisfies the parametric $h-$principle.\\
\end{enumerate}

\noindent {\bf $\FF_Y$ satisfies the parametric $h-$principle:}\\ To prove that $\FF_Y$ satisfies the parametric $h-$principle, 
it suffices by stratumwise flexibility of $\FF$  to show that 
$$\AAA = i_Z^* \FF_{Y_1} \times_{\FF_Z} i_Z^*\FF_{Y_2}$$
satisfies the parametric $h-$principle. 
(Recall that $Z=Y_1 \cap Y_2$. This reduction is exactly as in the proof Proposition \ref{prop-sheafh}).

By Lemma \ref{lem-hofib-fiberpdkt}, $\hofib(\AAA \to i_Z^* \FF_{Y_1})$ is homotopy equivalent to 
 $\hofib( i_Z^* \FF_{Y_2} \to \FF_Z)$. By the inductive hypothesis on height  (of $Z$),
 $\hofib( i_Z^* \FF_{Y_2} \to \FF_Z)$ 
 satisfies the parametric $h-$principle. Hence, so does 
 $\hofib(\AAA \to i_Z^* \FF_{Y_1})$. Again, by 
 the inductive hypothesis on height (of $Z$), 
 $i_Z^* \FF_{Y_1}$ satisfies the parametric $h-$principle. 
Hence, by Lemma \ref{lem-sessheafflex}, $\AAA$  satisfies the parametric $h-$principle. \\

\noindent {\bf $ \bH^X_Y$ satisfies the parametric $h-$principle:}\\  We know that the following sheaves satisfy the parametric $h-$principle:
\begin{enumerate}
	\item $\FF_Z$ and $\hofib(i_Z^* \FF_X \to \FF_Z)$   (by the inductive hypothesis
	 on closures of strata on lower height, and closed homotopy fibers on closures of strata on lower height)
	\item $\hofib(i_{Y_1}^* \FF_X \to \FF_{Y_1})$ (from the proof in the totally-ordered case)
	\item $\hofib(i_{Y_2}^* \FF_X \to \FF_{Y_2})$ (from the proof in the totally-ordered case)
\end{enumerate}
It suffices (as in the proof of Proposition \ref{prop-sheafh}) to show that $\hofib(i_Z^* \FF_X \to \AAA)$ satisfies the parametric $h-$principle.

Note now that by Lemma \ref{lem-sessheafflex}, $\hofib(\AAA\to \FF_Z)$ 
satisfies the parametric $h-$principle (since we have already shown that $\AAA$ does, and  the inductive hypothesis gives that $\FF_Z$ does).
Further, by the inductive hypothesis applied to the lower depth stratum closure $Z$,
the homotopy fiber $\hofib(i_Z^* \FF_X \to \FF_Z)$ satisfies the parametric $h-$principle. By Lemma \ref{lem-bendright}, 
$\hofib(i_Z^* \FF_X \to \AAA)$  is homotopy equivalent to 
$$\hofib\big(\hofib(i_Z^* \FF_X \to \FF_Z) \to \hofib(\AAA\to \FF_Z)\big).$$
By  Lemma \ref{lem-sessheafflex}, this satisfies the parametric $h-$principle. 
Hence, so does the sheaf $\hofib(i_Z^* \FF_X \to \AAA)$.
\end{proof}

\begin{eg}\rm{
{Recall from Example \ref{eg-stratflexnotflex} the stratified continuous sheaf $\FF = \{\FF_{\bbar{S}}, \FF_{\bbar{L}}\}$ on $X = [0, \infty)$ with stratification $\Sigma = \{S, L\}$, where $S = \{0\}$ and $L = (0, \infty)$. Since $\FF_S$ and $\FF_L$ are constant sheaves, they are both flexible. The open homotopy fiber sheaf $\HH^L_S$ is also a constant sheaf, valued in the space $\hofib(f)$. Therefore, $\HH^L_S$ is flexible as well. Thus, $\FF$ is both stratumwise flexible and infinitesimally flexible across strata. In other words, it satisfies the hypotheses of Theorem \ref{thm-hofibsflexg}. Consequently, the stratified continuous sheaf $\FF$ (in particular, its constituent continuous sheaf $\FF_{\bbar{L}}$) satisfies the parametric $h$--principle, despite $\FF_{\bbar{L}}$ not being flexible.}
}\end{eg}

\subsubsection{Flexibility versus h-principle}\label{sec-ctreg}  The aim of the example below is to illustrate the necessity of flexibility of $\FF$ on the top stratum
in Theorem \ref{thm-hofibsflexg} and Proposition \ref{prop-sheafh}. It emphasizes that 
it is not enough to assume that $\FF$ satisfies the parametric $h-$principle  on the top stratum.

Let $X=M$ be {a closed orientable 3-manifold}, and $F \subset M$ be an embedded closed orientable surface. Stratify $X$ with two strata $\Sigma = \{S, L\}$, where $L= M \setminus F$ and $S=F$. Thus, $\bbar{S} = S$ and $\bbar{L} = M$. Let $N$ be another 3-manifold not covered by $M$.

Consider the stratified sheaf $\FF$ on $(X, \Sigma)$ defined as follows:
\begin{enumerate}
	\item $\FF_{{L}}(U) = \imm (U,N)$ for $U \subset \bbar{L}$ open,
	\item $\FF_{{S}}(V) = \imm (V,N)$ for $V \subset S$ open.
\end{enumerate}

Then the restriction map $\res: i_S^* \FF_{{L}} \to \FF_{{S}}$ simply forgets the  normal bundle to $S$ in $M$.
We note the following:
\begin{enumerate}
	\item $\FF_{{L}}|L$ satisfies the parametric $h-$principle \cite[p. 79]{Gromov_PDR} as $L$ is open.
	\item $\FF_{{L}}|L$ is not flexible (since the dimensions of $M, N$ coincide, and both are compact).
	\item $\FF_{{S}}$ is flexible \cite[p. 79]{Gromov_PDR}.
	\item {$\hofib(\res)$ is flexible. Indeed, let us fix an immersion $\psi \in \FF_S(S) = \imm(S, N)$. Sections of $i_S^*\FF_{\bbar{L}}$ over an open set $U \subset S$ are simply immersions of $U$ in $N$ together with a transverse vector field. Consequently, a section of $\hofib(\res; \psi)$ over $U \subset S$ consists of a pair $(\{\psi_t\}, X)$ where $\{\psi_t : U \to N : 0 \leq t \leq 1\}$ is a regular homotopy of immersions with $\psi_1 = \psi|U$ and $X$ is a transverse vector field to $\psi_0$. Firstly, if we forget the data of the transverse vector field, this is the sheaf of paths of immersions of $S$ in $N$, based at $\psi$. This is a flexible sheaf, as $\imm(-, N)$ is flexible on $S$ (see, Item $(3)$). Secondly, given an immersion $f : S \to N$ and a compact disc $D^2 \subset S$, a vector field defined on $D^2$ transverse to $f$ can always be extended to a transverse vector field on all of $S$. Combining these facts give us the required solution to the homotopy lifting problem.}
\end{enumerate}

\begin{prop}\label{prop-3mfld} With $M, \FF, L, S$ as above,
	$\FF$ does not satisfy the parametric $h-$principle.
\end{prop}

\begin{proof}
	The Gromov diagonal construction applied to $\FF$ gives $\FF^\bullet$ homotopy equivalent to a sheaf $\GG$ given
	as follows. $\GG(U)$ consists of bundle maps $\psi:TU\to TN$ covering smooth maps $\psi_0: U \to  N$
	such that $\psi|T_xU$ is an isomorphism on every tangent space $T_xU$. In particular, 
	$\GG(M)$ consists of bundle maps $\psi:TM\to TN$ covering smooth maps $\psi_0: M \to  N$
	such that $\psi|T_xM$ is an isomorphism on every tangent space $T_xM$. Since $M$  is an orientable 3-manifold, it is parallelizable, so that $TM=M\times \R^3$.
	Therefore $\GG(M)$ contains $\psi$ covering a \emph{constant map} $\psi_0$. In particular,  $\GG(M)$ is non-empty.
	
	On the other hand, $\FF(M)$ consists of immersions from $M$  to $N$. Since $M, N$ have the same dimension, $\FF(M)$ consists of covering maps. Since  $N$  is not covered by $M$, $\FF$ is empty.
	Hence, $\FF$ does not satisfy the  parametric $h-$principle.
\end{proof}

\subsection{Microflexibility of stratified sheaves}\label{sec-micro} For the purposes of this subsection 
$X$ will denote a stratified space, and $V$  a manifold. $\FF$ will  be a (stratified) continuous sheaf.
The aim of this subsection is to generalize the following theorem of Gromov and its consequence below to stratified spaces and stratified sheaves over them. 
\begin{theorem}\cite[p. 78]{Gromov_PDR}\label{thm-gromovmicro2flex} 
	Let $Y = V \times  \R$,  and let $\Pi: Y \to V$ denote the projection
	onto the first factor. Let $\diff( V, \Pi)$ be the group of diffeomorphisms of $Y$ commuting with $\Pi$, where $V$ is identified with $V \times \{0\}$.
	Let $\FF$ be a microflexible (continuous) sheaf over $Y$ invariant under $\diff( V, \Pi)$.
	Then the restriction $\FF \vert V (= V \times \{0\})$ is a flexible sheaf over $V (= V \times \{0\})$.
\end{theorem}

The principal consequence is the following 
flexibility theorem for $\diff-$invariant sheaves. 

\begin{theorem}\cite[p. 78-79]{Gromov_PDR}\label{thm-gromovmicro2flex2}
	Let $\FF$ be a microflexible $\diff(V)-$invariant (continuous)  sheaf
	over a manifold $V$. Then the restriction to an arbitrary piecewise smooth polyhedron
	$K \subset V$ of positive codimension, $\FF|K$, is a flexible sheaf over $K$.
\end{theorem}

Recall that	a stratified (continuous) sheaf $\FF$ over a stratified space $X$ is \emph{stratumwise microflexible} if
	for every stratum $S$ of $X$, $\FF|S$ is  microflexible.

The main theorems of this section are now given below. 

\begin{theorem}\label{thm-micro2flexs}
	Let $Y = X \times  \R$ equipped with the product stratification,  and let $\Pi: Y \to X$ denote the projection
	onto the first factor. Let $\FF$ be a stratumwise microflexible  continuous sheaf over $Y$ invariant under $\sdiff( X, \Pi)$.
	Further, suppose that $\FF$ is infinitesimally microflexible across strata, i.e.\ for all strata $S<L$ of $Y$, $\HH^L_S$ (cf.\ Definition \ref{def-hh}) is microflexible.
	Then the restriction $\FF \vert X (= X \times \{0\})$ is a stratified sheaf over $X (= X \times \{0\})$ satisfying the parametric $h-$principle.
\end{theorem}

\begin{proof} Note first that $\HH^L_S$ is invariant under $\sdiff( X, \Pi)$ by Lemma \ref{lem-trivialhofibbdl}.
By hypothesis, {for all strata $S < L$ of $Y$,}
	\begin{enumerate}
		\item $\FF_L$ is microflexible. 
		\item $\HH^L_S$ is microflexible. 
	\end{enumerate}
	The strata {$S, L$ of $Y$ are of the form $S_X \times \R, L_X \times \R$, where $L_X = L\cap X, S_X = S \cap X$ are strata of $X$}.
	Note that {$S, L$ are manifolds}. Invariance of $\FF$ and  $\HH^L_S$  under $\sdiff( X, \Pi)$ implies $\diff (L_X \times \R, \Pi)-$invariance of $\FF_L$ and {$\diff (S_X \times \R, \Pi)-$invariance of $\HH^L_S$}. It follows from Theorem \ref{thm-gromovmicro2flex} that {the sheaves $\FF_L\vert{L_X \times \{0\}}$  and $\HH^L_S\vert {S_X \times \{0\}}$  are flexible.}
	Hence, by Theorem \ref{thm-hofibsflexg}, {$\FF \vert (X \times \{0\})$} satisfies the parametric $h-$principle.
\end{proof}

\begin{defn}
	A stratified subspace $K<X$ is said to be of positive codimension if for every stratum $S$ of $X$,
	$K\cap S$ has  positive codimension in $S$.
\end{defn}

\begin{theorem}\label{thm-micro2flexs2}
	Let $\FF$ be a stratumwise microflexible $\sdiff-$invariant stratified (continuous)  sheaf
	over $X$. Further, suppose that $\FF$ is infinitesimally microflexible across strata, i.e.\ the open homotopy fiber sheaves $\HH^L_S$ (Definition \ref{def-hh})  are 
 microflexible.
	Then the restriction $\FF|K$ to a stratified subspace
	$K \subset X$ of positive codimension  satisfies the parametric $h-$principle.
	
\end{theorem}

\begin{proof}
	Since $K$ is of positive codimension in $X$, for every $k\in K$, there is an open neighborhood $U_k$ of $k$ in $K$ such that $U_k \times (-1,1)$ embeds in $X$. The theorem now follows 
	from Theorem \ref{thm-micro2flexs} since locally flexible sheaves are flexible (by Gromov's localization lemma \cite[p. 79]{Gromov_PDR}).
\end{proof}

\begin{rmk}
	Theorems \ref{thm-micro2flexs} and \ref{thm-micro2flexs2} provide examples of  how to translate a positive codimension stratumwise microflexibility hypothesis into an 
	$h-$principle conclusion.
\end{rmk}

\section{{The Gromov diagonal  normal construction with smooth jets}}\label{sec-formalfnstrat}
We specialize the Gromov diagonal normal sheaf construction of  $\FF^\bullet$  in Definition \ref{def-formalfn} to the  sheaf of sections of a stratified bundle $E$ over a smooth stratified space $X$. Even in the manifold setup, an explicit connection between Gromov's $\FF^\bullet$ construction and the use of jets in \cite{em-book} is a little difficult to find. Hence, we provide a detailed treatment below. Remark \ref{rmk-formalasgerms}, which gives an explicit description of $\FF^\bullet$ will allow us to formalize this. It will turn out that
in the stratified context, $\FF^\bullet$ admits an inductive description up to homotopy in terms of two constituent sheaves:
\begin{enumerate}
	\item A purely topological {sheaf of germs of sections} (see Definition \ref{def-conegerms} below).
	\item A {sheaf of smooth jets} $\JJ^r_E$ when $E$ is a smooth bundle over a manifold (see Proposition \ref{prop-formalfnmflds} below).
\end{enumerate}
The aim of this section is to describe this structure of $\FF^\bullet$. In the process we answer Sullivan's question \ref{qn-sullivan}.

\subsection{Tangent microbundles on stratified spaces}\label{sec-microb}
The stratified tangent bundle (Definition \ref{str-tbl}) will turn out to be a stratified subbundle of the tangent microbundle to $X$ (Definition \ref{def-mic}).

Note that for $X$ a manifold, the tangent microbundle is {germinally} equivalent to the tangent bundle $TX$, as it coincides with the normal bundle to the diagonal $\diag(X) \subset  X \times X$.
For each stratum $S$ of $X$, $TS$ will thus refer to the tangent microbundle of the manifold $S$, i.e.\ it may be identified canonically with the germ of the zero-section from $S$ to the usual (manifold) tangent bundle of $S$.

The tangent microbundle to a stratified space $X$,  denoted by $tX$ henceforth, turns out to be a {stratumwise} bundle (see Definition \ref{def-stratumwisebdl}).
We provide an explicit 
description of $tX$  in terms of  {the} local structure {of $X$}.

Let ${tX :=(U, X, p)}$ be the tangent microbundle of $X$. For $x \in X$, consider the fiber $p^{-1}(x)$. This is a germ $U_x$ of a neighborhood of $x$ in $X$. Let $S$ be the unique stratum of $X$ containing $x$. Then there is {an} identification of $U_x$ with $W \times cA$ where $W$ is the germ of a ball around $x$ in $S$ and $cA$ is the germ of a cone on the link $A$ of $S$ in $X$, { with cone point $x$}. Thus, 
$${(p^{-1}(x), \{(x,x)\}) \cong (U_x, \{x\}) \cong (W, \{x\}) \times (cA, \{x\})}$$
as {germs of spaces}.
The bundle over $S$ whose fibers are (germs of) the cones $cA$ will be denoted as $NS$, and referred to as the normal cone microbundle of $S$ in $X$.

Hence the restriction of $(U, p)$ to $S$, i.e.\ $(p^{-1}(S) \cap U, S, p)$ is germinally equivalent to the direct sum, {i.e.\ fiberwise product} of the microbundles: 
$${(p^{-1}(S) \cap U, S, p) \cong (TS, S, p) \oplus (NS, S, p),}$$ 
where $TS$ is the tangent microbundle of $S$  and $(NS, S, p)$ is the  normal cone microbundle of $S$ in $X$. {This demonstrates that $tX$ is a stratumwise fiber bundle over $X$ according to Definition \ref{def-stratumwisebdl}, where the fiber over a point $x \in S$ of any particular stratum $S$ of $X$ is given by {$T_x S \times cA_S(x)$} where $A_S$ denotes the link of $S$ in $X$, and $cA_S(x)$ denotes the normal cone of $S$ in $X$ at the point $x$. We next describe a filtration of $tX$, which induces the canonical filtration of any normal cone $cA_S$ by stratum-closures.}\\

\noindent {\bf Relative tangent microbundle:}
Next, suppose that $L$ is a stratum of $X$ and let $Y= \bbar{L}$ be the stratum closure of $L$. Then $Y$ is stratified naturally by the strata of $X$ given by  the union of $L$  and those strata of $X$ that lie on the boundary of $L$.
The tangent microbundle for $Y$ can be constructed as above, replacing $X$ by $Y$. We denote the tangent microbundle $tY$ of $Y$ by $t(L; X)$ and call it the \emph{relative tangent microbundle (relative to $L$)}. Note that $t(L; X)$ is a microbundle over $Y$. If, moreover, $L$ is a dense stratum in $X$, then $t(L; X) = tX$.\\

\noindent {\bf Filtering the tangent microbundle:}
Observe that $tX$ admits a filtration by $t(L; X)$ for $L$ varying over strata of $X$.
Thus, $t(L; X)$ is a sub(micro)bundle of $tX$ restricted to any stratum $S < L$, as  
{$${t(L; X)\vert_{S} =TS \oplus N_{\bar{L}}(S)}$$ }
and  
$${tX|_S =TS \oplus  N_X(S)},$$ 
where $N_{\bar{L}}(S)$ and $N_X(S)$ denotes the normal cone microbundles of the stratum $S$ in $\bar{L}$ and $X$, respectively. So $t(L; X)\vert_{S} \subset t(X)\vert_{S}$. For any particular stratum $S$ of $X$, the collection $\{t(L; X) : S < L\}$ induces a filtration of the normal bundle $N_X(S)$ by the subbundles $\{N_{\bar{L}}(S) : S < L\}$. These in turn induce the filtration by stratum-closures $\{cA^L_S : S < L\}$ on any particular conical fiber $cA_S$. Here $A^L_S$ denotes the link of $S$ in $\bar{L}$.

\subsection{{Gromov diagonal normal construction for manifolds with jets}}\label{sec-formalfnmflds} We detail some of the points made  in Remark \ref{rmk-formalasgerms} and
 briefly recall Gromov's diagonal normal sheaf construction in the manifold context before generalizing to stratified spaces. Recall (Definition \ref{def-formalfn}) that if $\FF$ is a {continuous sheaf} over a manifold $M$, then the sheaf $\PP$ over $M \times M$ is given by  $\PP(U \times V) = \maps(U, \FF(V))$. Further, the Gromov diagonal normal sheaf $\FF^\bullet$ is obtained by restricting $\PP$ to the diagonal $\diag(M) \subset M\times M$.

Let $p:E \to M$ be a smooth fiber bundle over $M$ and $\FF$ be the sheaf of sections of $E$, i.e.\ $\FF(U) = \Gamma(U; E)$, where the space of sections $\Gamma(U; E)$ is equipped with {the quasitopology inherited from $\mathrm{Maps}(U, E)$}. Then 
$$\PP(U \times V) = \maps(U, \FF(V)) = \maps(U, \Gamma(V; E)).$$
{Therefore,} these consist of those maps $U \times  V \to E$ for which the restriction to {$\{u\} \times V$} gives a section of $E$ over $V$, {for any $u \in U$}.

Recall that a collection of elements $\phi_i \in \PP(U_i \times U_i), i=1, \cdots, m$ is consistent if for all $i \neq j$, $\phi_i=\phi_j$ on ${(U_i \cap U_j) \times (U_i \cap U_j)}$.
Then $\FF^\bullet$  can be described as follows. For any open set $W \subset \diag(M)$, an element of $\FF^\bullet(W)$ consists of a collection of consistent elements ${\phi_i \in \PP(U_i \times U_i)}$ where $\{U_i \times U_i\}$ is a basic open cover of $W$ in $M \times M$. Consistency of ${\phi_i}$ ensures that they define a well-defined germ of a map $(TW, W_0) \to {E}$ at the zero-section $W_0$ of the tangent bundle $TW \subset TM$. 

Let $M_0$ denote the zero-section of $TM$. Then, (global) sections of $\FF^\bullet$ may be viewed as certain germs of maps ${\psi} : (TM, M_0) \to E$. The fact that the maps $U_i \times U_i \to E$ are sections when restricted to the second factor implies the following two facts about the germ $\psi: (TM, M_0) \to E$: 
\begin{enumerate}
	\item $\psi$ is a section $M \to E$ when restricted to the 0-section $M_0$ of $TM$.
	\item $\psi$ is a germ of a section $(T_p M, 0) \to E$ when restricted to the tangent space $T_p M \subset TM$ at $p \in M$.
\end{enumerate}

Replacing $M$ by $U$, $\FF^\bullet(U)$ consists of germs of mappings $\psi_U: (TU, U_0) \to E$ such that $\psi_U$ {is} a section of $E$ when restricted to the 0-section $U_0 \subset TU$ and is a germ of a section of $E$ {over a neighborhood of $p$} when restricted to any tangent space $T_p U$ for $p \in U$. Thus, any element of $\FF^\bullet(U)$ consists of a section $s$ of $E$ over $U$ {decorated with a collection of} germs ${g_p : (U_p, p) \to (E, s(p))}$ of sections of $E$. Here, 
\begin{enumerate}
\item $g_p$ is defined on some neighborhood $U_p$ of $p \in U$, and sends $p$ to $s(p)$.
\item $p$ ranges over all $U$.
\end{enumerate}
 {Therefore,  we shall henceforth denote an element of $\FF^\bullet(U)$ as a tuple $(s, \{g_p : p \in U\})$. It is convenient to imagine this data as a \emph{base} section $s : U \to E$ of $E$ with the image $s(U) \subset E$ decorated by a \emph{field} of section-germs $\{g_p : p \in U\}$}.

There exists a natural {morphism of sheaves} $\Psi_r: \FF^\bullet\to \JJ^r_E$ from $\FF^\bullet$ to the {sheaf} of $r-$jets of sections of $E$ over $M$, {essentially given by setting $\Psi_r(s, \{g_p\})$ equal to the family of $r$-order Taylor polynomials of $g_p$ at $p$, as $p$ varies over $U$. We state a precise definition below:}

\begin{defn}\label{def-psi}
	{For any $(s, \{g_p : p \in U\}) \in \FF^\bullet(U)$,  define $\Psi_r(s, \{g_p\}) \in \JJ^r_E(U)$ to be the section of the $r$-jet bundle of $E$ such that at the point $p \in U$, the section takes the value $J^r_p g_p$. This defines a morphism of sheaves $\Psi_r : \FF^\bullet \to \JJ^r_E$.}
\end{defn}
%Then $\Psi_r$  is a homotopy equivalence. We summarize the discussion above in the following:

\begin{prop}\label{prop-formalfnmflds} For any $r \in \natls$,
	$\Psi_r: \FF^\bullet \to \JJ^r_E$ is a {weak} homotopy equivalence of sheaves. Equivalently, for any $r>0$, the Gromov diagonal normal sheaf $\FF^\bullet$ is naturally homotopy equivalent to the sheaf $\JJ^r_E$  of $r-$jets of sections of $E$.
\end{prop}

\begin{proof}
	Consider the {space $C^\infty_0(\R^n, \R^m)$ of germs of smooth maps from $\mathbb R^n$ to $\mathbb R^m$ at the origin}. Also, let ${P^r_0(\R^n, \R^m) \subset C^\infty_0(\R^n, \R^m)}$ denote the {subspace of} germs of polynomials (in $n$ variable) of degree at most $r$ at $0$. Let $T_r(f)$ denote the Taylor expansion of $f \in C^\infty_0(\R^n, \R^m)$ {at 0}, truncated at degree $r$. Then 
	$$f_t := T_r(f) +t (f-T_r(f)), t \in [0,1]$$ 
	furnishes a {deformation retraction of $C^\infty_0(\mathbb R^n, \mathbb R^m)$ onto $P^r_0(\mathbb R^n, \mathbb R^m)$. As locally $\FF^\bullet$ and $\JJ^r_E$ are isomorphic, respectively, to the sheaves $\mathrm{Maps}(-, C^\infty_0(\R^n, \R^m))$ and $\mathrm{Maps}(-, P^r_0(\R^n, \R^m))$, we conclude $\Psi_r$ is a weak homotopy equivalence.}
\end{proof}

\begin{rmk}\label{rmk-formalfn}
	In fact, $\Psi_\infty: \FF^\bullet \to \JJ^\infty_E$, sending ${(s, \{g_p : p \in U\}) \in \FF^\bullet(U)}$ to its infinite jets is also surjective. If we restrict only to analytic sections, then $\Psi_\infty$ is {moreover} injective. 
	
	$\Psi_1: \FF^\bullet \to \JJ^1_E$ given by $${\Psi_1(s, \{g_p\})= (s, \{dg_p\})}$$ is of particular significance. It replaces the {germ-field $\{g_p : p \in U\}$ decorating the base section $s$ by the tangent plane fields $\{dg_p : p \in U\}$}.
\end{rmk}

\subsection{Gromov diagonal normal construction for cones}\label{sec-formalfncone}
We would like to extend the linearized notion of formal $r-$jets of sections ensured by Definition \ref{def-psi} and Proposition \ref{prop-formalfnmflds} from the manifold context  to the context of  a stratified bundle  $P:E \to X$ over a stratified space. However, a full linearization is not possible 
and we shall provide a hybrid construction, interbreeding
\begin{enumerate}
	\item the linear structure within manifold strata provided by Definition \ref{def-psi} and Proposition \ref{prop-formalfnmflds},
	\item the germ construction in Remark \ref{rmk-formalasgerms} for cones on links using the local structure given by Corollary \ref{cor-trivialzn}.
\end{enumerate} 
In this subsection, we shall focus on the second ingredient, and in the next subsection indicate how to assemble these two together. Let  $P:E \to X$ be a stratified bundle.
For the purposes of this subsection, $E=cB, X=cA$, where $E=B\times [0,1)/B \times \{0\}$, and
$X=A\times [0,1)/A \times \{0\}$ and $A, B$ are compact {abstractly} stratified spaces.
Let $c_A$ (resp.\ $c_B)$ denote the cone-point of $cA$ (resp.\ $cB$). Let $cB^0$ (resp. $cA^0$) denoted the deleted cone ${cB \setminus \{c_B\}}$  (resp. ${cA \setminus \{c_A\}}$).

By Corollary \ref{cor-trivialzn}, there exists a stratified bundle $p: B \to A$ such that
(after reparametrization if necessary), $P(b, t) = (p(b), t)$. Hence, there exists a natural stratified bundle  $P^0: cB^0\to cA^0$  induced by $P$.

Let $\FF$ (resp. $\FF_w$) denote the sheaf of  controlled (resp.\ weakly controlled) sections of $P:E \to X$ in this case. 
Note that any section of $P:cA(=E) \to cB(=X)$ necessarily sends $c_A$ to $c_B$. Let $\FF^0$ (resp. $\FF_w^0$) denote the induced  sheaf of  controlled (resp.\ weakly controlled) sections of $P^0: cB^0\to cA^0$.

We shall first inductively describe the Gromov diagonal normal sheaf for $\FF^0$ (resp. $\FF_w^0$) 
using the sheaf $\LL$ (resp. $\LL_w$) of controlled {(resp.\ weakly controlled)} sections of $p:B \to A$ and
Lemma \ref{lem-mapsI2F}. Let $\LL^\bullet$  (resp. $\LL_w^\bullet$) denote the Gromov diagonal normal sheaves of  $\LL$ (resp. $\LL_w$).
In particular, when $A, B$ are manifolds, then (see Section \ref{sec-formalfnmflds}):
\begin{enumerate}
	\item $p:B \to A$ is a smooth bundle map,
	\item $\LL=\LL_w$ is the sheaf of smooth sections given by $\LL(U) =\Gamma (U, B) $
	\item $\LL^\bullet$ is homotopy equivalent to the sheaf $\JJ^r_B$ of  $r-$jets (Proposition \ref{prop-formalfnmflds}).
\end{enumerate}

Then  controlled sections of $P^0: cB^0\to cA^0$  are given by maps of the form $\sigma
:  cA^0 \to cB^0$ of the form $\sigma (a,t)
=(s_t(a),t)$, where $s_t: A \to B$ is a   controlled section of $p: B\to A$, i.e.\ $\sigma $ is a continuous $(0,1)-$parametrized family of sections
from $A$ to $B$. The same holds for controlled sections of $P^0: (P^0)^{-1}(U)\to U$ for all open $U$ in $cA^0$. {Therefore, we have an isomorphism of sheaves $\mathcal{F}^0 \cong \maps^p((0, 1), \mathcal{L})$, where $\maps^p$ denotes the parametric sheaf as defined in Definition \ref{defn-sheafmaps2ff2}.}
The following is now an immediate consequence of Lemma \ref{lemw2ff*p}:

\begin{lemma}\label{lem-deletedconesheaf} With $\FF^0, \LL$ as above, and open subsets $V \subset A$, and $W \subset (0,1)$, we have {a homotopy equivalence:
	$$(\FF^0)^\bullet(W \times V) \simeq \maps(W,\LL^\bullet(V)).$$}
\end{lemma}

Next, we describe the relationship between the sheaf of controlled sections and the sheaf of
weakly  controlled sections. Suppose $X=cA$ is equipped with a control structure (i.e.\ both a projection $\pi$ and a radial function $\rho$ in a neighborhood of $c_A$). Let $\rho_A$ denote the radial function on a small neighborhood of $c_A$ in $cA$. Without loss of generality, by shrinking 
$cA$ if necessary, we may assume that $\rho_A: cA^0 \to (0,1)$ is a fiber bundle with fiber $A$.
Pulling back $\rho_A$ under $P$ we obtain a radial function $\rho_B=P \circ \rho_A$ on $cB$. 

\begin{assume}\label{assume-control}
	Thus, without loss of generality, we assume  that $\rho_B=P \circ \rho_A$, i.e.\ $P$ is a
	controlled stratified bundle map from $cB$ to $cA$ in a neighborhood of $c$. Further, $P^0: cB^0 \to cA^0$ is a bundle map such that 
	$$\rho_B(t,b)= t = \rho_A (P^0(t,b)).$$
\end{assume}

We shall say that a section $s : cA^0 \to cB^0$ is \emph{levelwise weakly controlled} if $s(\{t\} \times A^0) \subseteq \{t\} \times B^0$ for all $t \in (0, 1)$ and the restriction  $s|(\{t\} \times A^0)$ is a weakly controlled map to $\{t\} \times B^0$.
With the control structure {on the domain and target of} $P : cB\to cA$ in place, the space of weakly controlled sections $\Gamma_w (cA^0, cB^0)$ fibers over the space of levelwise weakly controlled sections $\Gamma_\ell(cA^0, cB^0)$.  {The fibers of this fibration} are given by reparametrizing the $(0,1)-$direction in $cA^0$. More precisely, there exists a 
surjection $\Theta: \Gamma_w(cA^0, cB^0) \to \Gamma_\ell(cA^0, cB^0)$, such that for any $\sigma \in \Gamma (cA^0, cB^0)$, $\Theta^{-1} (\sigma) \cong \maps (A, \diff^+((0,1)))$, where $\diff^+((0,1))$ denotes the orientation-preserving diffeomorphisms of $(0,1)$. Thus, there is a natural product fibration:
\begin{equation*}\label{eqn-strong2weakcontrol}
	\Gamma_w (cA^0, cB^0) = \Gamma_\ell (cA^0, cB^0) \times \maps (A, \diff^+((0,1)))
\end{equation*}
This is because $	\Gamma_w (cA^0, cB^0)$ is a principal $\maps (A, \diff^+((0,1)))-$bundle
over $\Gamma_l (cA^0, cB^0) $  equipped with a natural section given by the inclusion $\Gamma_\ell(cA^0, cB^0) \hookrightarrow \Gamma_w(cA^0, cB^0)$ of levelwise weakly controlled sections into the weakly controlled sections.
Since $\diff^+((0,1))$ is contractible, we {obtain a homotopy equivalence between $\Gamma_w(cA^0, cB^0)$ and $\Gamma_\ell(cA^0, cB^0)$. Consequently, we obtain} 

\begin{cor}\label{cor-deletedconesheaf}
	With $\FF^0_w, \LL_w$ as above, and open subsets $V \subset A$, and $W \subset (0, 1)$, we have {a homotopy equivalence:
	$$(\FF^0_w)^\bullet(W \times V) \simeq \maps(W, (\LL_w)^\bullet(V)).$$}
\end{cor}

\begin{defn}\label{def-conegerms}
	The space of \emph{germs} 
	of  controlled  (resp.\ weakly controlled)  sections of $P:cB(=E) \to cA(=X)$ will be denoted by
	$\Gamma_c(cA,cB)$ (resp.\ $\Gamma_{c,w}(cA,cB)$).
\end{defn}

We are now in a position to note the following Proposition which allows us to assemble the descriptions in Lemma \ref{lem-deletedconesheaf} and Definition \ref{def-conegerms}.
This is useful in providing an inductive description of the Gromov diagonal normal sheaf of sections of a stratified bundle.

\begin{prop} \label{prop-decompcones}
	Any element of the sheaf $\FF^\bullet(U)$ (resp. $\FF_w^*(U)$) for an open $U \subset cA$ determines and is determined by the following:
	\begin{enumerate}
		\item a controlled (resp. weakly controlled) section $s$ over $U$. In particular, if $U=cA$, $s: cA \to cB$ is a global controlled (resp. weakly controlled) section,
		\item a germ at $c_A$ {given} by an element of $\Gamma_c(cA,cB)$ (resp.\ $\Gamma_{c,w}(cA,cB)$) if 
		$c_A \in U$,
		\item an element {$g = \{g_w : w \in U \setminus \{c_A\}\}$ of $(\FF^0)^\bullet(U \setminus \{c_A\})$ (resp. $(\FF^0_w)^\bullet(U \setminus \{c_A\})$) such that 
			\begin{enumerate}
				\item[(i)] The first coordinate of $g$ (as in Remark \ref{rmk-formalasgerms}) coincides with the restriction ${s\vert_{U\setminus \{c_A\}}}$, and 
				\item[(ii)] As $w \to c_A$, $g_w$ ``converges" to $s$, in the sense that the germs $\{g_w : w \in U \setminus \{c_A\}\}$ and the germ of $s$ at $c_A$ jointly define a continuous map from $\mathrm{Op}(\mathrm{diag}(U)) \subset U \times U$ to $cB$.
			\end{enumerate}
}
	\end{enumerate}
\end{prop}

\begin{proof} For concreteness, we work with $\FF$ and $\FF^\bullet$. The same argument works for  $\FF_w$ and $\FF_w^\bullet$. Further,
	Equation \ref{eqn-strong2weakcontrol} really ensures that up to homotopy, these sheaves are the same, as we can apply induction (on depth) to weakly controlled
	sections from $A$ to $B$ in 	Equation \ref{eqn-strong2weakcontrol}.
	
	We use the description of the Gromov diagonal normal sheaf from Remark \ref{rmk-formalasgerms}: any element of $\FF^\bullet(U)$ consists of a controlled section 
	$s$ over $U$ decorated with germs of sections ${\{g_w: w\in U\}}$. The controlled section 
	$s$ over $U$ contributes item 1 in the statement.
	Next,
	\begin{enumerate}
		\item For $w=c_A$, ${g_w}$ is given by an element as in item 2 in the statement.
		\item For $U'=U \setminus \{c_A\}$, ${\{g_w: w \in U'\}}$ {constitutes an} element as in item 3 in the statement.
	\end{enumerate}
	Finally, we observe that the choices in items 2, 3 are {essentially independent, barring the only restriction that the germs $\{g_w : \mathrm{Op}(w) \to cB : w \in U'\}$ and the germ of $s$ at $c_A$ given by $s : \mathrm{Op}(c_A) \to cB$ jointly define a continuous map from $\mathrm{Op}(\mathrm{diag}(U)) \subset U \times U \subset cA \times cA$ to $cB$}. Hence any choice as in items 2, 3 subject to the choice of a section $s$ as in item 1 furnishes an element of  $\FF^\bullet$.
\end{proof}

{Proposition \ref{prop-decompcones} allows us to decompose elements of $\FF^\bullet$ into two independent components.
Item 2 provides a purely topological component of $\FF^\bullet$, consisting of germs of sections $\mathrm{Op}_{cA}(c_A) \to E$. Item 3 on the other hand is expressible in terms of the sheaf $\LL^\bullet$ defined on the link $A$. This enables a recursive description as follows:
restricting $\LL^\bullet$ to any normal cone appearing in the link $A$,  elements in Item 3 of Proposition \ref{prop-decompcones} can be decomposed again as a hybrid of objects as in Item 2 and Item 3, but defined on a stratified space of lesser height.}

{ The case where $A, B$ are manifolds is the lowest height case. In this case, Proposition \ref{prop-formalfnmflds}
provides a completely linear description of $\FF^\bullet$ as a sheaf of jets (linearized germs). We note, however, that this last linear description is true only up to weak homotopy equivalence.}

We illustrate this recursive description by a simple example. 

\begin{eg}\label{eg-recstratjet}
\rm{Let $X = [0, \infty) \times [0, \infty)$ be stratified as a manifold with corners. Let $p : X \times \mathbb R \to X$ denote the trivial line bundle on $X$. Let $E= X \times \mathbb R$. Then, a formal section of $p$ over $X$ consists of a collection of germs $\{g_{(x, y)} : (x, y) \in X\}$. We observe,
\begin{enumerate}
\item The germ of a section $g_{(0, 0)}$ is a purely topological section of $E$ over $ \op (\{0,0\})$ as in Item $2$.  This germinal section has \emph{no linear structure} associated with it.
\item The germs $g_{(0, y)}$ and $g_{(x, 0)}$ are hybrid objects as in Item $3$. We decompose them into two components: 
	\begin{enumerate}
	\item[(i)] Restricting to the edge strata $(\{0\} \times (0, \infty))$ and $((0, \infty) \times \{0\})$, one obtains a formal section of $p$ over these edge strata. These may then be linearized to $r$--jets. These are the linear components
	of the germs $g_{(0, y)}$ and $g_{(x, 0)}$. 
	\item[(ii)] Restricting to the normal cones of an edge stratum, we obtain a $(0, \infty)$--parametrized family of formal sections over the normal cones. These are again, purely topological sections of $E$ over the normal cones.
	Note that the normal cones to the edge strata are half-closed intervals. The height of these  normal cones equals one, while the height of $X$ is two. These are the non-linear or purely topological  components of the germs $g_{(0, y)}$ and $g_{(x, 0)}$.
	\end{enumerate}
\item The germs $g_{(x, y)}$ for $(x, y) \in (0, \infty) \times (0, \infty)$ are germs of sections defined on a manifold stratum of $X$. Thus, these may be completely linearized to $r$--jets.
\end{enumerate}
}
\end{eg}

Note also that for manifolds, elements in $\FF^\bullet(U)$ may be identified with $U$-parametrized sections from the tangent space $T_pU \to E$ provided there is a way (e.g.\ a connection) of identifying $T_pU$ and $T_qU$ for all $p, q \in U$. Regarding the tangent bundle $TU$ as 
the germ of a neighborhood of the diagonal $\diag U \subset U \times U$, elements in $\FF^\bullet(U)$ are thus equivalent to $U-$parametrized sections of $E$ over the normal space $N_{(p,p)} (\diag U)$ to the diagonal at some point $(p,p) \in \diag U$.

Suppose $\FF$ is a sheaf of topological spaces over a manifold $M$ such that the inclusion $\FF \subset \FF^\bullet$ has an inverse given by a retraction of sheaves $r: \FF^\bullet \to \FF$ so that $r$ is a fibration. Let $\PP$ denote the sheaf over $M \times M$ given in Section \ref{sec-formalfnmflds}. Define a stratification of $X=M \times M$ with strata
$S= \diag M$ and $L= (M \times M)\setminus S$. Then we define a stratified sheaf $\RR$ over 
$X$ so that 
\begin{enumerate}
	\item {$\RR|{L} = \PP$}
	\item $\RR|S = \FF$
	\item the restriction map from $\RR|\bbar{L} $ to  $\RR|S$ is given by first restricting 
	$\PP$ to $S$, to obtain $\FF^\bullet$, and then composing with the fibration $r$.
\end{enumerate}
Then (Definition \ref{def-infstratflex}) $\RR$ is infinitesimally flexible across strata.

\subsection{Gromov diagonal normal construction: general case}\label{sec-gdn}
 For the purposes of this subsection,
$(E, \Sigma_E, \mathcal{N}_E)$ and $(X, \Sigma_X, \mathcal{N}_X)$ are {abstractly} stratified spaces (see Definition \ref{def-abtractstratsp} for notation) and $P : E \to X$ is a stratified fiber bundle. Then, Lemma \ref{lem-strbdl-trivialization} and Corollary \ref{cor-trivialzn} give us the following commutative diagram.

\begin{center}
	$
	\begin{CD}
		cB@>>>\til{N}@>\til{\pi}>>\til{S} \\
		@VPVV @VPVV @VPVV\\
		c{A}@>>>{N}@>{\pi}>>{S} \\
	\end{CD}
	$,
\end{center}
where the horizontal rows are fiber bundles.

\begin{comment}
	\widetilde{S} \arrow[d, "P"'] & \widetilde{T} \arrow[d, "P"'] \arrow["\widetilde{\pi}"', l] & c\widetilde{A} \arrow[d, "P"'] \arrow[l] \\
	S                            & T \arrow["\pi"', l]                             & cA \arrow[l]                            
\end{comment}

Recall (Remark \ref{rmk-formalasgerms}) that a formal section of $P: E \to X$ on an open subset $U \subset X$ is a germ of a continuous map $s^* : \op_{U \times U}(\diag(U)) \to E$ from the germ of an open neighborhood $\op_{U \times U}(\diag(U))$ of the diagonal $\diag(U) \subset U \times U$  such that 
\begin{enumerate}
	\item $s : U \to E$ defined by $s(u) := s^*(u, u)$ is a section of $E$ over $U$.
	\item For every $u \in U$, $s_u : \op_U(u) \to E$ defined by 
	{${g_u}(v) = s^*(u, v)$} is a germ of a section of $E$ in  $\op_U(u) \in U$.
	\item For every stratum $S \in \Sigma_X$ of $X$ intersecting $U$, $s$ is smooth on $S' = S \cap U$, i.e.\ $s^*|\op_{S' \times S'}(\diag({S'}))$ is a smooth germ of a map to the unique manifold stratum $\widetilde{S} \subset E$ containing $s(S)$.
\end{enumerate}

Henceforth,	in this subsection,
we shall refer to $s$ as the \emph{base} of the formal section. 

\begin{defn}\label{def-holosxn} A formal section 
	$s^*$ is called a \emph{holonomic section} of $P: E \to X$  over $U$ if  {$$s|\op (u) = {g_u}$$ }  for all $u \in U$.
\end{defn}

Let ${s^*}$ be a formal section of $P: E \to X$ over $X$. For every stratum $S$, we can restrict ${s^*}$ to $\op_{S \times X}(\diag (S))$. Note that $\op_{S \times X}(\diag (S))$ is isomorphic as a microbundle to $$(S \times S, p_1, \diag (S)) \oplus (N_S, \pi_S, 0_S),$$ where $N_S$ is the normal neighborhood of $S$ (cf.\  Definition \ref{def-abtractstratsp}) and $0_S$ is naturally identified with $S\subset N_S$. Here, $N_S$ is thought of as the (micro)normal bundle to $S$ in $X$ with fiber $cA$, where $A$ is the link of $S$ (see, for instance, the commutative diagram above).

\begin{defn}\label{def-tangnormformal} Using the microbundle-isomorphism
	$$\op_{S \times X}(\diag (S))\cong(S \times S, p_1, \diag (S)) \oplus (N_S, \pi_S, 0_S),$$
	\begin{enumerate}
		\item restricting $s^*$ to the first component, we obtain the \emph{tangential formal section}\\ ${\ss_{S, t}^*} : \op_{S \times S}(\diag (S)) \to E$, with base section ${\ss_{S, t}}$;
		\item restricting $s^*$ to the second component, we obtain the \emph{normal formal section}\\ ${\ss_{S, n}^*} : \op_{N_S}(0_S) \to E$, with base section $\ss_{S, n}$.
			\item $\ss_{S, n}^*$ and $\ss_{S, t}^*$ have the same base section $\ss_{S, n}^*|0_S = \ss_{S, t}^*|\diag(S) = \ss_S$.
	\end{enumerate}
	
\end{defn}

	\begin{rmk}[Normal formal is holonomic]\label{rmk-nfisholo} 
		\rm{Restriction of the normal formal section $\ss_{S, n}^* : \op_{N_S}(0_S) \to E$ to the fiber $c_xA=cA(x) \subset N_S$ of the normal bundle over a point $x \in S$ is a germ of a controlled section $\ss_{S, n}^*|\op_{cA(x)}(x)$ of the conical component of the stratified bundle $(\mathbf{I}, P_2) : cB \to cA$ near the cone point {$\{c_x\} \subset c_x A$}. Thus, for an open chart $V \subset S$ around $x$, we obtain a map:
			$$\phi_V : V \to \Gamma_c(cA, cB),\, \phi(x) := \ss_{S, n}^*|\op_{cA(x)}(x)$$
			On the other hand, the restriction of the normal formal section $\ss_{S, n}^*$ to the zero section $0_S \subset N_S$ of the normal bundle returns the base ${\ss}_S$ of the formal section. Therefore, $\ss_{S, n}^*\vert_{V}$ is a germ of a \emph{holonomic section} around $V \times cA \cong \pi_S^{-1}(V) \subset N_S$,
{\begin{align*}
			({\ss_S}\vert_{V}, \phi_V) : \op_{V \times cA}(V \times 0) &\to E\\
			(x, (a, t)) &\mapsto (\ss_S(x), \phi_V(x)(a, t)).
\end{align*}
The intuition behind this observation is as follows. The normal formal section $\ss_{S, n}^*$ is a formal section of $E$ over the germ of a tubular neighborhood $N_S$ of the stratum $S$. Loosely speaking, we can think of $\ss_{S, n}^*|_{S}$ as an $S$--parametrized family of formal sections of $E$ over germs of conical fibers $c_x A$,  $x \in S$. This is exactly the purpose $\phi_V$ serves, at least locally on open charts $V \subset S$. Given a formal section $\{g_p : p \in \op_{cA}(c_A)\}$ of $E$ defined over $\op_{cA}(c_A)$, there is a canonical way to extract a holonomic section of $E$ over $\op_{cA}(c_A)$. Namely, consider the holonomic germ $g_{c_A}$ at $c_A$.
The map $({\ss_S}\vert_{V}, \phi_V)$ now assembles the conical germs $\{g_{c_x}: x \in V\}$ together using the section $\ss_S : S \to E$.}

			We pause to emphasize that while this says that $\ss^*_{S, n}$ is globally a (germ of a) holonomic section of $E$ over $N_S$, there is not a meaningful way to write $\ss^*_{S, n}$ as a section of $E$ over $S$ (namely, ${\ss_S}$), together with a $S-$parametrized family of sections in $\Gamma_c(cA, cB)$. For one, observe that the bundle homomorphism $P : P^{-1}(N_S) \to N_S$ does not induce a unique map $(\mathbf{I}, P_2) : cB \to cA$ between the normal conical fibers, but rather a unique equivalence class of such maps under $\mathrm{Homeo}_c(cA)-$ and $\mathrm{Homeo}_c(cB)-$valued cocycles acting on the domain and range, respectively.
	}\end{rmk}

\begin{comment}
	content...

Note that the normal formal section ${\ss_{S, n}^*}$ may be regarded as a germ of a section of $P:E \to X$ over the normal bundle $N_S$ near $0_S$ (naturally identified with $S$), and $${\ss}_{S, n} = {s^*}_{S, n}|0_S= {s^*}_{S, n}|S = {\ss}_{S, t}= {\ss}_{S},$$
i.e.\ both the tangential and normal formal sections have the same base section
${\ss}_{S}$. 
\end{comment}

In general it is not possible to recover the germ ${s^*}|\op_{S \times X}(\diag (S))$ from the tangential and normal formal sections. However, for any $\varepsilon > 0$, ${s^*}|\op_{S \times X}(\diag (S))$ is $\varepsilon$-close to ${s^*}_{S, t} \oplus {s^*}_{S, n}$ by continuity in the $C^0$-norm.
Here, $${s^*}_{S,t} \oplus {s^*}_{S, n}(x, y, z) = ({s^*}_{S, t}(x, y), {s^*}_{S, n}(x, z)) \in \widetilde{N}_S \subset E$$ for $x \in S$, $y \in \op_S(x)$, $z \in \op_{cA(x)}(x)$, where $cA(x) =\pi^{-1}(x)$ (cf.\ commutative diagram above). For Definition \ref{def-crformal} below, 
we assume that $X, E$ are equipped with a metric as at the end of Section \ref{sec-stratfdjet}.
Further, when we say that two formal sections are $\varepsilon$-close, it is in the sense of closeness with respect to such a metric.

\begin{defn}\label{def-crformal} Let $\ss^* : \op_{X \times X}(\diag (X)) \to E$ be a formal section of $P: E \to X$ over $X$. Let $S < L$ be a pair of strata in $X$.
	The $\delta-$neighborhood of $S$ in $ L$ will be denoted as $N_\delta({S, L})$.
	We shall say that $\ss^*$ is of \emph{regularity $C^r$} if for all pairs $S < L$ and $\varepsilon > 0$, there exists $\delta > 0$ such that 
	\begin{enumerate}
		\item  $\ss^*_{S, n}|N_\delta({S, L}) : N_\delta({S, L}) \to E$ is smooth on the open stratum $L$,
		\item $\ss^*_{S, t} \oplus (\ss^*_{S, n}|N_\delta(S, L))$ is $\varepsilon$-close to $\ss^*_{L, t}$ in the $C^r$ norm. We shall summarize this condition by saying that $\ss^*_{L, t}$ is \emph{$C^r-$asymptotic} to $\ss^*_{S, t} \oplus ( \ss^*_{S, n}|N_\delta({S, L}) )$.
	\end{enumerate}
\end{defn}

{We briefly discuss the rationale behind this definition. We discussed above that the germ $s^*|{\mathrm{Op}_{S \times X}(\mathrm{diag}(S))}$ cannot be recovered from the tangential and normal formal sections, but is only $\varepsilon$--close to their direct sum $s^*_{S, t} \oplus s^*_{S, n}$, in the $C^0$--norm. As an explicit example, consider the case of $X = cA$, $S = \{c_A\}$. In this case, $s^*$ would be given by a collection of germs $$\{g_{(a, t)} : \mathrm{Op}_{cA}(\{(a, t)\}) \to E : (a, t) \in cA \setminus \{c_A\}\}$$ along with  a germ $$s : \mathrm{Op}_{cA}(c_A) \to E$$ at the cone point. The normal formal section in this case is $s$. Clearly $s^*$ carries more information than $s$ as $g_{(a, t)}$ are not simply restrictions of $s$ to $\mathrm{Op}(\{(a, t)\})$. However, as $t \to 0$, $g_{(a, t)}$ converges to $s$ in the sense described in Proposition \ref{prop-decompcones}, Condition (iii).}

{Below, we shall define a hybrid notion of \emph{stratified $r$--jets}.  We shall combine the following:
	\begin{enumerate}
	\item  smooth $r$--jets of sections restricted to each manifold stratum,
	\item  topological germs of sections restricted to the normal cones at the boundary of each manifold stratum.
	\end{enumerate}  
  In Definition \ref{def-psi} we had constructed a sheaf morphism that ``linearizes" any formal section to an $r$--jet, and in Proposition \ref{prop-formalfnmflds} we had proved that this is a weak homotopy equivalence. We wish to prove an analogous statement in the case of stratified $r$--jets. However, to associate a stratified $r$--jet to a formal section one requires some form of $C^r$-regularity for the collection of germs given by the formal section. This notion of regularity is what Definition \ref{def-crformal} formulates.}

{Let us illustrate this by specializing to  the case  $X = cA$, $S = \{c_A\}$ where $A$ is a manifold. Here, $s^*$ is $C^r$-regular if for all $\varepsilon > 0$, there exists $\delta > 0$ such that for all $a \in A$ and $t < \delta$, the germ $g_{(a, t)}$ is $\varepsilon$-close to $s$ on their common domain of definition, \emph{in the $C^r$--norm}. As their common domain of definition is a subset of the manifold stratum $\{(a, t) \in cA : 0 < t < \delta\} \subset cA \setminus \{c_A\}$, it is meaningful to compare them in the $C^r$ norm.}

The sheaf of $C^r-$regular holonomic (resp.\ formal) sections over $X$ will be denoted as $\FF_r$ (resp.\ $\FF_r^\bullet$). 
%We now refine Definition \ref{def-sj} and provide a description `orthogonal' to Proposition \ref{prop-formalfnnbhdstrat}. 
Let $W \subset X$ be open equipped with the inherited stratification.

	\begin{defn}\label{def-sjr}
		For every stratum $S \subset W$ of $W$, and a section $s : W \to E$,
		\begin{enumerate}
			\item Let $A_S$ be the link of $S$ in $X$,
			\item Let $\widetilde{S}$ be the unique stratum of $E$ containing $s(S)$,
			\item Let $B_S$ be the link of $\widetilde{S}$ in $E$,
			\item Let $p = (\mathbf{I}, P_2) : cB_S \to cA_S$ be the restriction of $P : E \to X$.
		\end{enumerate}
		Let $r \geq 1$. An element of $ \sjr^r(W)$ consists of a section $s : W \to E$ decorated by the following data corresponding to every stratum $S \subset W$:
		\begin{enumerate}
			\item A normal formal section $s^*_{S, n} : \op_{N_S}(0_S) \to E$ with base $s$,
			\item A formal $r-$jet $\sigma_S \in \mathcal{J}^r_E(S)$ of the fiber bundle $P : P^{-1}(S) \to S$.
		\end{enumerate}
		such that the following compatibility condition is satisfied. For every stratum $S \subset W$, consider $\sigma_S$ as an element  of the sheaf of  formal  sections of $E$ over $S$. Then, for any pair of strata $S < L$ of $W$, 
		\begin{equation*} \text{$\sigma_L$ is $C^r-$asymptotic to $\sigma_S \oplus s^*_{S, n}$}\end{equation*}
		We summarize the condition by saying $\{\sigma_S\}$ is \emph{normally $C^r-$compatible}.
	\end{defn}

\begin{prop}\label{prop-formalfnnbhdstrat2}
	For any $r\geq 1$,	the sheaf $\FF_r^\bullet$ is homotopy equivalent to  $\sjr^r$.
\end{prop}

\begin{proof}
		Consider the homomorphism of sheaves $\Phi : \FF_r^\bullet \to \sjr^r$ given on an open set $W \subset X$ by $\Phi(W) : \FF_r^\bullet(W) \to \sjr^r(W)$, where
		$$\Phi(W)(s^*) = (s, \{s^*_{S, n}\}, \{J^r(s^*_{S, t})\})$$
		Here, for every stratum $S \subset W$, $s^*_{S, n}$ and $s^*_{S, t}$ denote respectively the normal formal and tangential formal components of $s^*$ along $S$. Also, $J^r(s^*_{S, t})$ denotes the smooth $r-$jet $(s_{S, t}, \{J^r g_p : p \in S\})$ where we use the description $s^*_{S, t} = (s_S, \{g_p : p \in S\})$ of the tangential formal section as a base section on $S$ decorated by a germ-field of sections, as in Section \ref{sec-formalfnmflds}. This is a well-defined map, by the condition of normal $C^r-$compatibility (Definition \ref{def-sjr}).

		The candidate for a homotopy inverse is given by the inclusion $\iota : \sjr^r \hookrightarrow \FF_r^\bullet$  as a subsheaf, by considering a formal $r-$jet as a formal section. Observe that $\Phi \circ \iota = \mathrm{Id}$. To demonstrate that $\iota \circ \Phi(W)$ is homotopic to the identity map, we follow the proof of Proposition \ref{prop-formalfnmflds}. For every stratum $S$ of $W$, consider the straight-line homotopy
		$$F_t(s^*) = t s^* + (1- t) (J^r(s^*_{S, t}) \oplus s^*_{S, n}), \ t \in [0, 1]$$
		This establishes the desired deformation retract on every stratum.
\end{proof}

\begin{rmk}\label{rmk-sullivan}
\rm{
	Specializing the constructions of this entire section to the case of a manifold with corners, or even more specifically, to a simplex, the inductively defined structure given by Propositions \ref{prop-formalfnnbhdstrat2}, \ref{prop-decompcones} and \ref{prop-formalfnmflds} simplifies considerably, giving families of flags of tangent spaces. {We refer the reader to Example \ref{eg-recstratjet} where this description has been worked out explicitly in the case of the specific manifold with corners $X = [0, \infty) \times [0, \infty)$.
Thus, Example \ref{eg-recstratjet} gives a local description at open neighborhoods of vertices of a standard 2-simplex. }

{
In particular, this answers Sullivan's Question \ref{qn-sullivan} for the related notion of jets instead of forms, by giving an inductive description of stratified $r-$jets. Notice that, on a smooth manifold $M$, $1-$jets are affine $1-$forms. Specifically, there is an isomorphism $J^1(M) \cong \Bbb R \times \Omega^1(M)$, where $J^1(M)$ denotes the space of sections of the $1-$jet bundle of $M$. This relationship can be used to translate the inductive description of stratified jets into one for forms.}
}
\end{rmk}

\section{Holonomic approximation theorem and other consequences}\label{sec-hat}

\subsection{Flexibility of jet sheaves}\label{sec-jbdlflex}
Let $P: E \to X$ be a stratified bundle (Definition \ref{def-strbdl}). 
{For convenience of notation, let $\FF := \Gamma_E$ and $\FF_w := \Gamma_{E, w}$ denote the continuous sheaves of controlled and weakly controlled sections of $P$, respectively. Let
	$\HH^r_E, \HH^r_{E, w}, \JJ^r_E, \JJ^r_{E, w}$ denote the continuous sheaves of controlled holonomic, weakly controlled holonomic, controlled formal and weakly controlled formal $E$-valued $r$-jets of $P$ as in Section \ref{sec-stratfdjet}.
	Also, let $\sjr^r$ denote the sheaf given by Definition \ref{def-sjr}.}
	 If $P : E \to X$ is a stratified fiber bundle with \emph{manifold fibers}, the sheaf of all stratified $r-$jets will be denoted as  $\JJ^r_0$. (The sheaf $\JJ^r_0$ is of relevance in the example of a compact {connected} Lie group acting on a manifold $E$  with quotient a stratified space $X$. {Since we have already established in the discussion following Definition \ref{def-strbdl} that the \emph{connected components} of fibers of stratified bundles are manifolds, the sheaves $\JJ^r_0$ and $\sjr^r$ are closely related.)} 
	
{As in the Example following Definition \ref{def-sssheaf}, for any stratum $L\in \Sigma$ of $X$, let $(E_{\bbar{L}}, \bbar{L}, P_{\bbar{L}})$ denote the restricted stratified bundle $P_L:  P^{-1} (\bbar{L}) \to \bbar{L}$. We have the corresponding associated sheaves over $\bbar{L}$:
	$$\FF_{\bbar{L}}, \FF_{w,\bbar{L}}, \JJ^r_{E, \bbar{L}}, \JJ^r_{E,w, \bbar{L}}, \HH^r_{E, \bbar{L}}, \HH^r_{E, w, \bbar{L}}, \JJ^r_{0, \bbar{L}}, \sjr^r_{\bbar{L}}.$$
	We shall abuse notation slightly and refer to the stratified sheaf given by the collections of sheaves 
	$$\{\FF_{\bbar{L}}, \FF_{w,\bbar{L}}, \JJ^r_{E, \bbar{L}}, \JJ^r_{E,w, \bbar{L}}, \HH^r_{E, \bbar{L}}, \HH^r_{E, w, \bbar{L}}, \JJ^r_{0, \bbar{L}}, \sjr^r_{\bbar{L}} : L \in \Sigma\}$$
also by $\FF, \FF_w, \JJ^r_E, \JJ^r_{E,w}, \HH^r_E, \HH^r_{E, w}, \JJ^r_0, \sjr^r$. It will be clear from the context whether we are referring to the sheaf or the stratified sheaf over $X$.}
	We record the following  observation for concreteness:
	
	\begin{obs}\label{obs-holo=holojet}
	$\FF$ and $\FF_w$ are isomorphic to $\HH^r$ and $\HH^r_w$ respectively.
	\end{obs} 

The morphism from $\FF$ to $\HH^r$ is obtained by adjoining $r-$jets of holonomic sections, and {the inverse morphism} from $\HH^r$ to $\FF$ forgets the decoration.

\begin{defn}\label{def-sdr}  If X is a manifold, a \emph{differential relation $\OO$ (of order $r$)}  is a subsheaf of $\JJ^r_0$.
	
For a stratified space $X$, a \emph{stratified differential relation} $\{\OO_{\bbar{L}} : L < X\}$ (of order $r$)  is a stratified subsheaf of $\sjr^r$.
\end{defn}

We shall need an auxiliary combinatorial organizational tool.
	\begin{defn}\label{def-config}
	Let $(X, \Sigma)$, $(Y, \Sigma')$ be abstractly stratified spaces. A \emph{configuration of indexing sets}, or simply, a \emph{configuration}, is a set-map $\mathfrak{c} : \Sigma \to \Sigma'$
	between the indexing sets $\Sigma, \Sigma'$. We shall say that a stratum-preserving map $f : X \to Y$ is of configuration $\mathfrak{c}$ if $f(S) \subset \mathfrak{c}(S)$ for all $S \in \Sigma$.
\end{defn}

Let $(X, \Sigma)$, $(E, \Sigma')$ be the stratifications of $X, E$.
We assume, henceforth, in this subsection, that a configuration $\mathfrak{c} : \Sigma \to \Sigma'$ is fixed, and that the sheaves $\FF$, $\FF_w$ are \emph{implicitly decorated with an arbitrary, but fixed configuration $\mathfrak{c}$.} {That is, whenever $U \in \mathrm{Str}(X, \Sigma)$, the quasitopological spaces $\FF(U), \FF_w(U)$ will only consist of controlled (resp.\ weakly controlled) sections $s : U \to E$ of $P : E \to X$ such that $s$ is  of configuration $\mathfrak{c}$, with respect to the induced stratification of $U$.} For $E$ a manifold, $G$ a compact {connected} group, and $X=E/G$, the configuration $\mathfrak{c}$ (used to determine $\FF$ or $\FF_w$) is uniquely determined by $P$. This is because the fibers of $P:E \to X$ are {connected} manifolds, in particular fibers have a single stratum.
Hence, $\mathfrak{c} : \Sigma \to \Sigma'$ is automatically fixed.

\begin{theorem}\label{thm-diffrlnflex} {With the convention pertaining to the fixed configuration $\mathfrak{c}$ in the discussion above},
	the stratified sheaf $\FF$ is flexible. In particular, it satisfies the parametric $h-$principle, i.e.\ $\FF \hookrightarrow \FF^\bullet \simeq \sjr^r$ is a weak homotopy equivalence.   (Hence, by Observation \ref{obs-holo=holojet} $\HH^r$ is flexible).
\end{theorem}

The proof we give below also shows, mutatis mutandis, that $\FF_w$ is flexible.
We first prove $\FF$ is stratumwise flexible. We begin by proving two general lemmas pertaining to flexibility.

\begin{lemma}\label{lem-fixedF}Let $F$ be a fixed quasitopological space. Let $\maps(-,F)$ be the sheaf over a locally compact topological space $X$ given by 
	$$\maps(-,F)(U) :=\maps(U,F).$$
	Then, $\maps(-,F)$ is flexible.
\end{lemma}

\begin{proof}
		Let $K_1 \subset K_2$ be a pair of compact subsets of $X$, and $W$ be a topological space. Suppose $\phi : W \times I \to \maps(-, F)(K_1)$ is a homotopy with a given initial lift $\psi_0 : W \times \{0\} \to \maps(-, F)(K_2)$ of $\phi|W \times \{0\}$. By local compactness of $X$,
		\begin{align*}\maps(-, F)(K_i) &= \varinjlim_{U \supset K_i} \maps(U, F_i) \\
			&= \varinjlim_{U \supset K'} \varinjlim_{K' \supset K_i} \maps(U, F_i)\\
			&= \varinjlim_{K' \supset K_i} \varinjlim_{U \supset K'} \maps(U, F_i) \\
			&= \varinjlim_{K' \supset K_i} \maps(K', F_i)\end{align*}
		where $K'$ varies over compact neighborhoods of $K_i, i = 1, 2$. Therefore, we may find compact neighborhoods $K'_1 \supset K_1$, $K'_2 \supset K_2$ such that $\phi$ factors through $\phi' : W \times I \to \maps(K'_1, F)$ and $\psi_0$ factors through $\psi'_0 : W \times \{0\} \to \maps(K'_2, F)$. By further shrinking $K'_1, K'_2$ if necessary we may ensure $K'_1 \subset K'_2$ and $\psi'_0$ is a lift of $\phi'|W \times \{0\}$ to $K'_2$.

		Since $K_1' \subset K_2'$ is a compact inclusion and hence a cofibration, by \cite[pg. 50]{may}, $\maps(K_2', F) \to \maps(K_1', F)$ is a fibration. Therefore, $\phi'$ admits a lift $\psi' : W \times I \to \maps(K_2', F)$ such that $\psi'|W \times \{0\} = \psi'_0$. Let $\psi : W \times I \to \maps(-, F)(K_2)$ denote the germ of $\psi'$ around $K_2$. Then $\psi : W \times I \to \maps(-, F)(K_2)$ is a lift of $\phi$, with $\psi|W \times 0 = \psi_0$. This proves $\maps(-, F)(K_2) \to \maps(-, F)(K_1)$ is a fibration, establishing flexibility of $\maps(-, F)$. 
	\end{proof}

\begin{lemma}
		\label{lem-flexpdkt}Let {$\mathcal{G}_1, \mathcal{G}_2$} be flexible sheaves on a locally compact topological space $X$. Then {$\mathcal{G}_1 \times \mathcal{G}_2$} is flexible.
	\end{lemma}
	\begin{proof}
		Let $K_1 \subset K_2$ be a pair of compact subsets of $X$. By hypothesis, the restriction maps {$\mathcal{G}_1(K_2) \to \mathcal{G}_1(K_1)$} and {$\mathcal{G}_2(K_2) \to \mathcal{G}_2(K_1)$} are fibrations. Therefore, as a product of fibrations is a fibration, 
		{$$(\mathcal{G}_1 \times \mathcal{G}_2)(K_2) = \mathcal{G}_1(K_2) \times \mathcal{G}_2(K_2) \to \mathcal{G}_1(K_1) \times \mathcal{G}_2(K_1) = (\mathcal{G}_1 \times \mathcal{G}_2)(K_1)$$}
		is a fibration. This proves the lemma.
\end{proof}

\begin{cor}\label{cor-jetsflex}
	${\mathcal{F}}$ is stratumwise flexible.
\end{cor}
\begin{proof}
	{Since each stratum $S$ of $X$ is a manifold, the restriction of the stratified fiber bundle over $S$, $P_S : E_S := p^{-1}(S) \to S$, defines a genuine smooth bundle $(E_S, S, P_S)$. Let $F_S$ denote the fiber of $P_S : E_S \to S$.} Note that $\FF_S = i^*\FF_{\bbar{S}}$ is the sheaf of sections of {$P_S : E_S \to S$}. {Given any $x \in S$ we may choose an open neighborhood $U$ around $x$ over which $U$ trivializes. Therefore, $i^*_U\FF_S \cong i^*_U\maps (-,F_S)$}. Hence, $\FF_S$ is locally flexible by Lemma \ref{lem-fixedF}. Global flexibility now follows from Gromov's localization lemma \cite[p. 79]{Gromov_PDR}.
	
	Alternately, partition of unity directly allows us to glue families of sections over a family of sets to extend families of sections, and hence establish that $\FF_S$ is flexible.
\end{proof}

	\begin{proof}[Proof of Theorem \ref{thm-diffrlnflex}] Let $x \in X$ be a point, and $S$ be the unique stratum containing $x$. Let us choose a neighborhood $W \subset X$ of $x$ such that $W \cong V \times cA$, where $V = W \cap S$ and $cA$ is the normal cone of $S$ in $X$. Then $\mathcal{F}(W)$ is the quasitopological space of sections of $E$ over $W$. We obtain a map 
		$$\mathrm{res}^W_V : \mathcal{F}(W) \to \mathcal{F}_S(V)$$
		by restricting a section of $E$ over $W$ to a section of {$E_S$} over $V \subset S$. Explicitly, let $s : V \times cA \cong W \to E$ be a section in $\mathcal{F}(W)$. Let $\widetilde{S}$ be the unique stratum of $E$ containing $s(V)$, and $cB$ be the normal cone of $\widetilde{S}$ in $E$. Then by the local structure of stratified bundles, $s(v, a) = (t(v), f(v, a))$ for all $(v, a) \in V \times cA$ where $t : V \to \widetilde{S}$ is a section and $f : V \to \Gamma(cA, cB)$ is a $V-$parametrized family of sections of $E$ over $cA$. The map above is given by $\mathrm{res}^W_V(s) = t$. Consequently, $\mathrm{res}^W_V$ is equivalent to the following product fibration, given by projection to the first factor
		$$\mathcal{F}_S(V) \times \maps(V, \Gamma(cA, E)) \to \mathcal{F}_S(V).$$
		 As a corollary, we obtain $\mathcal{F}(W) \cong \mathcal{F}(V) \times \maps(V, \Gamma(cA, E))$. As this isomorphism is natural under restrictions to open subsets $W' \subset W$, $V' = W' \cap S \subset V$, it establishes an isomorphism of sheaves
		$$i^*_V\mathcal{F} \cong i^*_V\maps(-, \Gamma(cA, E)) \times i^*_W\mathcal{F}_S$$
		By Lemma \ref{lem-fixedF}, $\maps(-, \Gamma(cA, E))$ is flexible and by Corollary \ref{cor-jetsflex}, $\mathcal{F}_S$ is flexible. As restriction and products of flexible sheaves are flexible, we obtain $i^*_V\mathcal{F}$ is flexible. Therefore, $\mathcal{F}$ is locally flexible and hence by Gromov's localization lemma \cite[p. 79]{Gromov_PDR}, $\mathcal{F}$ is flexible.\end{proof}

\medskip

\noindent {\it A description of $\HH^L_S$:}\\
Let $S<L$ denote strata of $X$. Let $\HH^L_S$ denote the restriction of $$\bH^L_S:=\hofib (i_{\bbar{S}}^*\FF_{\bbar{L}}\to  \FF_{\bbar{S}}  )$$ to the topmost stratum of definition of $\bH^L_S$, i.e.\ to the (open) stratum $S$. Equivalently,
$$\HH^L_S=\hofib (i_{{S}}^*\FF_{\bbar{L}}\to  \FF_{{S}}  ),$$
where we assume that a  section
$\psi_S \in  {\FF_{{S}}(S)}$ has been fixed, and homotopy fibers are computed with respect to $\psi_S$. 

Let $U\subset S$ be a local (Euclidean) chart. Note that a small normal neighborhood of $U$ in $\bbar L$ is of the form $U \times cA^L_S$, 
where $A^L_S$ is the link of $S$ in $\bbar{L}$. 
%Also, let $A^L_S$ denote the intersection of $\AAA^L_S$ with $L$.
Then a neighborhood $N_{SL}$ of $S$ in $\bbar L$ is an $A^L_S-$bundle over $S$. {Let $S'$ be the unique stratum in $E$ containing $\psi(S)$}. Let ${B_{S'}}$ be the link of $S'$ in ${E}$.
%Let $B^L_S$ denote the intersection of $\BB^L_S$ with $L'$.
Then $\HH^L_S (U)=\HH^L_S (U,\psi_S)$ consists of two components:
\begin{enumerate}
	\item A section of $E$ over $N_{SL}(U)$ {restricting to $\psi_S\vert_{U}$ over the zero section $U \subset N_{SL}(U)$}. Germinally, this is equivalent to a map from $U$ to 
	$\Gamma_c (cA^L_S, c{B_{S'}})$ as in the  proof of Corollary  \ref{cor-jetsflex} above.  Note that by Lemma \ref{lem-fixedF}, the sheaf $\maps(-, \Gamma_c (cA^L_S, c{B_{S'}})$ is flexible. We shall refer to this
	as the {\it germinal $L-$component}.
	\item A path of sections over $U$ in $\FF_{{S}}$ starting at $\psi_S|U$, i.e.\ a continuous map $h:[0,1] \to \FF_{{S}} (U)$, such that $h(0) = \psi_S|U$. Let $P_{\psi}(U)$ denote the collection of such maps.
\end{enumerate}

Let $\GG$ be  a sheaf on $S$ defined by $$\GG(U) = P_\psi(U).$$ 

\begin{lemma}\label{lem-psapaceflex}
	$\GG$ is flexible.
\end{lemma}

\begin{proof}
	We note that for $U$
	a local chart, the restriction $\GG_U$ of $\GG$ to $U$ is given by
	$$\GG_U(V) = \maps \big((I\times V , \{0\}\times V), (F_S,\psi_S)\big).$$
	The homotopy extension property from subcomplexes of $S$ to the fiber $F_S$ then gives the lemma.
	
	{Alternatively, we may use Lemma \ref{lem-mapsI2F} to deduce the lemma.}
\end{proof}

{Using the proof of Theorem \ref{thm-diffrlnflex}, we can explicitly compute the homotopy fiber $\HH^L_S$ for the sheaf $\FF$ of sections of $P : E \to X$. Indeed, observe that for quasitopological spaces $X, Y$, the homotopy fiber of the product fibration $X \times Y \to Y$ over a point $y \in Y$ is \emph{homeomorphic} to $X \times P_y Y$, where $P_y Y \subset \maps(I, Y)$ consists of the collection of maps $\gamma : I \to Y$ with $\gamma(0) = y$, with the inherited quasitopology. Therefore,
	$$\HH^L_S \cong \maps(-, \Gamma_c (cA^L_S, c{B_{S'}})) \times \GG.$$
}

\begin{prop}\label{prop-hlsflex}
	$\HH^L_S$ is flexible.
\end{prop}

\begin{proof}{By Lemma \ref{lem-fixedF}, $\maps(-, \Gamma_c(cA^L_S, cB_{S'}))$ is flexible. By Lemma \ref{lem-psapaceflex}, $\GG$ is flexible. Therefore, using Lemma \ref{lem-flexpdkt}, we conclude $\HH^L_S$ is flexible.}\end{proof}

\begin{comment}
	\begin{proof}
		Consider compact sets $K_1 \subset K_2 \subset S$. 
		By Lemma \ref{lem-psapaceflex}, given a path of sections $\sigma_t'$ over $K_1$ with $\sigma_0' = \psi|K_1$, there exists a path of sections $\sigma_t$  over $K_2$ such that
		\begin{enumerate}
			\item $\sigma_0|K_2 =  \psi|K_2$,
			\item $\sigma_t|K_1 = \sigma_t'$.
		\end{enumerate}
		Next, suppose that the germinal $L-$component of $\HH^L_S(K_1)=\HH^L_S(K_1, \psi)$ is given by
		$h_1 \in  \maps(K_1, \Gamma_c (cA^L_S, c{B_{S'}})$. 
		Flexibility of the sheaf $\maps(-, \Gamma_c (cA^L_S, c{B_{S'}})$ (Lemma \ref{lem-fixedF}) implies that there exists 
		$h_2 \in  \maps(K_1, \Gamma_c (cA^L_S, c{B_{S'}})$ such that $h_2|K_1=h_1$.
		Then $(\sigma_t, h_2)$ for $t \in [0,1]$
		is an element of $\HH^L_S(K_2)=\HH^L_S(K_2, \psi)$ extending $(\sigma_t', h_1) $. This completes the proof of flexibility of $\HH^L_S$.
	\end{proof}

\subsection{Smooth stratified h-principle}\label{sec-shprin}
As in Section \ref{sec-jbdlflex} above, let $\JJ^r_E, \JJ^r_{E,w}, \HH^r_E$ denote the stratified sheaves of controlled, weakly controlled  and holonomic  $r-$jets  as in Section \ref{sec-stratfdjet}.
\end{comment}

Recall Gromov's convention  \cite[Section 1.4.1]{Gromov_PDR} of referring to  an arbitrarily small but non-specified neighborhood of a set $K \subset X$
by $\op K$.
The following are direct adaptations of Gromov's definitions of the smooth $h-$principle for manifolds from \cite[p. 37]{Gromov_PDR} to the stratified context.
We spell these out  for completeness.

\begin{defn}\label{def-hprin}

	A stratified differential relation $\RR$  is said to satisfy the 
	
	\begin{enumerate}
		\item \emph{stratified $h-$principle near a
			subset} $K \subset X$ if for every section $\phi: U(K) \to \RR$ on a neighborhood $U(K)$ of $\bbar{K}$, there exists an open neighborhood $U'$ of $\bbar{K}$, such that $\phi\vert_{U'}$ is homotopic to a 
		a holonomic section.
		\item \emph{stratified  parametric $h-$principle} near $K$ if the map $$f \mapsto J^r_f$$ 
		from the
		space of solutions of  $\RR$ on $\op K$ to  the space of sections $\op K \to  \RR$ is a weak
		homotopy equivalence.
		\item \emph{stratified  $h-$principle for extensions of 
			%$C^k-$solutions of  
			$\RR$, 
			%for some $k \geq r$,
			from  $K_1$ to   $K_2 \supset K_1$} if for every 
		%$C^k-$
		section $\phi_0: \op K_2\to \RR$
		which is holonomic on $K_1$, there exists a %$C^k-$
		homotopy to a holonomic 
		%$C^k-$
		section
		$\phi_1$ by a homotopy of sections $\phi_t: \op K_2\to \RR$, $t \in [0, 1]$, such that 
		$\phi_t \vert\op K_2$ is constant
		in $t$. 
		\item \emph{parametric stratified  $h-$principle for extensions of 
			%$C^k-$solutions of  
			$\RR$, 
			%for some $k \geq r$, 
			from  $K_1$ to a  $K_2 \supset K_1$} if the map $f \mapsto J^r_f$
		from the space of solutions of $\RR$ on $K_2$  to   the space of sections $\op K_2 \to  \RR$
		which are holonomic on $K_1$, is a weak
		homotopy equivalence.
	\end{enumerate}
\end{defn}

\begin{comment}\label{rmk-shefh2smoothh} For $X=V$ a manifold, it is
	pointed out by Gromov \cite[Remark $A^\prime$, p. 76]{Gromov_PDR}, that the sheaf-theoretic h-principle (resp.\ parametric sheaf theoretic  h-principle) of Definition \ref{def-sheafh} implies the smooth 
	h-principle (resp.\ parametric  h-principle) of Definition \ref{def-hprin} in the special case that $X$ consists of a single stratum.
\end{comment}

\begin{comment}\label{rmk-shefh2smoothhstrat}
	This is an analog of  Remark \ref{rmk-shefh2smoothh} for stratified sheaves and jets over stratified spaces. The stratified sheaf-theoretic h-principle (resp.\ parametric sheaf theoretic  h-principle) of Definition \ref{def-ssheafh} implies the smooth stratified 
	h-principle (resp.\ parametric stratified  h-principle) of Definition \ref{def-hprin}. This is because 
	\begin{enumerate}
		\item Stratumwise, this is exactly a consequence of Remark \ref{rmk-shefh2smoothh}.
		\item Across strata, the restriction maps are natural by Lemma \ref{lem-f2f*}.
	\end{enumerate}
\end{comment}

\subsection{Holonomic approximation for jet sheaves}\label{sec-emhat}
We turn now to generalizing 
the smooth versions of the $h-$principle due to Eliashberg-Mishachev \cite{em-expo,em-book} to stratified spaces. The following is the \emph{holonomic approximation theorem} for smooth bundles over smooth manifolds.

\begin{theorem}\label{em-hat}\cite[Theorem 1.2.1]{em-expo}\cite[Theorem 3.1.1, p.20]{em-book}
	Let $V$ be a manifold, $E \to V$ be a smooth bundle, and  $K \subset V$ be a polyhedron of positive codimension.
	Let ${f \in \JJ^r_E(\op\, K)}$ be a formal section. Then for {any} $\ep  > 0$,
 {$\delta > 0$, 	there exist a} 
	diffeomorphism $h : V \to V$ with $$||h-{\mathrm{Id}}||_{C^0} < \delta,$$ and a holonomic section ${\til{f} \in \JJ^r_E(\op\, K)}$  such
	that 
	\begin{enumerate}
		\item the image $h(K)$ is contained in the domain of the definition of the section $f$, and
		\item $||\til{f}- f|\op h(K) ||_{C^0} < \ep$.
	\end{enumerate}
	In fact, $h$ may be chosen as the time one value of a diffeotopy $h_t: t \in [0,1]$, with
	$h_0$ equal to the identity, $h_1=h$, and 
	for all $t \in [0,1]$, 
	\begin{enumerate}
		\item the image ${h_t(K)}$ is contained in the domain of the definition of the section $f$, and
		\item $||h_t-{\mathrm{Id}}||_{C^0} < \delta$.
	\end{enumerate}
\end{theorem}

\begin{comment}
	Eliashberg-Mishachev generalize Theorem \ref{em-hat} to locally integrable relations $\RR$
	in \cite[Theorem 13.4.1, p. 127]{em-book}i.e.\ replacing $\JJ^r_E$ by any 
	\emph{locally integrable} relation $\RR$, they point out that their proof of the holonomic approximation theorem \cite[Theorem 3.1.1, p.20]{em-book} goes through. 

{Commented out the remark on locally integrable relations; find a way to phrase it in a way compatible with our definition of a differential relation.}
\end{comment}

 Recall that diffeotopies are smooth 1-parameter families of diffeomorphisms
\cite[p. 37]{Gromov_PDR}.

\begin{defn}\label{def-crsmallisotopynormal} Let $S$ be a stratum of a stratified space $X$, with link $A$, and $N_S$ a normal neighborhood of $S$ in $X$; hence $N_S$ is a $cA-$bundle over $S$. Fix local trivializations $\{U_i\}$ and compatible
	local product metrics $(g_S^i, g_{cA}^i)$ on $U_i \times cA \subset N_S$.
	A stratified diffeotopy ${\{h_t: t \in [0,1]\}}$ of $X$ supported in $N_S$ is said to be normally $\ep-$small in the $C^r$ norm
	if the following hold.
	\begin{enumerate}
		\item $\diam (h_t(s): t \in [0,1]) < \ep$ for all $s \in U_i \times cA \subset N_S$ and all $i$.
		\item Let ${\phi_{(t,x)} : cA(x) \to cA(h_t(x))}$ denote the map induced by $h_t$ from the cone ${cA(x)}$ 
		at the point $x \in S$ to the cone ${cA(h_t(x))}$ 
		at the point ${h_t(x)} \in S$. We demand that
		for all strata $J$ of $cA \setminus \{c_A\}$ and $y \in J$, the $C^r-$norms of  $\phi_{(t,x)}$ are bounded above by $\ep$ for all $x\in S$ and $t \in [0,1]$.  
	\end{enumerate}
\end{defn}

Note that the conclusion of Theorem \ref{em-hat} only ensures $C^0-$closeness in the conclusion.
However, Definition \ref{def-crformal} allows us to introduce the notion of $C^r-$closeness of formal sections \emph{in the normal direction}.

\begin{defn}\label{def-ncrclose}
	{We shall say that a pair of $C^r$-regular formal sections $\phi, \psi$ of $P: E \to X$ over $U \subset X$ are \emph{normally $\varepsilon$ $C^r$-close} if 
		\begin{enumerate}
			\item The continuous maps $\phi|\diag(U)$ and $\psi|\diag(U)$ are $\varepsilon$-close in the $C^0$-norm
			\item For every pair of strata $S < L$ intersecting $U$, and $S' = S \cap U, L' = L \cap U$, the germs
			$\phi_{S', n}, \psi_{S', n}$ are $C^r-$asymptotic (cf.\ Definition \ref{def-crformal}), in the following sense: for every $\varepsilon_1 > 0$, there exists $\delta > 0$ such that $\phi_{S', n}|N_\delta(S', L')$ and $\psi_{S',n}|N_\delta(S', L')$ are $\varepsilon_1$-close in the $C^r$-norm.
	\end{enumerate}}
\end{defn}

{\begin{rmk}\label{rmk-ncrclose}If for every stratum $S$ of $X$ intersecting $U$, $\phi|_S, \psi|_S$ are $C^r$-close, then $\phi, \psi$ are normally $C^r$-close as well. The converse need not be true.\end{rmk}}

To extend Theorem \ref{em-hat} to stratified spaces, we need some basic differential topology facts about stratified spaces. The following allows us to extend diffeotopies across strata.

\begin{lemma}\label{lem-isotopyext} Let $(X, \Sigma, \mathcal{N})$ be an abstractly stratified space and $S \in \Sigma$ be a stratum. Let $h : S \times I \to S$ be a diffeotopy of $S$ supported on a {compact} set $K \subset S$. Then there exists an extension {of $h$} to a stratified diffeotopy $H : X \times I \to X$. 
	If $h$ is $C^0$-small, the extension $H$ is normally $C^r-$small for any $r>0$.
	Moreover, if $h$ is $C^r$-small, $H$ is stratumwise $C^r-$small as well.
\end{lemma}

As a prerequisite to the proof and for later use as well, we will state and prove a general result regarding fiber bundles. The main content of this result is a fiber-preserving analogue of the homotopy extension property. To set it up, let $p : E \to X$, $p' : E \to X'$ be fiber bundles with fibers spaces $Y$ and $Y'$, respectively. Fix basepoints $y \in Y$, $y' \in Y'$ and suppose furthermore that $p$, $p'$ have as their structure groups the groups of basepoint-preserving homeomorphisms $\mathrm{Homeo}(Y, y)$ and $\mathrm{Homeo}(Y', y')$, respectively. Let $s : X \to E, s' : X' \to E'$ denote the canonical sections of $p$ and $p'$ parametrizing the fiberwise basepoints. Note that $p \times \mathrm{id} : E \times I \to X \times I$ is also a fiber bundle with fiber space $Y$ and structure group $\mathrm{Homeo}(Y, y)$, with a canonical section $s \times \mathrm{id} : X \times I \to E \times I$ traced out by the fiberwise basepoints, as before. Further, suppose $X, X', E, E'$ are equipped with metrics compatible with their topology.

\begin{lemma}\label{lem-fiberhep}
		Suppose $f : X \to X'$ and $g : E \to E'$ are maps such that
		\begin{enumerate}
			\item $g$ covers $f$, i.e.\ $p' \circ g = f \circ p$, and
			\item $g$ preserves the sections $s, s'$, i.e.\ $g \circ s = s' \circ f$.
		\end{enumerate}
		Let $F : X \times I \to X'$ be a homotopy such that $F|X \times \{0\} = f$. Then, there exists a map $G : E \times I \to E'$ such that 
		\begin{enumerate}
			\item $G$ covers $F$, i.e.\ $p \circ F = G \circ (p \times \mathrm{id})$, and
			\item $G$ preserves the sections $s \times \mathrm{id}, s'$, i.e.\ $G \circ (s \times \mathrm{id}) = s' \circ F$.
		\end{enumerate}
		Moreover, if $\mathrm{diam}(F(\{x\} \times I)) < \varepsilon$ uniformly for all $x \in X$, then we may choose $G$ such that $\mathrm{diam}(G(\{e\} \times I)) < \varepsilon$ uniformly for all $e \in E$.
	\end{lemma}

\begin{proof}  Consider the fiber bundle $F^* E'$ over $X \times I$. This is a principal $\mathrm{Homeo}(Y', y')$-bundle over $X \times I$ which restricts to $f^* E'$ over $X \times \{0\}$, therefore there is an isomorphism of principal $\mathrm{Homeo}(Y', y')$-bundles $(f^* E') \times I \to F^* E'$ over $X \times I$. The map $g : E \to E'$ covering $f : X \to X'$ furnishes a fiber-preserving map $h : E \to f^* E'$ covering the identity map on $X$. We collect all these maps in the following commutative diagram:
		$$
		\begin{CD}
			E \times I @>{h \times \mathrm{id}}>> (f^*E') \times I @>{\cong}>> F^*E' @>>> E' \\
			@V{p\times \mathrm{id}}VV @VVV @VVV @V{p'}VV\\
			X \times I @= X \times I @= X \times I @>{F}>> X'
		\end{CD}
		$$
	Let $G : E \times I \to E'$  be the composition of the three horizontal arrows on the top. Then by commutativity of the outer rectangle, $G$ satisfies Condition $(1)$. Next, observe that the leftmost and the rightmost squares satisfy Condition $(2)$.  Indeed, the canonical sections of $F^* E'$ and $(f^* E') \times I$ are furnished by the  basepoint $y'$ as their structure groups are in $\mathrm{Homeo}(Y', y')$. The middle square is an isomorphism of principal $\mathrm{Homeo}(Y', y')$-bundles; hence it must necessarily preserve the relevant canonical sections. This proves that $G$ satisfies Condition $(2)$ as well.
	
	For the final assertion, observe that as $F(\cdot, t)$ stays uniformly $\varepsilon$-close to $f$ for all $t \in I$, the cocycles of $F^* E'$ are $\varepsilon$-close to those of $(f^*E') \times I$. By  the proof of Proposition 1.7 in \cite[p. 20]{hatcher-notes}, we see that this implies that  the fiber-preserving homeomorphism $(f^* E') \times I \to F^* E'$ of the top horizontal arrow in the middle square can be chosen to be uniformly $\varepsilon$-close to identity, where both sides are consider as metric subspaces of $E' \times X \times I$. Thus, the composition of the last two top horizontal arrows $(f^* E') \times I \to E'$ is uniformly $\varepsilon$-close to the map induced by the constant homotopy $X \times I \to X'$, $(x, t) \mapsto f(x)$. The image of $\{z\} \times I \subset (f^* E') \times I \to E'$ under the latter has diameter $0$. Therefore, under the former, it has diameter uniformly bounded by $\varepsilon$. This shows that  the composition $G : E \times I \to E'$ satisfies the desired property $\diam G(\{e\} \times I) < \varepsilon$ for all $e \in E$.
\end{proof}

\begin{proof}[Proof of Lemma \ref{lem-isotopyext}] Let $A$ denote the link of $S$ in $X$.
	A tubular neighborhood $N_S$ of $S$ in $X$ is a $cA$-bundle over $S$, i.e.\ a fiber bundle over $S$ with fiber $cA$ (cf.\ Corollary \ref{cor-trivialzn}). The structure group of this bundle is $\mathrm{Homeo}(cA, \{c_A\})$ where $\{c_A\} \subset cA$ is the cone point. By Lemma \ref{lem-fiberhep}, we can extend $h : S \times I \to S$ to a homotopy $f : N_S \times I \to N_S$. Moreover, since $h(\cdot, t)$ is a diffeomorphism, $f(\cdot, t)$ is an isomorphism of bundles, for all $t \in I$.
	
	Next, we construct a stratified diffeotopy {$\tilde{h} : N_S \times [0, 1] \to N_S$ by defining $\tilde{h}(x, t) = f(x, 3t)$ for $(x, t) \in N_S \times [0, 1/3]$, $\tilde{h}(x, t) = f(x, 2-3t)$ for $(x, t) \in N_S \times [1/3, 2/3]$, and $\tilde{h}(x, t) = x$ for $(x, t) \in N_S \times [2/3, 1]$, smoothing at $N_S \times \{1/3\}$ and $N_S \times \{2/3\}$ if required. We extend to a diffeotopy $H : X \times I \to X$ by defining $H(x, t) = x$ for all $x \in X \setminus K$ and $t \in [0,1]$.}
	
	We now prove the second assertion. Note first that $N_S$ is equipped with a continuous metric that is stratumwise smooth. We can change the metric to an equivalent metric that is locally a product metric on $N_S|K$ thinking of $N_S$ as a $cA$ bundle over $S$ to satisfy the conditions of Definition \ref{def-crsmallisotopynormal}. Then, locally on any $U \times cA$ we extend $h$ by the identity on the second coordinate. The resulting extension  is then normally trivial, in particular, normally $C^r-$small on $N_S|K$. Completing the extension to a diffeotopy $H : X \times I \to X$ as above, we see that $H$ is normally $C^r-$small for any $r>0$.
	
	  Finally note that for any stratum $L$ such that $S < L$, the $C^r$-distance of $H(\cdot, t)$ and $\mathrm{id}$ on $L$ is comparable to the sum of the $C^r$-distance restricted to $S$ and the normal $C^r$-distance.
  The last assertion follows.  \end{proof}

Combining Theorem \ref{em-hat} and Lemma \ref{lem-isotopyext} we have:
\begin{cor}\label{cor-isotopyext} Let $X, S$ be as in Lemma \ref{lem-isotopyext} above,
	$P: E \to X$  be a {stratified} bundle, and  $K \subset X$  be a substratified space of positive codimension. Let $A$ denote the link of $S$ and $N_S$ a normal neighborhood of $S$ given by a {$cA-$bundle} over $S$.
	Let $K_S = K \cap S$, and let  $N(K_S) $ denote the restriction of the bundle $N_S$ to $K_S$. 
	Let ${f \in \sjr^r_E(\op\, K)}$ be a $C^r-$regular formal section. Then for any $\varepsilon > 0,\delta > 0$, there exist
	\begin{enumerate}
		\item  a
		diffeotopy $h_t : S \to S$ with $h=h_1$ and  $||h_t- {\mathrm{Id}}||_{C^0} < \delta$ for all $t \in [0,1]$,
		\item an extension $H_t: X \to X $ of $h_t$ supported in a small neighborhood of  $N(K_S) $,
	\end{enumerate}
	and a holonomic section ${\til{f} \in \sjr^r_E(\op\, K_S)}$  such
	that 
	\begin{enumerate}
		\item the image $h(K_S)$ is contained in the domain of the definition of the section $f$, 
		\item $||\til{f}- f|\op_X h(K_S) ||_{C^0} < \ep$,
		\item $\til{f}$ and $f|\op_X h(K_S)$ are normally $C^r-$close on $N_S$.
	\end{enumerate}
\end{cor}

\begin{proof}
	The existence of a diffeotopy $h_t : S \to S$ with $||h_t-{\mathrm{Id}}||_{C^0} < \delta$ such that 
	\begin{enumerate}
		\item for all $t\in [0,1]$, the image  $ h_t(K_S)$ is contained in the domain of the definition of the section $f$, and 
		\item there exists a holonomic section $\til{f}_S : \op_S h(K_S) \to E$ such that $||{\til{f}_S}- f|\op_S h(K_S) ||_{C^0} < \ep$,
	\end{enumerate}
	is guaranteed by Theorem \ref{em-hat}.
	Note that so far ${\til{f}_S}$ is defined only on $S$, and the neighborhood $\op_S h(K_S)$ is only taken within $S$. 
	
	It remains to extend  ${\til{f}_S}$ into $N_S$ and extend the neighborhood $\op_S h(K_S)$ to an open neighborhood $\op_X h(K_S)$. We first apply Lemma \ref{lem-isotopyext} to $h_t$ to obtain  a stratified diffeotopy $H_t$ such that 
	\begin{enumerate}
		\item $H_t$ is supported in a small neighborhood of  $N(K_S) \subset N_S$. 
		\item $H_t$  is normally $C^r-$small.
	\end{enumerate}
	Let $H=H_1$. Then the first item above guarantees that  $H$ is defined and possibly unequal to the identity on $N(K_S)$.
	Let $cA(x)$ denote the normal cone of $S$ in $X$ at $x$. Note that ${H(cA(x)) = cA(h(x))}$ (in fact, the proof of Lemma \ref{lem-isotopyext} shows that $H$ may be chosen to be the identity in the normal coordinate). 
	
	To extend ${\til{f}_S}$ into $N_S$ holonomically, it suffices to define a holonomic {extension}
	of  ${\til{f}_S}$  on the restriction of $N_S$ to ${K_S}$. By Remark \ref{rmk-nfisholo}, the normal formal section $f_{S, n}$ associated to $f$ is a holonomic section of $E$ over $N_S$. Our main strategy here is to ``graft the conical component of $f_{S, n}$ with $\til{f}_S$" to build a section $\til{f} : \op_X h(K_S) \to E$. The associated holonomic stratified $r$-jet would be the desired element $\til{f} \in \sjr^r(\op K_S)$. We accomplish this in an indirect manner in light  of the warning at the end of Remark \ref{rmk-nfisholo}.
	
	As $\widetilde{f}_S$ and $f|S$ are $\varepsilon$-close as formal $r$-jets, the sections $\widetilde{f}|_S$ and $ f_{S, n} $ from $ \op_S(K) \to E$ are $\varepsilon-$close as well. Let $\widetilde{S}$ denote the unique stratum containing the image of $S$ under  $\widetilde{f}|_S$ and $ f_{S, n} $. (The  existence of such an $\widetilde{S}$ is guaranteed if $\varepsilon > 0$ is sufficiently small.) Let us moreover fix an open neighborhood $U \subset S$ of $K_S$ in $S$ which is contained in the domains of definition of both $\widetilde{f}|_S$ and $f_{S, n}$. Let $\sigma : U \times I \to \widetilde{S}$ be a homotopy through sections $\sigma_t : U \to \widetilde{S}$, between $\sigma_0 = f_{S, n}$ and $\sigma_1 = \widetilde{f}|_S$. By Lemma \ref{lem-fiberhep}, there exists an extension 
		$$\widetilde{\sigma} : N_S(U) \times I \to N_{\widetilde{S}} \subset E$$
		where  $N_{\widetilde{S}}$ is the normal cone bundle of $\widetilde{S}$ in $E$, and  $N_S(U)$ denotes restriction of $N_S$ to $U$. If $\varepsilon > 0$ is sufficiently small, we may ensure that $h(K_S) \subset U$. Let $\widetilde{f} := \widetilde{\sigma}_1|\op_X h(K_S)$. Then $\widetilde{f}$ is a holonomic extension of $\widetilde{f}_S|\op_X h(K_S)$  into a neighborhood of $K_S$ in $X$. Therefore,
		\begin{enumerate}
			\item $\til f$ is holonomic on $\op_X h(K_S)$ by construction.
			\item Since $f$ is $C^r-$regular by hypothesis and equals $\til f$ on germs of normal cones on $S$, $f$  and $\til f$ are normally $C^r-$close on $N_S$.
		\end{enumerate} 
		This completes the proof.
\end{proof}

It is a well-known fact that
for a smooth bundle $P: E \to M$ over a compact manifold $M$, any two sufficiently close sections $s_1, s_2:
M \to E$ are smoothly isotopic through sections. To see this, fix $s_1(M) = M_1 \subset E$, and let $N_\ep M_1 \subset E$ be a regular normal neighborhood obtained, for instance, by equipping $E$ with  a Riemannian metric, and using the exponential map $\exp$ to exponentiate from the normal bundle $N_EM_1$ to $M_1$ down to $E$. Let $H_t$ denote the linear homotopy on $N_EM_1$ that collapses all the linear fibers down to $M_1$ (identified with the zero section of $N_EM_1$).
If $s_2(M) \subset N_\ep M_1$, then $\exp \circ H_t \circ \exp^{-1}$ gives a smooth isotopy of $s_2$ to $s_1$.

The same proof generalizes in a straightforward way to stratified spaces.
Let $P: E \to Y$ be a stratified fiber bundle over a compact stratified space $Y$ and $s_1: Y \to E$ be a stratified section. Let $s_1(Y) = Y_1$. Embedding $E$  in a smooth manifold $E'$ by Theorem \ref{natsume}, equipping $E'$ with a Riemannian 
metric $g$, and restricting $g$ to $E$, we obtain {a} stratumwise 
Riemannian metric $g_s$ on $E$.
Let $N_\ep Y_1 \subset E$ be a regular normal neighborhood of $Y_1$ in $E$ obtained by exponentiating the stratumwise normal bundle with respect to the stratumwise 
Riemannian metric $g_s$. Then a fiberwise linear homotopy exists as  in the manifold case.
This establishes the following.

\begin{lemma}\label{lem-closebyisotopic}
	Let $P: E \to Y$ be a stratified fiber bundle over a compact stratified space $Y$ and $s_1: Y \to E$ be a smooth stratified section. Further, assume that $E$ is equipped with a metric $d_s$ that is stratumwise smooth Riemannian. Let $s_2: Y \to E$ be a smooth stratified section. Then for all $\ep >0$ there exists $\delta >0$ such that if $d_s(s_1(y), s_2(y))<\delta$ for all $y \in Y$, then there exists a stratified isotopy
	$H: Y \times [0,1] \to E$ such that 
	\begin{enumerate}
		\item $H(y,0) = s_1(y)$
		\item  $H(y,1) = s_1(y)$
		\item $d_s(s_1(y), H(y,t))<\ep$ for all $y \in Y$ and $t \in [0,1]$.
	\end{enumerate} 
	
\end{lemma}

\begin{rmk}\label{rmk-closebyisotopic}
	Compactness is not essential in the proof of Lemma \ref{lem-closebyisotopic}. All that was required was the existence of a normal neighborhood as the image of an open neighborhood of the zero-section. This goes through for $Y$ non-compact as well provided we allow the thickness of the open neighborhood to be non-constant.
\end{rmk}

As a consequence of Lemma \ref{lem-closebyisotopic} we have the following:

\begin{cor}\label{cor-interpolate}
	Let $P: E \times (-\ep,1+\ep) \to Y\times (-\ep,1+\ep)$ be a stratified fiber bundle over a  stratified space $Y \times (-\ep,1+\ep)$, where $Y$ is a compact stratified space.
	Let $s_1, s_2:  Y\times (-\ep,1+\ep) \to  E \times (-\ep,1+\ep) $ be two stratified smooth sections
	that are $\delta-$close in the $C^r-$norm. Then there exists a section $s_3:  Y\times (-\ep,1+\ep) \to  E \times (-\ep,1+\ep) $ such that $s_3$ interpolates smoothly between $s_1, s_2$, i.e.\
	\begin{enumerate}
		\item $s_3|(-\ep,0] =s_1|(-\ep,0]$
		\item $s_3|[1,1+\ep) =s_2|[1,1+\ep)$
		\item $s_3$ is $\delta-$close to both $s_1, s_2$ in the $C^r-$norm.
	\end{enumerate}
	
	More generally, if there exists a stratum $S$ of $Y$ and a submanifold $S'$ of $S$ such that
	$s_1|S' \times (-\ep,1+\ep)$ and $s_2|S' \times (-\ep,1+\ep)$   are $\delta-$close in the $C^r-$norm, then there exists a section $s_3:  Y\times (-\ep,1+\ep) \to  E \times (-\ep,1+\ep) $ such that $s_3$ interpolates smoothly between $s_1, s_2$ with $C^r-$closeness along $S'$, i.e.\
	\begin{enumerate}
		\item $s_3|(-\ep,0] =s_1|(-\ep,0]$
		\item $s_3|[1,1+\ep) =s_2|[1,1+\ep)$
		\item $s_3|S' \times (-\ep,1+\ep)$ is $\delta-$close to both $s_1|S' \times (-\ep,1+\ep), s_2|S' \times (-\ep,1+\ep)$ in the $C^r-$norm.
	\end{enumerate}
\end{cor}

\begin{proof}
	Let $H_t$  be the homotopy between $s_1, s_2$ in the discussion preceding Lemma \ref{lem-closebyisotopic}. Setting $s_3'(x,r) =H_r(x,r)$ for $r \in [0,1]$ furnishes a linear interpolation. Smoothing slightly at the end-points, e.g.\ by choosing a smooth monotonically increasing bijective function $g: [0,1]\to [0,1]$, and setting 
	$s_3(x,r) =H_r(x,g(r))$ for $r \in [0,1]$ gives the required $s_3$.
\end{proof}

We are now in a position to prove the stratified version of Theorem \ref{em-hat}. All substratified spaces $K \subset X$ below will be assumed to be tamely embedded below, i.e.\ if $K$ is non-compact, then the closure $\bbar{K}$ is a deformation retract of a small regular neighborhood.

\begin{theorem}\label{em-hats}
	Let $X$ be {an abstractly stratified space} equipped with a compatible metric, $E \to X$ be a stratified bundle, and $K \subset X$ be a relatively compact
	stratified subspace of positive codimension.
	Let {$f \in \sjr^r_E(\op K)$} be a $C^r-$regular formal section. Then for arbitrarily small $\ep > 0$, 
	$\delta > 0$, there exist a stratified 
	diffeomorphism $h : X \to X$ with $$||h-{\mathrm{Id}}||_{C^0} < \delta,$$ and a stratified holonomic section {$\til{f} \in \sjr^r_E(\op K)$} such
	that 
	\begin{enumerate}
		\item the image $h(K)$ is contained in the domain of  definition of the section $f$, 
		\item $||\til{f}- f|\op\, h(K) ||_{C^0} < \ep$.
		\item $\til{f}, f|\op\, h(K)$ are normally $\varepsilon$ $C^r$-close.
	\end{enumerate}
	The same applies for $\sjr^r_{E,w}$ in place of $\sjr^r_E$.
\end{theorem}

\begin{proof}
The proof proceeds by induction on the depth (cf.\ Definition \ref{def-Idec}) of $X$. If $X$ has depth one, it is a manifold, and Theorem \ref{em-hat} furnishes the result.
	
	Let $S$ be the lowest stratum (i.e.\ the stratum of greatest depth) that $K$ intersects. We note that there might be more than one such minimal stratum $S_i$ of possibly varying depths with $K \cap S_i \neq \emptyset$, in which case we shall repeat the argument below for each of these. For convenience of exposition, we assume there is a unique such $S$. Let $K_S = K \cap S$. Then $K_S$ is compact; else $K$ would intersect {a stratum of depth lower than that of $S$, but there is none such.}
	
	Theorem \ref{em-hat} ensures the existence of a {self-}diffeomorphism {$h_S$} of $S$ supported in a neighborhood of $K_S = K \cap S$ and a holonomic section {$\til{f_S} \in \sjr^r_E(\op h(K_S))$} satisfying the conclusions of the theorem, but only on $S$. Further, Corollary \ref{cor-isotopyext}
	allows us to 
	\begin{enumerate}
		\item extend {$h_S$} to a stratified {self-diffeomorphism $h_e$ of all of $X$} supported in
		$N_S \cap \op_X (K_S)$.
		\item extend {$\til{f_S} \in \sjr^r_E(\op K_S)$ to a stratified holonomic section $\til{f_e} \in \sjr^r_E(\op_X h(K_S))$} defined on an open neighborhood $ \op_X\, h(K_S)$ of $h(K_S)$ in $X$.
	\end{enumerate}
	We may assume without loss of generality that $\op_X\, h(K_S)$ is the restriction of the normal bundle $N_S$ to $\op_S\, h(K)$.
	
	Next, %let $X_1 = X\setminus S$.  to  More precisely, we need to
	delete a (very) small closed neighborhood $N_\eta(S)$ of $S$ in $X$ to obtain $X_1$, and let $K_1 = K \cap X_1 \subset X_1$. Then the depth of $X_1$ is strictly less than $X$, and induction may be applied to obtain
	\begin{enumerate}
		\item A stratified {self-diffeomorphism $h_1$ of $X_1$}, supported in $ \op_{X_1} K_1$ such that $h_1$ is $\delta-$close to the identity in the {$C^0$} norm.
		\item a stratified holonomic section {$\til{f_1} \in \sjr^r_E(\op_{X_1} K)$} defined on an open neighborhood of $h(K_1)$ in $X_1$.
	\end{enumerate}
	such that 
	\begin{enumerate}
		\item the image  $ h_1(K_1)$ is contained in the domain of  definition of the section $f$, 
		\item $||\til{f_1}- f|\op\, h_1(K_1) ||_{C^0} < \ep$.
		\item $\til{f_1}, f|\op_{X_1}\, h(K_1)$ are normally $\varepsilon$ $C^r$-close on $X_1$.
	\end{enumerate}
	
	Composing $h_1$ with a further {$C^0-$small} stratified diffeomorphism $h_2$ supported on ${\mathcal{A} = N_{2\eta}(S)\setminus \bbar{N_{\eta}(S)}}$, we may assume that $h_2 \circ h_1$ and $h_e$ agree on $\AAA$. Let $h$ be the  stratified diffeomorphism  obtained by pasting these two diffeomorphisms along $\AAA$. Note that $h$ is $C^0-$small, and hence lifts to give a bundle map from $E|\op_X K$ to $E|\op_X h(K)$ covering $h: \op_X K \to \op_X h(K)$.
	
	Choosing $\eta$ in $N_\eta(S)$ sufficiently small, we may assume that  the domain of  the extension $\til{f_e}$ given by the normal bundle to $\op_S h(K)$ constructed earlier in the proof has normal fibers of diameter at least $2\eta$. Thus, on the stratified "annulus" ${\AAA'} = (N_{2\eta}(S)\setminus N_{\eta}(S)) \cap \op_X \, h(K)$, we have two holonomic sections $\til{f_e}$  and $\til{f_1}$. (We are assuming here that the sections have been composed with the bundle map covering $h: \op_X K \to \op_X h(K)$ described in the previous paragraph.)

	Since $\eta$ is small, and since both $\til{f_e}$  and $\til{f_1}$ are normally $C^r-$close to $f$ on {$\AAA'$}, they are close to each other.
	Hence,  using Corollary \ref{cor-interpolate}, we can  interpolate between the sections  $\til{f_e}|N_\eta(S)$ and $\til{f_1}|(X\setminus N_{2\eta}(S)) $ to  obtain 
	a holonomic section $\til f$ on all of $\op_X K$. Since $\eta$ is small, all three conclusions of the theorem are satisfied by $h$ and $\til f$.
\end{proof}

\begin{rmk}
\rm{
{
We pause to contrast the statement of Theorem \ref{em-hats} with Theorem \ref{thm-diffrlnflex}. Theorem \ref{thm-diffrlnflex} implies, in particular, that any formal section of the stratified bundle $E \to X$ over any open set is homotopic to a holonomic section. But this homotopy need not be $C^0-$small.
}

{For instance, let $X$ be a stratified space, and $X \times \Bbb R \to X$ be the trivial bundle. From Remark \ref{rmk-formalasgerms}, we see that a formal section over $X$ consists of a pair $(f, g)$, where $f : X \to \Bbb R$ is a function, and $g : tX \to \Bbb R$ is a function on the tangent microbundle $tX$ of $X$ such that $g$ restricts to $f$ on the zero section of $tX$. Indeed, the graph of $f$ defines the base of the formal section and the graph of $g_x := g|\op_{t_x X}(\{x\})$ at every point $x \in X$ defines the decorating field of section-germs. If $f|\op_X(\{x\}) \neq g_x$, we cannot approximate $(f, g)$ in the $C^0-$norm by a holonomic section. This is because holonomic sections would necessarily be of the form $(h, h_*)$ for some function $h : X \to \Bbb R$, where $h_* : tX \to \Bbb R$ is the `micro-derivative' of $h$, given by $h_*(x, y) = h(y)$ for all $(x, y) \in tX$.
}

{
However, the conclusion of Theorem \ref{em-hats} produces a holonomic section arbitrarily $C^0-$close to a chosen formal section, hence homotopic to the formal section by a  $C^0-$small homotopy. The important caveat here is that the domain of the section $\op(K)$ must be a thickening of a positive codimensional stratified subspace $K \subset X$. Moreover, we are allowed to make $C^0-$small perturbations to the domain, so as to create enough `wiggle-room' in the directions normal to $K$ in $X$, for the approximation to hold.
}
}
\end{rmk}

\subsection{Application: Immersions}\label{sec-immn}

A host of examples of open $\diff-$invariant relations have been enumerated by Eliashberg-Mishachev \cite{em-expo} in the manifold context and the holonomic approximation theorem deduced for these:

\begin{enumerate}
\item open
manifold immersions, \item open
manifold submersions, \item open
manifold k-mersions (i.e.\ mappings of rank at least $k$), \item mappings with nondegenerate
higher-order osculating spaces, \item mappings transversal to foliations, or more generally, to
arbitrary distributions, \item construction of generating system of exact differential $k-$forms, \item symplectic
and contact structures on open manifolds, etc.
\end{enumerate}

Most of these (in particular, immersions, submersions and $k-$mersions of open manifolds) have natural potential generalizations to the stratified context,  by replacing the use of Theorem \ref{em-hat} by Theorem \ref{em-hats}. \\

\noindent  {\bf Stratified Immersions:}
We illustrate the extra ingredient necessary over and above \cite{em-book,Gromov_PDR} by studying the sheaf of stratified immersions in this paper. We postpone a more detailed treatment of applications to subsequent work. Here, we shall prove a  stratified analog of the Smale-Hirsch theorem \cite[Chapter 8.2]{em-book} (see Theorem \ref{thm-immns} below).

\begin{defn}\label{def-immns}
		Let $X, Y$ be abstractly stratified spaces, and let $tX$ and $tY$ denote the tangent microbundles to $X$ and $Y$ respectively.
		A \emph{stratified  immersion} of $X$ into $Y$ is a weakly controlled map $i : X \to Y$ such that the induced 
		map $i_*: tX \to tY$ is a fiberwise stratified embedding of stratified spaces.
	\end{defn}

	A stratified immersion $i : X \to Y$ will be said to be of \emph{positive codimension}, if the following is true: for any stratum $S$ of $X$, let $S'$ denote the unique stratum of $Y$ such that $i(S) \subset S'$. Then $i(S) \subset S'$ is an immersed submanifold of positive codimension. Let $\imm^\mathfrak{c}(X,Y)$ denote the quasitopological space of \emph{positive codimension stratified  immersions} of a stratified space $X$ in a stratified space $Y$ such that the stratified  immersions are of configuration $\mathfrak{c}$.
	
	Given  stratified spaces $(Y, \Sigma')$  and  $(X, \Sigma)$, 
	and a configuration $\mathfrak{c}$ (cf.\ Definition \ref{def-config}), we may construct a sheaf $\imm^{\mathfrak{c}}(-, Y)$ on $X$ by defining, for each element $U \in \mathrm{Str}(X, \Sigma)$ of the stratified site (Definition \ref{def-stratfdsite}), 
		$$\imm^\mathfrak{c}(-, Y)(U) = \imm^\mathfrak{c}(U, Y)$$
		where $U$, being an open subset in a stratum-closure of $X$, is equipped with the induced stratification from $X$. The restriction maps are defined simply by restriction of the underlying map of an immersion to a smaller subset. 
	
For any stratum $S \in \Sigma$, let  $\imm_{\bbar{S}}^\mathfrak{c}(-, Y)$ denote the sheaf of positive codimension stratified immersions of elements of the stratified site of the stratified space $(\bbar{S}, \bbar{S} \cap \Sigma)$ as in Definition \ref{def-sssheaf}. Let $\imm_S^\mathfrak{c}(-, Y) := i^*_S \imm_{\bbar{S}}^\mathfrak{c}(-, Y)$ denote its restriction to the open stratum $S$.}
	
	Note that $\imm^\mathfrak{c}(-, Y)$ is a stratified subsheaf of $\maps(-, Y)$, where $\maps(-, Y)$ is  identified with the sheaf of sections of the surjective map $P : X \times Y \to X$. $P$ is an example of a stratumwise bundle (Definition \ref{def-stratumwisebdl}), but not a stratified bundle (Example \ref{eg-pdkt}). Nevertheless, by Remark \ref{rmk-formalasgerms}, we may describe sections of the Gromov diagonal construction $\maps(-, Y)^\bullet$ over an element of the site $U \in \str(X, \Sigma)$ in terms of germs of maps $\sigma : (t(U), U) \to Y$. For every $p \in U$, we may consider $\sigma\vert {t_p U} : t_p U \to Y$ as a map $t_p U \to t_{\sigma(p)} Y$. In this process, we can identify sections of $\maps(-, Y)^\bullet$ over $U$ as microbundle morphisms $t U \to t Y$ covering a map $U \to Y$.
	
	\begin{defn}A \emph{stratified formal immersion} of $X$ into $Y$ is a pair $(F, f)$ consisting of 
			\begin{enumerate} 
				\item A weakly controlled map $f : X \to Y$,
				\item A fiber-preserving microbundle morphism $F : tX \to tY$,
			\end{enumerate} 
			such that $F$ covers $f$, and $F$ is  a fiberwise stratified embedding.
		\end{defn}

It follows from the discussion above that sections of the Gromov diagonal construction $\imm^\mathfrak{c}(-, Y)^\bullet$ over $U \in \str(X, \Sigma)$ can be identified with stratified formal immersions $(F, f)$ of $U$ into $Y$, where $f : X \to Y$ is of configuration $\mathfrak{c}$. The canonical inclusion $\imm^\mathfrak{c}(-, Y) \hookrightarrow \imm^\mathfrak{c}(-, Y)^\bullet$ over $U$ is given by sending a stratified (holonomic) immersion $i : X \to Y$ to the stratified formal immersion $(i_*, i)$ where $i_* : tX \to tY$ is given by the restriction of $i \times i : X \times X \to Y \times Y$ to the diagonal.
	
	\begin{comment}
		Note that in order  to study $\imm(X,Y)$, we need to look at sections of $P: X\times Y \to X$
		and $tP: tX\times tY \to tX$. Both are examples of stratumwise bundles (Definition \ref{def-stratumwisebdl}), but not examples of stratified  bundles (Example \ref{eg-pdkt}).
		We, nevertheless, prove the following:
	\end{comment}
	
	\begin{theorem}\label{thm-immns} For any stratified spaces $(X, \Sigma)$  and $(Y, \Sigma')$,  and configuration $\mathfrak{c}$ as above,
	the stratified continuous sheaf $\imm^\mathfrak{c}(-,Y)$ on $X$ satisfies the parametric $h-$principle.
	\end{theorem}

	\begin{proof}
		By Theorem \ref{thm-hofibsflexg}, it suffices  to check flexibility of $\HH^L_S$.
		Let $S < L$ be strata in $X$ and let $S' = \mathfrak{c}(S), L' = \mathfrak{c}(L)$ be the strata in $Y$ corresponding to $S, L$ through the configuration $\mathfrak{c}$. Let $t_{S,L}$ denote the microtangent bundle to $S$ in $\bbar L$. Let $A_L$ (resp.\ $B_{L'}$) denote the link of $S$ (resp.\ $S'$) in $\bbar{L}$ (resp.\ $\bbar{L'}$), and $\{c\} \subset cA_L$ (resp.\ $\{c'\} \subset cB_{L'}$) denote the cone points of the respective normal cones. Then the microtangent bundle $t_{S,L}$ of $S$ in $L$ is of the form 
			$$t_{S,L}  = TS \oplus N_{S,L},$$
			where $TS$ is the tangent bundle to $S$, and $N_{S,L}$ is a $cA_L-$bundle over $S$. We refer to $N_{S,L}$ as the normal cone bundle of $S$ in $L$. The morphism 
			$$\mathrm{res}^L_S : i^*_S \imm^{\mathfrak{c}}_{\bbar{L}}(-, Y) \to \imm_S^{\mathfrak{c}}(-, Y)$$
			is given by restricting a germ of a stratified immersion defined on a neighborhood $N_{S, L}$ of $S$ in $\bbar{L}$, to the zero section $S \subset N_{S, L}$. Let $U \subset S$ be a chart such that $N_{S, L}$ trivializes as $U \times cA_L$ over $U$. Then, elements of $i^*_S\imm^{\mathfrak{c}}_{\bbar{L}}(U, Y)$ consist of germs of stratified immersions 
			$$\varphi : (U \times cA_L, U \times \{c\}) \to Y$$ 
			such that $\varphi|U \times \{c\}$ is an immersion of $U$ in $S'$. Observe that, for every $p \in U$, 
			$$\varphi\vert_{\{p\} \times cA_L} : (cA_L, \{c\}) \to (cB_{L'}, \{c'\}) \subset Y$$ 
			is a germ of a stratified immersion at the cone point $c'$. Here, the cone  $cB_{L'}$ is determined by the configuration  $\mathfrak{c}$. In other  words, the configuration  $\mathfrak{c}$ induces a germ of a stratified embedding  at the cone point. Indeed, the microtangent space $t_c(cA_L)$ to the cone $cA_L$ at the cone point $c$ is a germ of a neighborhood of $\{c\}$ in $cA_L$. The latter is germinally homeomorphic to $cA_L$ itself. Let $\mathrm{Emb}_c(cA_L, cB_{L'})$ denote the quasitopological space of germs of such embeddings. Thus, we have an induced map $\Phi_\varphi$ given by
			$$\Phi_\varphi : U \to \mathrm{Emb}_c(cA_L, cB_{L'}), \; \Phi_{\varphi}(p) = \varphi\vert_{\{p\} \times cA_L}.$$
			Therefore, $	\varphi \mapsto (\varphi\vert_{U \times \{c\}}, \Phi_\varphi)$ furnishes a homeomorphism:
			\begin{gather*}i^*_S\imm^{\mathfrak{c}}_{\bbar{L}}(U, Y) \to \mathrm{Imm}(U, S') \times \maps(U, \mathrm{Emb}_c(cA_L, cB_{L'}))
			\end{gather*}
			This homeomorphism is natural with respect to passing to smaller open subsets $V \subset U$. It commutes with  projections to $\mathrm{Imm}(U, S')$ under restriction maps $\mathrm{res}^L_S$ on the left-hand side. It also commutes with the canonical projection to the first factor on the right-hand side. Therefore, we have an isomorphism of continuous sheaves over $U \subset S$,
			$$i^*_S\imm^{\mathfrak{c}}_{\bbar{L}}(-, Y) \cong \mathrm{Imm}(-, S') \times i^*_U\maps(-, \mathrm{Emb}_c(cA_L, cB_{L'})).$$
			Under this isomorphism, $\mathrm{res}^L_S$ is equivalent to the  projection to the first factor of the product sheaf on the right-hand side. Therefore, as in Proposition \ref{prop-hlsflex},
			$$i^*_U \mathcal{H}^L_S \cong  P_{\psi}\mathrm{Imm}(-, S') \times i^*_U \maps(-, \mathrm{Emb}_c(cA_L, cB_{L'}))$$
			By Gromov's Open Extension Theorem \cite[p. 86]{Gromov_PDR}, $\mathrm{Imm}(-, S')$ is a flexible sheaf on $S$ since $\mathrm{dim}(S) < \mathrm{dim}(S')$ by the positive-codimension hypothesis. {Moreover, by Lemma \ref{lem-fixedF}, $\maps(-, \mathrm{Emb}_c(cA_L, cB_{L'}))$ is flexible, and therefore so is its restriction to an open subset.} Thus, the arguments from Proposition \ref{prop-hlsflex} apply, and we conclude that $\mathcal{H}^L_S$ is a flexible sheaf on $S$.
		
	Finally, for every stratum $S$ of $X$, $\mathrm{Imm}^{\mathfrak{c}}_S(-, Y) = \mathrm{Imm}(-, \mathfrak{c}(S))$ is a flexible sheaf on $S$ once again by the Open Extension Theorem \cite[p. 86]{Gromov_PDR}. Therefore, $\imm^\mathfrak{c}$ is both stratumwise flexible, and infinitesimally flexible across strata. By Theorem \ref{thm-hofibsflexg}, $\imm^\mathfrak{c}$ satisfies the parametric $h-$principle.\end{proof}

{Specializing Theorem \ref{thm-immns} to the case of manifolds with corners, we immediately obtain a relative version of the Smale-Hirsch theorem as a corollary:
\begin{cor}
Let $M, N$ be manifolds with corners. Let $\Sigma_M, \Sigma_N$ denote the natural stratifications, respectively. Let $\mathfrak{c} : \Sigma_M \to \Sigma_N$ be a configuration such that $\dim S < \dim \mathfrak{c}(S)$. Then, the stratified continuous sheaf $\mathrm{Imm}^{\mathfrak{c}}(-, N)$ over $M$ satisfies the parametric $h-$principle.
\end{cor}
}

\newcommand{\etalchar}[1]{$^{#1}$}

\end{document}